\documentclass[english]{article}
\usepackage[T1]{fontenc}
\usepackage[latin9]{inputenc}
\usepackage{xcolor}
\usepackage{pdfcolmk}
\usepackage{amsmath}
\usepackage{amssymb}
\usepackage{graphicx}
\usepackage{esint}
\PassOptionsToPackage{normalem}{ulem}
\usepackage{ulem}

\makeatletter

\providecommand{\tabularnewline}{\\}
\providecolor{lyxadded}{rgb}{0,0,1}
\providecolor{lyxdeleted}{rgb}{1,0,0}
\DeclareRobustCommand{\lyxsout}[1]{\ifx\\#1\else\sout{#1}\fi}

\title{Newtheorem and theoremstyle test}
\author{Michael Downes\\updated by Barbara Beeton}

\usepackage[exercise]{amsthm}

\newtheorem{thm}{Theorem}[section]
\newtheorem{cor}[thm]{Corollary}

\newtheorem{lem}[thm]{Lemma}

\theoremstyle{remark}
\newtheorem*{rmk}{Remark}

\theoremstyle{plain}
\newtheorem{Def}{Definition}

\newtheoremstyle{note}% name
  {3pt}%      Space above
  {3pt}%      Space below
  {}%         Body font
  {}%         Indent amount (empty = no indent, \parindent = para indent)
  {\itshape}% Thm head font
  {:}%        Punctuation after thm head
  {.5em}%     Space after thm head: " " = normal interword space;
  {}%         Thm head spec (can be left empty, meaning `normal')

\theoremstyle{note}

\newtheoremstyle{citing}% name
  {3pt}%      Space above, empty = `usual value'
  {3pt}%      Space below
  {\itshape}% Body font
  {}%         Indent amount (empty = no indent, \parindent = para indent)
  {\bfseries}% Thm head font
  {.}%        Punctuation after thm head
  {.5em}%     Space after thm head: " " = normal interword space;
  {\thmnote{#3}}% Thm head spec

\theoremstyle{citing}
\newtheoremstyle{break}% name
  {9pt}%      Space above, empty = `usual value'
  {9pt}%      Space below
  {\itshape}% Body font
  {}%         Indent amount (empty = no indent, \parindent = para indent)
  {\bfseries}% Thm head font
  {.}%        Punctuation after thm head
  {\newline}% Space after thm head: \newline = linebreak
  {}%         Thm head spec

\theoremstyle{break}

\theoremstyle{exercise}

\swapnumbers
\theoremstyle{plain}

\let\lvert=|\let\rvert=|

\usepackage{babel}

\usepackage{babel}

\usepackage{babel}

\makeatother

\usepackage{babel}
\begin{document}

\title{Constant mean curvature surfaces of Delaunay type along a closed
geodesic}

\author{Shiguang Ma\thanks{School of Mathematical Sciences and LPMC, Nankai University, Tianjin,
P.R.China, 300071, msgdyx8741@nankai.edu.cn.}}
\maketitle
\begin{abstract}
In this paper, we construct Delaunay type constant mean curvature
surfaces along a nondegenerate closed geodesic in a 3-dimensional
Riemannian manifold. 
\end{abstract}
\tableofcontents{}

\section{Introduction}

\subsection{The history and the main result}

Constant mean curvature (CMC) surfaces are a class of important submanifolds.
Let $(M^{m+1},g)$ be a Riemannian manifold of $m+1$ dimension. We
consider the embedded CMC hypersurfaces. In early 1990s, R. Ye proved
in \cite{CMC-Sphere-Point-Ye} the existence of the foliation of constant
mean curvature spheres in Riemannian manifolds around the nondegenerate
critical points of the scalar curvature. In 1996, in \cite{Huisken-Yau}
G. Huisken and S.T. Yau proved the existence of constant mean curvature
foliation in the asymptotically flat end (of a manifold) with positive
mass. Huisken and Yau's result were extended by L. Huang in \cite{Huang-CMC}
and C. Nerz in \cite{Nerz-CMC}. Similar problems were also considered
in asymptotically hyperbolic manifolds. For this topic, see \cite{Rigger-CMC-Hyperbolic,Neves-Tian1,Neves-Tian2,Mazzeo-Pacard-CMC-hyperbolic}.
In \cite{Pacard-Xu-CMC-Degenerate}, Pacard and Xu proved the existence
of constant mean curvature spheres around the degenerate critical
points of scalar curvature which can be regarded as a complement of
Ye's result. In \cite{CMC-tubes-Pacard-Mazzeo}, Mazzeo and Pacard
proved the existence of constant mean curvature tubes along a closed
nondegenerate geodesic. Geodesic is a kind of simple minimal submanifold.
In \cite{CMC-Along-submanifold}, Mahmoudi, Mazzeo and Pacard proved
the existence of constant mean curvature hypersurfaces along minimal
submanifolds. In contrast to the result of R. Ye in \cite{CMC-Sphere-Point-Ye},
the CMC hypersurfaces constructed in \cite{CMC-tubes-Pacard-Mazzeo}
and \cite{CMC-Along-submanifold} constitute a partial foliation,
that is a foliation with gaps. There is a good reason for such gaps,
around which, bifurcation occurs. In particular, the CMC surfaces
which bifurcate from the tubes are of Delaunay type, which are the
main objects those will be constructed in this paper. 

We will give the definition of Delaunay surfaces in the next section.
The Delaunay surfaces were discovered in 1841 by C. Delaunay in \cite{Delaunay-Original}.
It is a one parameter family of complete non-compact surfaces in $\mathbb{R}^{3}.$
The Delaunay surfaces have rotational symmetry, so they  are generated
by an ODE. One may refer to \cite{Eells-Delaunay-surface} for a description
of Delaunay surfaces. They play an analogous role in the theory of
complete CMC surfaces as catenoids do in the theory of complete minimal
surfaces. In \cite{Schoen-Minimal-surface}, it is proved that any
complete minimal immersion $M^{n}\subset\mathbb{R}^{n+1}$ with two
embedded ends and with finite total curvature must be a catenoid or
a pair of planes. And it is proved in \cite{Korevaar-Kusner-Solomon}
that any CMC surface embedded in $\mathbb{R}^{3}$ having two ends
is a Delaunay surface. Another fact is that any end of a complete
minimal surface in $\mathbb{R}^{3}$ of finite total curvature must
be asymptotic to a catenoid or a plane. Parallelling to this fact,
each end of an embedded CMC surface with finite topology converges
exponentially to the end of some Delaunay surface. Delaunay surfaces
can be generalized in a proper way as CMC hypersurfaces which exists
in $\mathbb{S}^{n},\mathbb{R}^{n}$ and $\mathbb{H}^{n}.$ In the
more recent work \cite{Delauany-Cohomogeneity-one}, Bettiol and Piccione
managed to construct Delaunay type CMC surfaces in cohomogeneity one
manifolds. Cohomogeneity one manifolds are those support an isometric
action of a Lie group such that the orbit space $M/G$ is one dimensional.
We see that the metrics of cohomogeneity one manifolds are not generic.
For generic metrics, the existence of Delaunay type CMC surfaces is
unknown, despite some partial results. In this paper we focus on the
existence of Delaunay type CMC surfaces along closed geodesics in
generic metrics. As mentioned just now, this kind of surfaces can
be regarded as the bifurcation branches of the CMC tubes constructed
in \cite{CMC-tubes-Pacard-Mazzeo}. One can refer to \cite{CMC-tubes-Pacard-Mazzeo}
for a description of the moduli space of CMC surfaces along $\Gamma$
which are isotopic to geodesic tubes, which is the motivation of this
paper.

Let's state the main theorem roughly:

\begin{thm}\label{main theorem} Suppose $(M^{3},g)$ is a 3-dimensional
Riemannian manifold and $\Gamma$ is a simple closed embedded geodesic
with nondegenerate Jacobi operator. Then for any $\tau_{0}\in(0,\frac{1}{4})$
we can find $\varepsilon_{0}>0$ which depends on the manifold $M$
and $\tau_{0}$ such that there is a monotone sequence $\varepsilon_{n}\rightarrow0$
with $\varepsilon_{0}>\varepsilon_{1}>\cdots>\varepsilon_{n}>\cdots$
such that along the geodesic there are at least two embedded Delaunay
type CMC surfaces of size $\varepsilon_{n}$ , with mean curvature
$2/\varepsilon_{n}$ and with Delaunay parameter close to $\tau_{0}.$

\end{thm}

The terms ``Delaunay type'', ``size $\varepsilon_{n}$'' and ``Delaunay
parameter close to $\tau_{0}$'' will be made clear in Theorem \ref{rigorous main thm}
which is the rigorous version of this theorem. The assumption that
the Jacobi operator is nondegenerate is a mild restriction which holds
for generic metrics. Our method can also be used in the case that
there is symmetry.

\begin{cor}\label{rotational symmetry}If the metric has rotational
symmetry (at least in a tubular neighborhood of $\Gamma$) with $\Gamma$
being the axis, then in Theorem \ref{main theorem} we can remove
the condition that the geodesic is nondegenerate.

\end{cor}

In \cite{CMC-Along-submanifold}, the authors also did similar things
as in \cite{CMC-tubes-Pacard-Mazzeo}, namely, if $\Sigma\subset M$
is a compact nondegenerate closed minimal submanifold and $\Sigma$
is at least of codimension two, then the authors constructed a partial
foliation of CMC hypersurfaces condensing along $\Sigma.$ The authors
believe that there should be also bifurcation phenomenon in this setting.
However geometric picture of the bifurcating CMC surfaces is still
not clear in this case. 

We use the perturbation method to solve this problem. Yet the argument
is long and involved. This paper is organized as follows: 

In Section 2, we revise some basic facts of Delaunay unduloïd (which
we call Delaunay surface) embedded in the Euclidean space $\mathbb{R}^{3}$
and we make a description on how we arrange an initial surface along
the closed geodesic and how we perturb the initial surface. Then after
long calculations we get the expression (\ref{eq:mean curvature expression})
for the mean curvature of the perturbed initial surface. 

In Section 3 we analyze the Jacobi operator of the perturbed initial
surface. We divide the function space into 3 parts, according to the
invariant subspaces of the Jacobi operator. When restricting the Jacobi
operator to each part we have high mode, 1st mode and 0th mode. 

For high mode, it is easy to prove that the operator is invertible
and the inverse has good bounds. The main result of high mode is contained
in Subsection \ref{subsec:High-mode}. 

For 1st mode, after careful examination, we find the Jacobi operator
converges in certain sense to the Jacobi operator of the geodesic,
which is invertible by the assumption that the geodesic is nondegenerate.
Here we prove an ``average 1'' lemma (Lemma \ref{average 1 lemma}),
which verifies the convergence. The main result of 1st mode is contained
in Theorem \ref{main thm 1st mode}.

The reader may refer to Subsection \ref{subsec:The-sketch-of} for
the sketch of the whole procedure and the main difficulties of the
proof. Due to the difficulties, in $0$th mode, we consider a nonlinear
ODE and prove an ``average 0'' lemma. The reader is suggested to
understand Subsection \ref{subsec:The-sketch-of} before going deep
into the details of $0$th mode. In Subsection \ref{subsec:0th-mode},
we analyze the 0th mode. The main theorem of 0th mode is Theorem \ref{0th mode nonlinear ode solution},
whose proof takes 6 steps. 

In Section \ref{sec:The-existence-of}, we will prove the existence
of the Delaunay type CMC surfaces, using a fixed point argument. One
can also refer to Subsection \ref{subsec:The-sketch-of} for the sketch
of this section. 

\section{Geometry of Delaunay surfaces}

\subsection{The initial surface and the perturbation}

\subsubsection{Delaunay surfaces in Euclidean space $\mathbb{R}^{3}$}

First we give a brief revision of the definition of Delaunay surfaces
in Euclidean space $\mathbb{R}^{3}.$ There are two kinds of Delaunay
surfaces in $\mathbb{R}^{3},$ Delaunay nodoïds and Delaunay unduloïds.
The first type can be immersed into $\mathbb{R}^{3},$ and the second
type can be embedded into $\mathbb{R}^{3}.$ In this paper, by Delaunay
surface, we always mean Delaunay unduloïd. The Delaunay unduloïd $D_{\tau_{0}},0<\tau_{0}<\frac{1}{4}$
can be parameterized by
\[
X_{\tau_{0}}(s,\theta):=(\phi_{\tau_{0}}(s)\cos\theta,\phi_{\tau_{0}}(s)\sin\theta,\psi_{\tau_{0}}(s)),
\]
where $(s,\theta)\in\mathbb{R}\times S^{1}$ and $(\phi,\psi)$ is
the solution to the following system
\begin{equation}
\begin{cases}
\dot{\phi}^{2}+(\phi^{2}+\tau_{0})^{2}=\phi^{2}, & \phi(0)=\frac{1-\sqrt{1-4\tau_{0}}}{2},\\
\dot{\psi}=\phi^{2}+\tau_{0}, & \psi(0)=0,
\end{cases}\label{Delaunay definition}
\end{equation}
where the derivative ``$\cdot$'' is taken with respect to parameter
$s.$ 

From $\dot{\psi}=\phi^{2}+\tau_{0}$, $\psi$ is strictly increasing
in $s$. So we can also regard $\phi$ as a function of $\psi.$ Easy
calculation yields, $\phi(\psi)$ satisfies 
\begin{equation}
\begin{cases}
\phi_{\psi\psi}-\phi^{-1}(1+\phi_{\psi}^{2})+2(1+\phi_{\psi}^{2})^{\frac{3}{2}}=0,\\
\phi(0)=\frac{1-\sqrt{1-4\tau_{0}}}{2},\\
\phi_{\psi}(0)=0,
\end{cases}\label{eq:Delannay Definition directly.}
\end{equation}
which can be regarded as an equivalent definition of Delaunay surfaces.
One may refer to \cite{Mazzeo-Pacard-Delaunay-ends} for more information
of Delaunay surfaces. Direct calculations shows that $X_{\tau_{0}}(s,\theta)=(\phi\cos\theta,\phi\sin\theta,\psi)\subset(\mathbb{R}^{3},g_{edu})$
has mean curvature
\[
-\phi_{\psi\psi}(1+\phi_{\psi}^{2})^{-3/2}+\phi^{-1}(1+\phi_{\psi}^{2})^{-1/2}\equiv2,
\]
 which is independent of $\tau_{0}\in(0,\frac{1}{4}).$

\begin{rmk}

Easy calculation gives that $\dot{\phi},\dot{\psi},\ddot{\phi},\ddot{\psi},\tau_{0}$
can be expressed as functions of $(\phi,\phi_{\psi}).$ 

\end{rmk}

\begin{rmk}The solution to (\ref{Delaunay definition}) or (\ref{eq:Delannay Definition directly.})
is periodic because 
\[
\tau(\phi,\phi_{\psi})=-\phi^{2}+\frac{\phi}{\sqrt{1+\phi_{\psi}^{2}}}\equiv\tau_{0}.
\]
  And the solution satisfies
\[
\frac{1-\sqrt{1-4\tau_{0}}}{2}\leq\phi(\psi)\leq\frac{1+\sqrt{1-4\tau_{0}}}{2}.
\]
So $\phi(\psi)$ attains its minimum at $\psi=0.$ 

\end{rmk}

\subsubsection{\label{section 2.2}Fermi coordinates and Taylor expansion of the
metric near the geodesic}

From now on we discuss the geometry of Delaunay type surface along
a geodesic in a Riemannian 3-manifold. Fix an arc length parametrization
$x_{0}$ of the geodesic $\Gamma,$ $x_{0}\in[0,L_{\Gamma}]$, where
$L_{\Gamma}$ is the length of $\Gamma.$ We denote the normal bundle
of $\Gamma$ by $N\Gamma.$ Choose a parallel orthonormal basis $E_{1},E_{2}$
for $N\Gamma$ (say along $[a,b]$) which determines a coordinate
system
\[
x:(x_{0},x_{1},x_{2})\mapsto\exp_{\Gamma(x_{0})}(x_{1}E_{1}+x_{2}E_{2}):=F(x),
\]
and we denote the corresponding coordinate vector fields by $X_{\alpha}:=F_{*}(\partial_{x_{\alpha}}).$
We adopt the convention that indices $i,j,k,\cdots\in\{1,2\}$ while
$\alpha,\beta,\cdots\in\{0,1,2\}.$ Let $r=\sqrt{x_{1}^{2}+x_{2}^{2}}.$
By Gauss' Lemma $r$ is the geodesic distance from $x$ to $\Gamma$
and the vector $\partial_{r}=\frac{1}{r}(x_{1}X_{1}+x_{2}X_{2})$
is perpendicular to $X_{0}$. We also denote $\partial_{\theta}=-x_{2}X_{1}+x_{1}X_{2}$,
where $(r,\theta)$ is the polar coordinate. 

It is easy to see that the metric coefficients $g_{\alpha\beta}=<X_{\alpha},X_{\beta}>$
equal $\delta_{\alpha\beta}$ along $\Gamma.$ Now we are going to
calculate higher order terms in the Taylor expansions of $g_{\alpha\beta.}$
By the notation $\tilde{O}(r^{m}),$ we mean a function $f$ such
that it and its partial derivatives of any order, with respect to
the vector fields $X_{0}$ and $x_{i}X_{j},$ are bounded by $Cr^{m}$
in some fixed tube $\{p|r(p,\Gamma)\leq r_{0}\}.$ 

First for the covariant derivative, we have

\begin{lem}\label{Connection coefficient estimate}For $\alpha,\beta=0,1,2,$
\[
\nabla_{X_{\alpha}}X_{\beta}=\sum_{\gamma=0}^{2}\tilde{O}(r)X_{\gamma,}
\]
and more precisely for $\alpha=\beta=0,$ we have
\[
\nabla_{X_{0}}X_{0}=-\sum_{i,j=1}^{2}R(X_{j},X_{0},X_{i},X_{0})_{p}x_{i}X_{j}+\sum_{\gamma=0}^{2}\tilde{O}(r^{2})X_{\gamma},
\]
where $R(X_{i},X_{j},X_{k},X_{l})=<\nabla_{X_{i}}\nabla_{X_{j}}X_{k}-\nabla_{X_{j}}\nabla_{X_{i}}X_{k}-\nabla_{[X_{i},X_{j}]}X_{k},X_{l}>.$

\end{lem}

\begin{proof}

We follow Lemma 2.1 in \cite{CMC-tubes-Pacard-Mazzeo}. At any point
$p\in\Gamma$
\[
\nabla_{X_{0}}X_{0}=\nabla_{X_{0}}X_{j}=\nabla_{X_{j}}X_{0}=\nabla_{X_{i}}X_{j}=0
\]
where the last one holds because on $\Gamma$, $\nabla_{X_{i}}X_{i}=0,\nabla_{X_{i}+X_{j}}(X_{i}+X_{j})=0.$
So the first equality follows. For the second one note that 
\begin{align*}
X_{i}<\nabla_{X_{0}}X_{0},X_{j}>_{p} & =<\nabla_{X_{i}}\nabla_{X_{0}}X_{0},X_{j}>_{p}+<\nabla_{X_{0}}X_{0},\nabla_{X_{i}}X_{j}>_{p}\\
 & =<\nabla_{X_{i}}\nabla_{X_{0}}X_{0},X_{j}>_{p}+\tilde{O}(r^{2})\\
 & =<R(X_{i},X_{0})X_{0},X_{j}>_{p}+<\nabla_{X_{0}}\nabla_{X_{i}}X_{0},X_{j}>_{p}+O(r^{2})\\
 & =<R(X_{i},X_{0})X_{0},X_{j}>_{p}+\tilde{O}(r),
\end{align*}
which implies the second one.

\end{proof}

The next lemma gives the expansion of the metric coefficients in Fermi
coordinates. 

\begin{lem}\label{gij expansion}In the same notation as before,
we have
\begin{equation}
\begin{cases}
g_{ij}(q) & =\delta_{ij}+\frac{1}{3}R(X_{k},X_{i},X_{l},X_{j})_{p}x_{k}x_{l}+\tilde{O}(r^{3}),\\
g_{0i}(q) & =\frac{2}{3}R(X_{k},X_{0},X_{l},X_{i})_{p}x_{k}x_{l}+\tilde{O}(r^{3}),\\
g_{00}(q) & =1+R(X_{k},X_{0},X_{l},X_{0})_{p}x_{k}x_{l}+\tilde{O}(r^{3}).
\end{cases}\label{eq:metric expansion}
\end{equation}

\end{lem}

\begin{proof}The reader may refer to the proof of Proposition 2.1
in \cite{CMC-tubes-Pacard-Mazzeo} for the proof . Here we give more
accurate expansion for $g_{0i.}$

\end{proof}

\subsubsection{\label{2.1.3}Initial Delaunay surfaces and the perturbation}

First we arrange an initial Delaunay surface of size $\varepsilon$
along the geodesic. Fix a point $p_{0}\in\Gamma$ with $x_{0}(p_{0})=0,\mod L_{\Gamma}$
and a parameter $\tau_{0}\in(0,1/4)$. We assume $\phi_{\tau_{0}}(\psi)$
is the solution to (\ref{eq:Delannay Definition directly.}). Suppose
the smallest positive period of $\phi(\psi)$ is $\psi_{1}(\tau_{0}).$ 

\begin{Def}We define the set of ``proper size'' for $L_{\Gamma}$
and $\tau_{0}$ by 
\[
{\rm PS}(L_{\Gamma},\tau_{0})=\{\varepsilon_{N}=\frac{L_{\Gamma}}{\psi_{1}(\tau_{0})N};N\in\mathbb{N}^{+}\}.
\]

\end{Def}

In the following, we always assume $\varepsilon\in{\rm PS}(L_{\Gamma},\tau_{0})$.
We can arrange a Delaunay type initial surface of size $\varepsilon$
around the geodesic. Let's make it precise.

The unit circle bundle is locally trivialized by the map
\[
[a,b]\times S^{1}\ni(x_{0},\Upsilon)\mapsto(\Gamma(x_{0}),\sum_{j=1}^{2}\Upsilon_{j}E_{j})\in SN\Gamma.
\]
The image 
\[
F(x_{0},\varepsilon\phi_{\tau_{0}}(\frac{x_{0}}{\varepsilon})\Upsilon)
\]
can be defined locally and extended globally, as it does not depend
on the choice of orthonormal basis $E_{i}.$ We denote the image by
$\mathcal{D}_{\phi_{\tau_{0}},p_{0},\varepsilon}.$ 

Consider the following perturbation of $\mathcal{D}_{\phi_{\tau_{0}},p_{0},\varepsilon}$,
denoted by $\mathcal{D}_{\phi_{\tau_{0}},p_{0},\varepsilon}(w,\eta)$,
where $w$ is a function on unit circle bundle $SN\Gamma$ and $\eta$
is a section of $N\Gamma.$ 

Fix $\varepsilon>0,$ and denote the image
\[
F(x_{0},\varepsilon(\phi(\frac{x_{0}}{\varepsilon})+w(\frac{x_{0}}{\varepsilon},\theta))\Upsilon+\eta(x_{0}));
\]
by $\mathcal{D}_{\phi_{\tau_{0}},p_{0},\varepsilon}(w,\eta).$ It
is obtained by first taking the vertical graph of the function $\varepsilon w$
over the initial Delaunay surface $\mathcal{D}_{\phi_{\tau_{0}},p_{0},\varepsilon}$
and then translating it by $\eta.$ 

First it is clear what do the derivatives with respect to $x_{0},\theta$
mean for a function on $SN\Gamma.$ For a smooth section $\eta$ of
$N\Gamma,$ locally we may write it as $\eta_{1}E_{1}+\eta_{2}E_{2}.$
By $\frac{\partial\eta}{\partial x_{0}}$ we mean $\frac{\partial\eta_{1}}{\partial x_{0}}E_{1}+\frac{\partial\eta_{2}}{\partial x_{0}}E_{2}$
and it is similar for higher order derivatives. We will work in the
following spaces.
\begin{itemize}
\item $C_{x_{0}}^{m}(\Gamma,N\Gamma),C_{x_{0}}^{m}(SN\Gamma),0\leq m\leq\infty$
are the spaces in which all functions have continuous derivatives
up to order $m$, with respect to $x_{0}$ (and also $\theta$ for
the second one). $C_{x_{0}}^{m,\alpha}(\Gamma,N\Gamma),C_{x_{0}}^{m,\alpha}(SN\Gamma),0\leq m<\infty,0<\alpha<1$
are the usual Hölder spaces, where the derivatives are taken with
respect to $x_{0}$ ( and also $\theta$ for the second one). If we
replace $\alpha$ with $1$, they are the usual Lipschitz spaces.
If we replace $x_{0}$ with $y_{0}$, then the derivatives are taken
with respect to $y_{0}$ (and also $\theta$ if needed). 
\item $C_{\varepsilon}^{m,\alpha}(\Gamma,N\Gamma)$, $C_{\varepsilon}^{m,\alpha}(SN\Gamma)$
are the modified Hölder spaces, where the derivatives are taken with
respect to $\varepsilon x_{0}=\psi$ (and also $\theta$ for the second
one). 
\item $C_{x_{0},\varepsilon}^{m,\alpha}(\Gamma,N\Gamma),C_{x_{0},\varepsilon}^{m,\alpha}(SN\Gamma)$
are the modified Hölder spaces with $\|f\|_{C_{x_{0},\varepsilon}^{m,\alpha}}=\|f\|_{C^{0}}+\|\partial_{x_{0}}f\|_{C_{\varepsilon}^{m-1,\alpha}}.$ 
\item $W_{\varepsilon}^{1,2}(SN\Gamma)$ is the modified Sobolev space,
where the derivatives are taken with respect to $\psi$ and $\theta.$
$W_{\varepsilon}^{-1,2}$ is the dual space of $W_{\varepsilon}^{1,2}$. 
\item $\|(f,g)\|_{\varepsilon,\alpha}=\|f\|_{C_{x_{0}}^{1}}+\varepsilon\|g\|_{C_{\varepsilon}^{\alpha}}.$
\end{itemize}
Sometimes we omit the symbols like $\Gamma,SN\Gamma$ in the norms
when it is clear from the context.

For $p\in\Gamma,$ let $S_{p}^{1}$ denote the unit circle fibre of
$SN\Gamma$ over $p.$ Any function $w$ on $SN\Gamma$ decomposes
into a sum of three terms
\begin{equation}
w=w_{0}+w_{1}+\tilde{w}.\label{eq:w-decomposition}
\end{equation}
Here the restriction of any one of $w_{0},w_{1},\tilde{w}$ to each
$S_{p}^{1}$ lies in the span of the eigenfunctions $\xi_{j}$ on
$S^{1}$ with $j=0,j=1,2,$ and $j>2,$ respectively. $w_{0}$ is
a function on $\Gamma.$ 
\begin{align*}
w_{1}(s,\theta) & =w_{1}^{1}(s)\xi_{1}+w_{1}^{2}(s)\xi_{2}\\
 & =w_{1}^{1}(s)\cos\theta+w_{1}^{2}(s)\sin\theta.
\end{align*}
Note that any linear combination of $\xi_{1}$ and $\xi_{2}$ can
be identified with a translation in $\mathbb{R}^{2}$ ($\xi_{1}$
and $\xi_{2}$ correspond to the translations in $x$ and $y$ direction).
Correspondingly, $w_{1}$ is canonically associated to a section $\eta$
of the normal bundle $N\Gamma.$ At last
\[
\tilde{w}(s,\theta)=\sum_{j>2}\tilde{w}_{j}(s)\xi_{j}.
\]

We denote by $\Pi_{0},$ $\Pi_{1}$ and $\tilde{\Pi}$ the projections
onto these three components respectively. We assume $\Pi_{1}w=0$
and the $\Pi_{1}$ part of $w$ is actually represented by $\eta.$ 

Now we can state Theorem \ref{main theorem} rigorously.

\begin{thm}\label{rigorous main thm}

Suppose that $(M^{3},g)$ is a Riemannian manifold of $3$ dimension
and $\Gamma$ is a simple closed embedded geodesic with nondegenerate
Jacobi operator. Then for any $\tau_{0}\in(0,1/4)$ there is $\varepsilon_{0}>0$
such that when $0<\varepsilon<\varepsilon_{0}$ and $\varepsilon\in{\rm PS}(L_{\Gamma},\tau_{0})$
and $i=1,2,$ we have two different Delaunay type surfaces $\mathcal{D}_{\phi_{\tau_{0}},p_{i},\varepsilon}(w_{0,i}+\tilde{w}_{i},\eta_{i})$
along the geodesic which satisfy 
\[
H(\mathcal{D}_{\phi_{\tau_{0}},p_{i},\varepsilon}(w_{0,i}+\tilde{w}_{i},\eta_{i}))=\frac{2}{\varepsilon}.
\]
Here $p_{i}\in\Gamma$ and $w_{0,i},\eta_{i},\tilde{w}_{i}$ belong
to $0$th part, 1st part and high part respectively. Moreover for
uniform constant $C$ 
\begin{eqnarray*}
\varepsilon\|w_{0,i}\|_{C_{\varepsilon}^{2,\alpha}}+\|\eta_{i}\|_{C_{x_{0},\varepsilon}^{2,\alpha}}+\|\tilde{w}_{i}\|_{C_{\varepsilon}^{2,\alpha}} & \leq & C\varepsilon^{2}.
\end{eqnarray*}

\end{thm}

\subsection{The mean curvature of $\mathcal{D}_{\phi_{\tau_{0}},p_{0},\varepsilon}(w,\eta)$}

\subsubsection{The first fundamental form of $\mathcal{D}_{\phi_{\tau_{0}},p_{0},\varepsilon}(w,\eta)$\label{subsec:The-first-fundamental}}

Now we calculate the first fundamental form of $\mathcal{D}_{\phi_{\tau_{0}},p_{0},\varepsilon}(w,\eta)$
with respect to the coordinate $(s,\theta)$ at the point $q=F(\varepsilon\psi(s),\varepsilon(\phi(s)+w(s,\theta))\Upsilon(\theta)+\eta(\varepsilon\psi(s)))$.
Suppose $p=F(\varepsilon\psi(s),0).$ First we have
\begin{equation}
\begin{cases}
\partial_{s} & =\varepsilon(\dot{\psi}X_{0}+(\dot{\phi}+\frac{\partial w}{\partial s})\Upsilon+\dot{\psi}\frac{\partial\eta}{\partial x_{0}}),\\
\partial_{\theta} & =\varepsilon((\phi+w)\Upsilon_{\theta}+\frac{\partial w}{\partial\theta}\Upsilon),
\end{cases}\label{tangential vectors}
\end{equation}
and $x_{k}(q)=\varepsilon(\phi(s)+w(s,\theta))\Upsilon^{k}+\eta^{k},k=1,2.$ 

\begin{Def}\label{L,Q,E,def}

In the following, $L(w,\eta)$ denotes any expression which is a linear
differential operator (of order at most 2), which satisfies
\[
\|L(w,\eta)\|_{C_{\varepsilon}^{\alpha}}\leq C(\|w\|_{C_{\varepsilon}^{2,\alpha}(SN\Gamma)}+\|\eta\|_{C_{x_{0},\varepsilon}^{2,\alpha}(\Gamma,N\Gamma)}),
\]
where $C$ is independent of $\rho.$ Similarly, $Q(w,\eta)$ denotes
any nonlinear differential operator (of order less than or equal to
$2$) in $w$ and $\eta$ which vanishes quadratically in the pair
$(w,\eta)$ and such that 
\begin{align*}
\|Q(w_{1},\eta_{1})-Q(w_{2},\eta_{2})\|_{C_{\varepsilon}^{\alpha}}\leq & C\sup_{i=1,2}(\|w_{i}\|_{C_{\varepsilon}^{2,\alpha}(SN\Gamma)}+\|\eta_{i}\|_{C_{x_{0},\varepsilon}^{2,\alpha}(\Gamma,N\Gamma)})\\
 & \times(\|w_{1}-w_{2}\|_{C_{\varepsilon}^{2,\alpha}(SN\Gamma)}+\|\eta_{1}-\eta_{2}\|_{C_{x_{0},\varepsilon}^{2,\alpha}(\Gamma,N\Gamma)}).
\end{align*}
Here the spaces $C_{\varepsilon}^{\alpha}$ are either equal to $C_{\varepsilon}^{\alpha}(SN\Gamma)$
or $C_{\varepsilon}^{\alpha}(\Gamma,N\Gamma)$ according to the range
of $L$ and $Q$. Finally, by $O(\varepsilon^{k})$ we mean $\varepsilon^{k}E(\phi(\psi),\phi_{\psi}(\psi),\psi)$,
where $E(\phi(\psi),\phi_{\psi}(\psi),\psi)$ denotes any smooth function
with 
\[
|\frac{\partial E(\phi,\phi_{\psi},\psi)}{\partial\phi}|+|\frac{\partial E(\phi,\phi_{\psi},\psi)}{\partial\phi_{\psi}}|\leq C(\tau_{0}),\,\,|\frac{\partial E(\phi,\phi_{\psi},\psi)}{\partial\psi}|\leq C\varepsilon,
\]
where $C$ is a uniform constant. 

\end{Def}

From (\ref{eq:metric expansion}) we know

\begin{lem}\label{metric expansion for Delaunay}
\begin{eqnarray*}
<X_{0},X_{0}>_{q} & = & 1+\varepsilon^{2}\phi^{2}R(\Upsilon,X_{0},\Upsilon,X_{0})_{p}+2\varepsilon\phi R(\Upsilon,X_{0},\eta,X_{0})_{p}\\
 &  & +\varepsilon^{2}L(w,\eta)+Q(w,\eta)+O(\varepsilon^{3}),\\
<X_{i},X_{j}>_{q} & = & \delta_{ij}+\frac{1}{3}\varepsilon^{2}\phi^{2}R(\Upsilon,X_{i},\Upsilon,X_{j})_{p}+\frac{1}{3}\varepsilon\phi(R(\Upsilon,X_{i},\eta,X_{j})_{p}\\
 &  & +R(\eta,X_{i},\Upsilon,X_{j})_{p})+\varepsilon^{2}L(w,\eta)+Q(w,\eta)+O(\varepsilon^{3}),\\
<X_{0},X_{i}>_{q} & = & \frac{2}{3}\varepsilon^{2}\phi^{2}R(\Upsilon,X_{0},\Upsilon,X_{i})_{p}+\frac{2}{3}\varepsilon\phi(R(\Upsilon,X_{0},\eta,X_{i})_{p}\\
 &  & +R(\eta,X_{0},\Upsilon,X_{i})_{p})+\varepsilon^{2}L(w,\eta)+Q(w,\eta)+O(\varepsilon^{3}).
\end{eqnarray*}

\end{lem}

We use these expansions to obtain the expansions of the first fundamental
form,

\begin{lem}\label{first fundamental form}

\begin{eqnarray*}
\varepsilon^{-2}<\partial_{s},\partial_{s}> & = & \phi^{2}+\varepsilon^{2}\phi^{2}\dot{\psi}^{2}R(\Upsilon,X_{0},\Upsilon,X_{0})+2\varepsilon\phi\dot{\psi}^{2}R(\Upsilon,X_{0},\eta,X_{0})\\
 &  & +\frac{4}{3}\varepsilon\phi\dot{\phi}\dot{\psi}R(\Upsilon,X_{0},\eta,\Upsilon)+2\dot{\phi}\frac{\partial w}{\partial s}+2\dot{\phi}\dot{\psi}<\Upsilon,\frac{\partial\eta}{\partial x_{0}}>_{e}\\
 &  & +\varepsilon^{2}L(w,\eta)+Q(w,\eta)+O(\varepsilon^{3}),\\
\varepsilon^{-2}<\partial_{s},\partial_{\theta}> & = & \frac{2}{3}\varepsilon^{2}\phi^{3}\dot{\psi}R(\Upsilon,X_{0},\Upsilon,\Upsilon_{\theta})+\frac{2}{3}\varepsilon\phi^{2}\dot{\psi}(R(\Upsilon,X_{0},\eta,\Upsilon_{\theta})\\
 &  & +R(\eta,X_{0},\Upsilon,\Upsilon_{\theta}))+\frac{1}{3}\varepsilon\phi^{2}\dot{\phi}R(\eta,\Upsilon,\Upsilon,\Upsilon_{\theta})+\dot{\phi}\frac{\partial w}{\partial\theta}\\
 &  & +\phi\dot{\psi}<\frac{\partial\eta}{\partial x_{0}},\Upsilon_{\theta}>_{e}+\varepsilon^{2}L(w,\eta)+Q(w,\eta)+O(\varepsilon^{3}),\\
\varepsilon^{-2}<\partial_{\theta},\partial_{\theta}> & = & \phi^{2}+2\phi w+\frac{1}{3}\varepsilon^{2}\phi^{4}R(\Upsilon,\Upsilon_{\theta},\Upsilon,\Upsilon_{\theta})+\frac{2}{3}\varepsilon\phi^{3}R(\Upsilon,\Upsilon_{\theta},\eta,\Upsilon_{\theta})\\
 &  & +\varepsilon^{2}L(w,\eta)+Q(w,\eta)+O(\varepsilon^{3}).
\end{eqnarray*}

Note that all the curvatures here and after are taken on $p\in\Gamma.$

\end{lem}

\begin{proof}The proof is direct calculation, using Lemma \ref{metric expansion for Delaunay}.
For example for the first one $\varepsilon^{-2}<\partial_{s},\partial_{s}>$,
first we have 
\[
\varepsilon^{-2}<\partial_{s},\partial_{s}>=<\dot{\psi}X_{0}+(\dot{\phi}+\frac{\partial w}{\partial s})\Upsilon+\dot{\psi}\frac{\partial\eta}{\partial x_{0}},\dot{\psi}X_{0}+(\dot{\phi}+\frac{\partial w}{\partial s})\Upsilon+\dot{\psi}\frac{\partial\eta}{\partial x_{0}}>.
\]
 And we get $6$ different terms on the right hand side. For each
one we can use Lemma \ref{metric expansion for Delaunay}. Finally
we can proof the lemma. 

\end{proof}

\subsubsection{Normal vector}

Now we are going to find the expansion of the unit normal vector of
$\mathcal{D}_{\phi_{\tau_{0}},p_{0},\varepsilon}(w,\eta).$ First
we take
\begin{equation}
N_{0}=\frac{1}{\phi}(\dot{\phi}X_{0}-\dot{\psi}\Upsilon),\label{N0}
\end{equation}
which is the unit normal vector of $\mathcal{D}_{\phi_{\tau_{0}},p_{0},\varepsilon}(0,0)$
when curvature vanishes. We may assume the unit normal vector $N$
of $\mathcal{D}_{\phi_{\tau_{0}},p_{0},\varepsilon}(w,\eta)$ has
the form
\begin{equation}
N=\frac{1}{k}(N_{0}+a_{1}\partial_{s}+a_{2}\partial_{\theta}),\label{eq:normal vector}
\end{equation}
where $k$ is the norm of $N_{0}+a_{1}\partial_{s}+a_{2}\partial_{\theta}.$
We will get the expansion for $a_{1},a_{2}$ and $k.$ First

\begin{equation}
\begin{cases}
0 & =<kN,\partial_{s}>=<N_{0},\partial_{s}>+a_{1}<\partial_{s},\partial_{s}>+a_{2}<\partial_{s},\partial_{\theta}>,\\
0 & =<kN,\partial_{\theta}>=<N_{0},\partial_{\theta}>+a_{1}<\partial_{s},\partial_{\theta}>+a_{2}<\partial_{\theta},\partial_{\theta}>.
\end{cases}\label{eq:a1,a2,equation}
\end{equation}

\begin{lem}\label{normal-vector-tangential}
\begin{eqnarray*}
\varepsilon^{-1}<N_{0},\partial_{s}> & = & -\frac{\dot{\psi}}{\phi}\frac{\partial w}{\partial s}-\frac{\dot{\psi}^{2}}{\phi}<\frac{\partial\eta}{\partial x_{0}},\Upsilon>_{e}+\varepsilon^{2}\phi\dot{\phi}\dot{\psi}R(\Upsilon,X_{0},\Upsilon,X_{0})\\
 &  & +2\varepsilon\dot{\phi}\dot{\psi}R(\Upsilon,X_{0},\eta,X_{0})+\frac{2}{3}\varepsilon(\dot{\phi}^{2}-\dot{\psi}^{2})R(\Upsilon,X_{0},\eta,\Upsilon)\\
 &  & +\varepsilon^{2}L(w,\eta)+Q(w,\eta)+O(\varepsilon^{3}),\\
\varepsilon^{-1}<N_{0},\partial_{\theta}> & = & \frac{2}{3}\varepsilon^{2}\phi^{2}\dot{\phi}R(\Upsilon,X_{0},\Upsilon,\Upsilon_{\theta})\\
 &  & +\frac{2}{3}\varepsilon\phi\dot{\phi}(R(\Upsilon,X_{0},\eta,\Upsilon_{\theta})+R(\eta,X_{0},\Upsilon,\Upsilon_{\theta}))\\
 &  & +\varepsilon^{2}L(w,\eta)+Q(w,\eta)+O(\varepsilon^{3}).
\end{eqnarray*}

\end{lem}

\begin{proof}The proof is again direct calculations using (\ref{tangential vectors})
and (\ref{N0}).

\end{proof}

We denote $g_{ss}=<\partial_{s},\partial_{s}>,g_{s\theta}=g_{\theta s}=<\partial_{s},\partial_{\theta}>,g_{\theta\theta}=<\partial_{\theta},\partial_{\theta}>.$
From Lemma \ref{first fundamental form} we have 
\[
\left(\begin{array}{cc}
g_{ss} & g_{s\theta}\\
g_{s\theta} & g_{\theta\theta}
\end{array}\right)=\varepsilon^{2}\phi^{2}\left(\begin{array}{cc}
1+\sigma_{1} & \sigma_{2}\\
\sigma_{2} & 1+\sigma_{3}
\end{array}\right),
\]
where
\begin{equation}
\begin{cases}
\sigma_{1} & =\varepsilon^{2}\dot{\psi}^{2}R(\Upsilon,X_{0},\Upsilon,X_{0})+2\varepsilon\phi^{-1}\dot{\psi}^{2}R(\Upsilon,X_{0},\eta,X_{0})\\
 & +\frac{4}{3}\varepsilon\phi^{-1}\dot{\phi}\dot{\psi}R(\Upsilon,X_{0},\eta,\Upsilon)+2\phi^{-2}\dot{\phi}\frac{\partial w}{\partial s}\\
 & +2\phi^{-2}\dot{\phi}\dot{\psi}<\Upsilon,\frac{\partial\eta}{\partial x_{0}}>_{e}+\varepsilon^{2}L(w,\eta)+Q(w,\eta)+O(\varepsilon^{3}),\\
\sigma_{2} & =\frac{2}{3}\varepsilon^{2}\phi\dot{\psi}R(\Upsilon,X_{0},\Upsilon,\Upsilon_{\theta})+\frac{2}{3}\varepsilon\dot{\psi}(R(\Upsilon,X_{0},\eta,\Upsilon_{\theta})+R(\eta,X_{0},\Upsilon,\Upsilon_{\theta}))\\
 & +\frac{1}{3}\varepsilon\dot{\phi}R(\eta,\Upsilon,\Upsilon,\Upsilon_{\theta})+\phi^{-2}\dot{\phi}\frac{\partial w}{\partial\theta}+\phi^{-1}\dot{\psi}<\frac{\partial\eta}{\partial x_{0}},\Upsilon_{\theta}>_{e}\\
 & +\varepsilon^{2}L(w,\eta)+Q(w,\eta)+O(\varepsilon^{3}),\\
\sigma_{3} & =2\phi^{-1}w+\frac{1}{3}\varepsilon^{2}\phi^{2}R(\Upsilon,\Upsilon_{\theta},\Upsilon,\Upsilon_{\theta})+\frac{2}{3}\varepsilon\phi R(\Upsilon,\Upsilon_{\theta},\eta,\Upsilon_{\theta})\\
 & +\varepsilon^{2}L(w,\eta)+Q(w,\eta)+O(\varepsilon^{3}).
\end{cases}\label{eq:sigma(1,2,3)}
\end{equation}
Notice that $\sigma_{1}\sigma_{3}-\sigma_{2}^{2}=\varepsilon^{2}L(w,\eta)+Q(w,\eta)+O(\varepsilon^{3}).$
We have 
\[
\det\left(\begin{array}{cc}
g_{ss} & g_{s\theta}\\
g_{s\theta} & g_{\theta\theta}
\end{array}\right)\cong\varepsilon^{4}\phi^{4}(1+\sigma_{1}+\sigma_{3})
\]
and 
\[
\det\left(\begin{array}{cc}
g_{ss} & g_{s\theta}\\
g_{s\theta} & g_{\theta\theta}
\end{array}\right)^{-1}\cong\varepsilon^{-4}\phi^{-4}(1-\sigma_{1}-\sigma_{3}).
\]
So the inverse matrix
\begin{eqnarray}
\left(\begin{array}{cc}
g^{ss} & g^{s\theta}\\
g^{s\theta} & g^{\theta\theta}
\end{array}\right) & = & \det\left(\begin{array}{cc}
g_{ss} & g_{s\theta}\\
g_{s\theta} & g_{\theta\theta}
\end{array}\right)^{-1}\left(\begin{array}{cc}
g_{\theta\theta} & -g_{s\theta}\\
-g_{s\theta} & g_{ss}
\end{array}\right)\nonumber \\
 & \cong & \varepsilon^{-2}\phi^{-2}\left(\begin{array}{cc}
1-\sigma_{1} & -\sigma_{2}\\
-\sigma_{2} & 1-\sigma_{3}
\end{array}\right).\label{eq:inverse matrix of metric}
\end{eqnarray}
 From this and Lemma \ref{normal-vector-tangential} we can get

\begin{lem}
\[
\left(\begin{array}{c}
a_{1}\\
a_{2}
\end{array}\right)\cong-\varepsilon^{-2}\phi^{-2}\left(\begin{array}{c}
<N_{0},\partial_{s}>\\
<N_{0},\partial_{\theta}>
\end{array}\right)=(\begin{array}{c}
O(\varepsilon)+\varepsilon^{-1}L(w,\eta)+\varepsilon^{-1}Q(w,\eta)\\
O(\varepsilon)+\varepsilon L(w,\eta)+\varepsilon^{-1}Q(w,\eta)
\end{array}).
\]

\end{lem}

\begin{proof}From (\ref{eq:a1,a2,equation}), we have
\begin{equation}
\left(\begin{array}{c}
a_{1}\\
a_{2}
\end{array}\right)=\varepsilon^{-2}\phi^{-2}\left(\begin{array}{cc}
1-\sigma_{1} & -\sigma_{2}\\
-\sigma_{2} & 1-\sigma_{3}
\end{array}\right)\left(\begin{array}{c}
-<N_{0},\partial_{s}>\\
-<N_{0},\partial_{\theta}>
\end{array}\right).\label{eq:a1-a2-expressions}
\end{equation}
By direct calculation, we have $\sigma_{i}<N_{0},\partial_{s}>$ and
$\sigma_{i}<N_{0},\partial_{\theta}>$ are in fact $\varepsilon^{3}L(w,\eta)+\varepsilon Q(w,\eta)+O(\varepsilon^{4}).$
So we can get the conclusion.

\end{proof}

From (\ref{eq:a1,a2,equation}), 
\[
k^{2}=<N_{0},N_{0}>+a_{1}<N_{0},\partial_{s}>+a_{2}<N_{0},\partial_{\theta}>.
\]
From
\begin{eqnarray*}
<N_{0},N_{0}> & = & <\frac{1}{\phi}(\dot{\phi}X_{0}-\dot{\psi}\Upsilon),\frac{1}{\phi}(\dot{\phi}X_{0}-\dot{\psi}\Upsilon)>\\
 & = & \frac{\dot{\phi}^{2}}{\phi^{2}}<X_{0},X_{0}>+\frac{\dot{\psi}^{2}}{\phi^{2}}<\Upsilon,\Upsilon>-2\frac{\dot{\phi}\dot{\psi}}{\phi^{2}}<X_{0},\Upsilon>\\
 & = & 1+\varepsilon^{2}\dot{\phi}^{2}R(\Upsilon,X_{0},\Upsilon,X_{0})+2\varepsilon\frac{\dot{\phi}^{2}}{\phi}R(\Upsilon,X_{0},\eta,X_{0})\\
 &  & -\frac{4}{3}\varepsilon\frac{\dot{\phi}\dot{\psi}}{\phi}R(\Upsilon,X_{0},\eta,\Upsilon)+\varepsilon^{2}L(w,\eta)+Q(w,\eta)+O(\varepsilon^{3}),
\end{eqnarray*}
\begin{eqnarray*}
a_{1}<N_{0},\partial_{s}> & = & (\varepsilon E(\phi,\phi_{\psi},\psi)+\varepsilon^{-1}L(w,\eta)+\varepsilon^{-1}Q(w,\eta))\varepsilon(-\frac{\dot{\psi}}{\phi}\frac{\partial w}{\partial s}\\
 &  & -\frac{\dot{\psi}^{2}}{\phi}<\frac{\partial\eta}{\partial x^{0}},\Upsilon>_{e}+\varepsilon^{2}\phi\dot{\phi}\dot{\psi}R(\Upsilon,X_{0},\Upsilon,X_{0})\\
 &  & +2\varepsilon\dot{\phi}\dot{\psi}R(\Upsilon,X_{0},\eta,X_{0})+\frac{2}{3}(\dot{\phi}^{2}-\dot{\psi}^{2})R(\Upsilon,X_{0},\eta,\Upsilon)\\
 &  & +\varepsilon^{2}L(w,\eta)+Q(w,\eta)+O(\varepsilon^{3}))\\
 & = & \varepsilon^{2}L(w,\eta)+Q(w,\eta)+O(\varepsilon^{4}),
\end{eqnarray*}
\begin{eqnarray*}
a_{2}<N_{0},\partial_{\theta}> & = & (O(\varepsilon)+\varepsilon L(w,\eta)+\varepsilon^{-1}Q(w,\eta))\varepsilon(\frac{2}{3}\varepsilon^{2}\phi^{2}\dot{\phi}R(\Upsilon,X_{0},\Upsilon,\Upsilon_{\theta})\\
 &  & +\frac{2}{3}\varepsilon\phi\dot{\phi}(R(\Upsilon,X_{0},\eta,\Upsilon_{\theta})+R(\eta,X_{0},\Upsilon,\Upsilon_{\theta}))\\
 &  & +\varepsilon^{2}L(w,\eta)+Q(w,\eta)+O(\varepsilon^{3}))\\
 & = & \varepsilon^{2}L(w,\eta)+Q(w,\eta)+O(\varepsilon^{4}),
\end{eqnarray*}
we have 
\begin{eqnarray*}
k^{2} & = & 1+\varepsilon^{2}\dot{\phi}^{2}R(\Upsilon,X_{0},\Upsilon,X_{0})+2\varepsilon\frac{\dot{\phi}^{2}}{\phi}R(\Upsilon,X_{0},\eta,X_{0})-\frac{4}{3}\varepsilon\frac{\dot{\phi}\dot{\psi}}{\phi}R(\Upsilon,X_{0},\eta,\Upsilon)\\
 &  & +\varepsilon^{2}L(w,\eta)+Q(w,\eta)+O(\varepsilon^{3}).
\end{eqnarray*}
So we have

\begin{lem}

\begin{eqnarray}
k & = & 1+\frac{\varepsilon^{2}}{2}\dot{\phi}^{2}R(\Upsilon,X_{0},\Upsilon,X_{0})+\varepsilon\frac{\dot{\phi}^{2}}{\phi}R(\Upsilon,X_{0},\eta,X_{0})-\frac{2}{3}\varepsilon\frac{\dot{\phi}\dot{\psi}}{\phi}R(\Upsilon,X_{0},\eta,\Upsilon)\nonumber \\
 &  & +\varepsilon^{2}L(w,\eta)+Q(w,\eta)+O(\varepsilon^{3}).\label{eq:k-expression}
\end{eqnarray}

\end{lem}

\subsubsection{The second fundamental form and the mean curvature of $\mathcal{D}_{\phi_{\tau_{0}},p_{0},\varepsilon}(w,\eta).$}

Here we prove

\begin{eqnarray}
 & H(\mathcal{D}_{\phi_{\tau_{0}},p_{0},\varepsilon}(w,\eta))= & \frac{2}{\varepsilon}+\frac{1}{\varepsilon}\mathcal{L}_{SN\Gamma}w+<\mathcal{J}\eta,\Upsilon>\nonumber \\
 &  & +\varepsilon(F_{1}(\phi,\phi_{\psi})\star R_{1}+F_{2}(\phi,\phi_{\psi})\star R_{2})+\varepsilon^{2}E(\phi,\phi_{\psi},\psi)\nonumber \\
 &  & +F_{3}(\phi,\phi_{\psi})\star R_{3}(\eta)+\varepsilon L(w,\eta)+\varepsilon^{-1}Q(w,\eta).\label{eq:mean curvature expression}
\end{eqnarray}
where
\begin{align}
\mathcal{L}_{SN\Gamma}w= & -\frac{\dot{\psi}}{\phi^{3}}(\frac{\partial^{2}w}{\partial s^{2}}+\frac{\partial^{2}w}{\partial\theta^{2}})-2(\phi^{2}-\tau_{0})\frac{\dot{\phi}}{\phi^{4}}\frac{\partial w}{\partial s}-\frac{\dot{\psi}}{\phi^{3}}w,\label{L 0 and high}\\
<\mathcal{J}\eta,\Upsilon>= & -\frac{\dot{\psi}^{3}}{\phi^{3}}<\frac{\partial^{2}\eta}{\partial x_{0}^{2}},\Upsilon>-\frac{1}{\varepsilon}(\frac{\dot{\psi}\ddot{\psi}}{\phi^{3}}+2(\phi^{2}-\tau_{0})\frac{\dot{\phi}\dot{\psi}}{\phi^{4}})<\frac{\partial\eta}{\partial x_{0}},\Upsilon>\nonumber \\
 & +\phi^{-2}(2\dot{\phi}\ddot{\psi}+2\frac{\dot{\phi}^{2}\dot{\psi}}{\phi}+\frac{\dot{\psi}^{3}}{\phi}-2\frac{\dot{\psi}^{2}}{\phi}(\phi^{2}-\tau_{0})-2\phi\dot{\phi}^{2})\nonumber \\
 & \times R(\Upsilon,X_{0},\eta,X_{0}),\label{J eta}
\end{align}
\begin{align*}
F_{1}(\phi,\phi_{\psi})\star R_{1}= & \frac{1}{3}\dot{\psi}R(\Upsilon,\Upsilon_{\theta},\Upsilon,\Upsilon_{\theta})\\
 & +\phi^{-2}(\phi\dot{\phi}\ddot{\psi}+2\dot{\phi}^{2}\dot{\psi}+\dot{\psi}^{3}-(\phi^{2}-\tau_{0})\dot{\psi}^{2}-\phi^{2}\dot{\phi}^{2})\\
 & \times R(\Upsilon,X_{0},\Upsilon,X_{0}),\\
F_{2}(\phi,\phi_{\psi})\star R_{2}= & \frac{2}{3}\dot{\phi}R(\Upsilon_{\theta},X_{0},\Upsilon,\Upsilon_{\theta}),\\
F_{3}(\phi,\phi_{\psi})\star R_{3}(\eta)= & \phi^{-2}(\frac{2}{3}\dot{\phi}\ddot{\phi}+\frac{2}{3}\frac{\dot{\phi}^{3}}{\phi}-\frac{2}{3}\dot{\psi}\ddot{\psi}-\frac{4}{3}(\phi^{2}-\tau_{0})\frac{\dot{\phi}\dot{\psi}}{\phi}-\frac{2}{3}\phi\dot{\phi}+\frac{4}{3}\phi\dot{\phi}\dot{\psi})\\
 & \times R(\Upsilon,X_{0},\eta,\Upsilon)+\frac{2}{3}\phi^{-1}\dot{\phi}R(\Upsilon_{\theta},X_{0},\eta,\Upsilon_{\theta}).
\end{align*}

For our later needs, we assume
\begin{align*}
F_{4}(\phi,\phi_{\psi}) & =\phi^{-2}(\frac{2}{3}\dot{\phi}\ddot{\phi}+\frac{2}{3}\frac{\dot{\phi}^{3}}{\phi}-\frac{2}{3}\dot{\psi}\ddot{\psi}-\frac{4}{3}(\phi^{2}-\tau_{0})\frac{\dot{\phi}\dot{\psi}}{\phi}+\frac{4}{3}\phi\dot{\phi}\dot{\psi}).
\end{align*}

\begin{rmk}
\begin{equation}
\begin{cases}
\Pi_{0}(R(\Upsilon_{\theta},X_{0},\Upsilon,\Upsilon_{\theta})) & =0,\\
\Pi_{0}(R(\Upsilon,X_{0},\eta,\Upsilon)-R(\Upsilon_{\theta},X_{0},\eta,\Upsilon_{\theta})) & =0,\\
\Pi_{1}(R(\Upsilon,\Upsilon_{\theta},\Upsilon,\Upsilon_{\theta})) & =0,\\
\Pi_{1}(R(\Upsilon,X_{0},\Upsilon,X_{0})) & =0,\\
\Pi_{1}(R(\Upsilon,X_{0},\eta,\Upsilon)) & =0,\\
\Pi_{1}(R(\Upsilon_{\theta},X_{0},\eta,\Upsilon_{\theta})) & =0.
\end{cases}\label{Tail term projection}
\end{equation}

And from the second one, 
\[
\Pi_{0}(F_{3}(\phi,\phi_{\psi})\star R_{3}(\eta))=F_{4}(\phi,\phi_{\psi})R(\Upsilon,X_{0},\eta,\Upsilon).
\]

\end{rmk}

The proof of (\ref{eq:mean curvature expression}) is very long calculation,
which can be found in Appendix \ref{(APP)The-calculation-of-mean curvature}.

\section{Jacobi operator}

In this section we study the linear operators which appear in the
expression of $H(\mathcal{D}_{\phi_{\tau_{0}},p_{0},\varepsilon}(w,\eta)).$

\subsection{Basic properties}

Consider (\ref{L 0 and high})(\ref{J eta}). The operator $\mathcal{L}_{SN\Gamma}$
is conjugate to the Jacobi operator which corresponds to the second
variation of the energy functional. 

It is easy to see that
\begin{eqnarray*}
\mathcal{L}_{SN\Gamma}: & C_{\varepsilon}^{2,\alpha}(SN\Gamma) & \mapsto C_{\varepsilon}^{0,\alpha}(SN\Gamma)
\end{eqnarray*}
 is bounded uniformly in $\varepsilon.$ 
\[
\mathcal{J}:C_{x_{0},\varepsilon}^{2,\alpha}(\Gamma,N\Gamma)\mapsto C_{\varepsilon}^{\alpha}(\Gamma,N\Gamma)
\]
and 
\[
\|\mathcal{J}(\eta)\|_{C_{\varepsilon}^{\alpha}}\leq\frac{C}{\varepsilon}\|\eta\|_{C_{x_{0},\varepsilon}^{2,\alpha}}.
\]

We let
\[
\mathcal{L}_{0}=\mathcal{L}_{SN\Gamma}|_{\Pi_{0}(C_{\varepsilon}^{2,\alpha}(SN\Gamma))},\tilde{\mathcal{L}}=\mathcal{L}_{SN\Gamma}|_{\tilde{\Pi}(C_{\varepsilon}^{2,\alpha}(SN\Gamma))}.
\]

We are going to study the mapping properties of $\tilde{\mathcal{L}},\mathcal{J},\mathcal{L}_{0}$
in three different modes, i.e. high mode, 1st mode, 0th mode. 

\subsection{High mode\label{subsec:High-mode}}

In this mode, we are going to prove that 
\[
\tilde{\mathcal{L}}:\tilde{\Pi}C_{\varepsilon}^{2,\alpha}(SN\Gamma)\rightarrow\tilde{\Pi}C_{\varepsilon}^{0,\alpha}(SN\Gamma)
\]
is an isomorphism whose inverse is bounded independent of $\varepsilon$. 

First it is clear that 
\[
\tilde{\mathcal{L}}(\tilde{\Pi}C_{\varepsilon}^{2,\alpha}(SN\Gamma))\subseteq\tilde{\Pi}C_{\varepsilon}^{0,\alpha}(SN\Gamma).
\]

From $\frac{\partial}{\partial s}=\varepsilon\dot{\psi}\frac{\partial}{\partial x_{0}}=\dot{\psi}\frac{\partial}{\partial\psi},$
we have, for $w,v\in\tilde{\Pi}W_{\varepsilon}^{1,2}(SN\Gamma)$ 
\begin{eqnarray*}
\tilde{\mathcal{L}}w & = & -\frac{\dot{\psi}}{\phi^{3}}(\frac{\partial^{2}w}{\partial s^{2}}+\frac{\partial^{2}w}{\partial\theta^{2}})-2(\phi^{2}-\tau_{0})\frac{\dot{\phi}}{\phi^{4}}\frac{\partial w}{\partial s}-\frac{\dot{\psi}}{\phi^{3}}w\\
 & = & -\frac{\phi}{\dot{\psi}^{3}}(\frac{\dot{\psi}^{3}}{\phi^{2}}\frac{\partial}{\partial\psi}(\frac{\dot{\psi}^{3}}{\phi^{2}}\frac{\partial}{\partial\psi}w)+\frac{\dot{\psi}^{4}}{\phi^{4}}w+\frac{\dot{\psi}^{4}}{\phi^{4}}\frac{\partial^{2}w}{\partial\theta^{2}})
\end{eqnarray*}
Consider the bounded bilinear functional $B:\tilde{\Pi}C_{\varepsilon}^{2,\alpha}(SN\Gamma)\times\tilde{\Pi}C_{\varepsilon}^{2,\alpha}(SN\Gamma)\rightarrow\mathbb{R}$
defined by
\begin{eqnarray*}
B(v,w) & = & \int_{SN\Gamma}(v(\phi\tilde{\mathcal{L}})w)d\theta d\psi.
\end{eqnarray*}
We have, for some positive constant $C(\tau_{0})$ which only depends
on $\tau_{0},$ 
\begin{eqnarray}
B(w,w) & = & \int_{SN\Gamma}(\frac{\dot{\psi}^{3}}{\phi^{2}}|\frac{\partial}{\partial\psi}w|^{2}-\frac{\dot{\psi}}{\phi^{2}}w^{2}+\frac{\dot{\psi}}{\phi^{2}}|\frac{\partial w}{\partial\theta}|^{2})d\theta d\psi\nonumber \\
 & \geq & C(\tau_{0})\int_{SN\Gamma}(|\frac{\partial w}{\partial\psi}|^{2}+|\frac{\partial w}{\partial\theta}|^{2}+|w|^{2})d\theta d\psi.\label{eq:coercive}
\end{eqnarray}
The inequality holds because for $w\in\tilde{\Pi}W_{\varepsilon}^{1,2}(SN\Gamma)$,
we have 
\begin{eqnarray*}
\int_{SN\Gamma\cap\{\psi=\psi_{0}\}}|\frac{\partial w}{\partial\theta}|^{2}d\theta & \geq & 4\int_{SN\Gamma\cap\{\psi=\psi_{0}\}}|w|^{2}d\theta,
\end{eqnarray*}
for every $\psi_{0}.$

From (\ref{eq:coercive}) and the Lax-Milgram theorem we know $\phi\tilde{\mathcal{L}}$
is invertible and 
\begin{eqnarray*}
\|w\|_{W_{\varepsilon}^{1,2}} & \leq & C(\tau_{0})\|\phi\tilde{\mathcal{L}}w\|_{W_{\varepsilon}^{-1,2}}.
\end{eqnarray*}
 And from standard regularity theory of elliptic PDE we can get
\begin{eqnarray*}
\|w\|_{C_{\varepsilon}^{2,\alpha}} & \leq & C(\tau_{0})\|\phi\tilde{\mathcal{L}}w\|_{C_{\varepsilon}^{0,\alpha}}\leq C(\tau_{0})\|\tilde{\mathcal{L}}w\|_{C_{\varepsilon}^{0,\alpha}}.
\end{eqnarray*}

\subsection{1st mode}

In this mode, we are going to prove that
\[
\mathcal{J}\eta:C_{x_{0},\varepsilon}^{2,\alpha}(\Gamma,N\Gamma)\rightarrow C_{\varepsilon}^{\alpha}(\Gamma,N\Gamma)
\]
is invertible and the inverse is independent of $\varepsilon$. 

\subsubsection{Some preparations}

First we need to find the relationship between the operator $\mathcal{J}$
and the Jacobi operator $\mathcal{J}_{A}$ of the geodesic. Notice
that 
\begin{eqnarray*}
\frac{\dot{\psi}^{3}}{\phi}\mathcal{J}\eta & = & -\frac{\dot{\psi}^{3}}{\phi^{2}}\frac{\partial}{\partial x_{0}}(\frac{\dot{\psi}^{3}}{\phi^{2}}\frac{\partial\eta}{\partial x_{0}})\\
 &  & -\frac{\dot{\psi}^{3}}{\phi^{3}}(2\dot{\phi}\ddot{\psi}+2\frac{\dot{\phi}^{2}\dot{\psi}}{\phi}+\frac{\dot{\psi}^{3}}{\phi}-2\frac{\dot{\psi}^{2}}{\phi}(\phi^{2}-\tau_{0})-2\phi\dot{\phi}^{2})R(\eta,X_{0})X_{0}.
\end{eqnarray*}
 We define $y_{0}$ by 
\begin{equation}
\begin{cases}
dy_{0} & =\frac{\phi^{2}}{\dot{\psi}^{3}}dx_{0},\\
y_{0}(0) & =0.
\end{cases}\label{eq:definition of y0(x0)}
\end{equation}
 Then $dy_{0}=\frac{\phi^{2}}{\dot{\psi}^{3}}dx_{0}$ and $\frac{\dot{\psi}^{3}}{\phi^{2}}\frac{\partial}{\partial x_{0}}=\frac{\partial}{\partial y_{0}}$.
We know $x_{0}=\varepsilon\psi.$ So the period of $\phi$ and $\psi$
or the derivatives of them (in $x_{0}$ coordinate) have period of
order $\varepsilon.$ So coefficients such as $\frac{\phi^{2}}{\dot{\psi}^{3}}$
and $\frac{\dot{\psi}^{3}}{\phi^{3}}(2\dot{\phi}\ddot{\psi}+2\frac{\dot{\phi}^{2}\dot{\psi}}{\phi}+\frac{\dot{\psi}^{3}}{\phi}-2\frac{\dot{\psi}^{2}}{\phi}(\phi^{2}-\tau_{0})-2\phi\dot{\phi}^{2})$
are highly oscillating in $x_{0}$ coordinate. To understand the mean
value of the coefficients in the right way is the key to understand
the operator $\mathcal{J}$. 

Suppose $\psi\in[a_{1},b_{1}]$ is one period of $\phi$ . Suppose
\[
\frac{\int_{a_{1}}^{b_{1}}\frac{\phi^{2}}{\dot{\psi}^{3}}d\psi}{\int_{a_{1}}^{b_{1}}d\psi}=I_{1}.
\]
 $I_{1}$ is approximately the ratio of the length of $y_{0}$ and
that of $x_{0}.$ In some sense $dy_{0}\cong I_{1}dx_{0}.$ We have

\[
\frac{\dot{\psi}^{3}}{\phi}\mathcal{J}\eta=-\frac{\partial^{2}\eta}{\partial y_{0}^{2}}|_{y_{0}(x_{0})}-\Psi_{1}(\phi,\phi_{\psi})R(\eta,X_{0})X_{0}|_{x_{0}},
\]
where 
\[
\Psi_{1}(\phi,\phi_{\psi})=\frac{\dot{\psi}^{3}}{\phi^{3}}(2\dot{\phi}\ddot{\psi}+2\frac{\dot{\phi}^{2}\dot{\psi}}{\phi}+\frac{\dot{\psi}^{3}}{\phi}-2\frac{\dot{\psi}^{2}}{\phi}(\phi^{2}-\tau_{0})-2\phi\dot{\phi}^{2}).
\]
Now we need the average of $\Psi_{1}(\phi,\phi_{\psi})$ in the coordinate
$y_{0}.$ Note that $dy_{0}=\frac{\phi^{2}}{\dot{\psi}^{3}}dx_{0}=\varepsilon\frac{\phi^{2}}{\dot{\psi}^{3}}d\psi.$
If we assume $y_{0}(a_{1})=y_{1},y_{0}(b_{1})=y_{2},$ then we have
\begin{eqnarray}
 &  & \frac{\int_{y_{1}}^{y_{2}}\Psi_{1}(\phi,\phi_{\psi})dy_{0}}{\int_{y_{1}}^{y_{2}}dy_{0}}\nonumber \\
 & = & \frac{\int_{a_{1}}^{b_{1}}\frac{1}{\phi}(2\dot{\phi}\ddot{\psi}+2\frac{\dot{\phi}^{2}\dot{\psi}}{\phi}+\frac{\dot{\psi}^{3}}{\phi}-2\frac{\dot{\psi}^{2}}{\phi}(\phi^{2}-\tau_{0})-2\phi\dot{\phi}^{2})d\psi}{\int_{a_{1}}^{b_{1}}\frac{\phi^{2}}{\dot{\psi}^{3}}d\psi}\nonumber \\
 & = & I_{2}.\label{Psi1 average I2}
\end{eqnarray}
This indicates that in some sense
\begin{eqnarray*}
 &  & -\frac{\partial^{2}\eta}{\partial y_{0}^{2}}-\Psi_{1}(\phi,\phi_{\psi})R(\eta,X_{0})X_{0}\\
 & \cong & -\frac{\partial^{2}\eta}{\partial y_{0}^{2}}|_{y_{0}}-I_{2}R(\eta,X_{0})X_{0}|_{x_{0}}.
\end{eqnarray*}
But $\frac{\partial}{\partial y_{0}}\cong I_{1}^{-1}\frac{\partial}{\partial x_{0}}.$
So 
\begin{eqnarray*}
-\frac{\partial^{2}\eta}{\partial y_{0}^{2}}|_{y_{0}}-I_{2}R(\eta,X_{0})X_{0}|_{x_{0}} & \cong & -I_{1}^{-2}(\frac{\partial^{2}\eta}{\partial x_{0}^{2}}|_{x_{0}}+I_{1}^{2}I_{2}R(\eta,X_{0})X_{0})|_{x_{0}}.
\end{eqnarray*}
 So if 
\[
I_{1}^{2}I_{2}=1,
\]
 we will have the chance to unearth the Jacobi operator of the geodesic
$\mathcal{J}_{A}.$ Fortunately, it is true. Due to easy calculation,
it is equivalent to the following lemma

\begin{lem}(``average 1'' lemma)\label{average 1 lemma}
\begin{eqnarray*}
 &  & \int_{a_{1}}^{b_{1}}\frac{1}{\phi}(2\dot{\phi}\ddot{\psi}+2\frac{\dot{\phi}^{2}\dot{\psi}}{\phi}+\frac{\dot{\psi}^{3}}{\phi}-2\frac{\dot{\psi}^{2}}{\phi}(\phi^{2}-\tau_{0})-2\phi\dot{\phi}^{2})d\psi\cdot\int_{a_{1}}^{b_{1}}\frac{\phi^{2}}{\dot{\psi}^{3}}d\psi\\
 & = & (b_{1}-a_{1})^{2}.
\end{eqnarray*}

\end{lem}

The proof of this lemma is direct calculations. The proof can be found
in Appendix \ref{(APP)-Average1 and 0}. 

Let $\tilde{y}=I_{1}^{-1}y_{0}$. First we define a map $\mathfrak{F}:\Gamma\rightarrow\Gamma$.
$\mathfrak{F}(p)=q$ if and only if $x_{0}(p)=\tilde{y}(q).$ The
following graph illustrate the definition of $\mathfrak{F}.$ 

\includegraphics[scale=0.4]{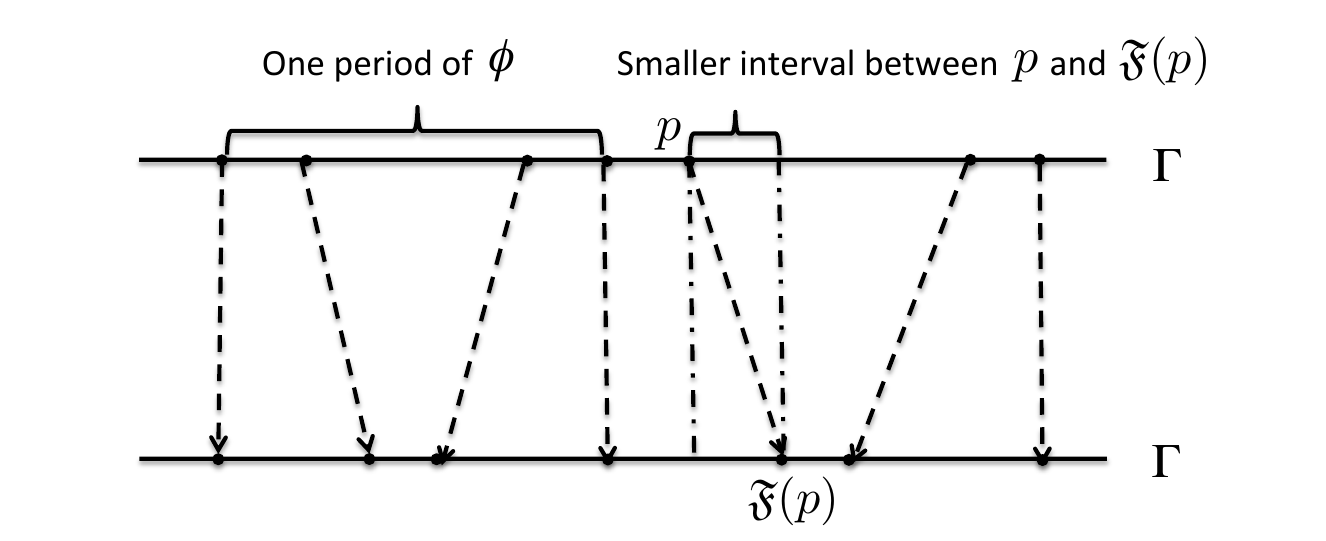}

Let $P_{p}^{\mathfrak{F}(p)}:(N\Gamma)^{*}\otimes N\Gamma|_{p}\rightarrow(N\Gamma)^{*}\otimes N\Gamma|_{\mathfrak{F}(p)}$
be defined by parallel translating any element of $(N\Gamma)^{*}\otimes N\Gamma|_{p}$
to the space $(N\Gamma)^{*}\otimes N\Gamma|_{\mathfrak{F}(p)}$, along
the smaller interval between $p$ and $\mathfrak{F}(p)$. Let $\tilde{\mathcal{J}}_{A}=-\frac{\partial^{2}}{\partial y_{0}^{2}}-I_{2}(P_{p}^{\mathfrak{F}(p)}R(\cdot,X_{0})X_{0})$. 

\begin{lem}\label{eta estimate}

$I_{1}^{2}\tilde{\mathcal{J}}_{A}$ is conjugate to $\mathcal{J}_{A}.$
So $\tilde{\mathcal{J}}_{A}$ is invertible and 
\begin{equation}
\|\eta\|_{C_{y_{0}}^{2}}\leq C\|\tilde{\mathcal{J}}_{A}\eta\|_{C^{0}}.\label{Ja tilda estimate}
\end{equation}

\end{lem}

\begin{proof}

Let $\eta_{1}$ be a smooth section of $N\Gamma$. Define $(\eta_{1}\circ\mathfrak{F})(p)\in N\Gamma(p)$
by parallel translating $\eta_{1}$ from $\mathfrak{F}(p)$ to $p$
along the smaller interval between them. For another section $\eta_{2},$
$\eta_{2}\circ\mathfrak{F}^{-1}$ can be defined similarly. 
\begin{align*}
I_{1}^{2}\tilde{\mathcal{J}}_{A}(\eta_{1})(\mathfrak{F}(p)) & =-\frac{\partial^{2}\eta_{1}(\mathfrak{F}(p))}{\partial\tilde{y}^{2}}-(P_{p}^{\mathfrak{F}(p)}R(\cdot,X_{0})X_{0})(\eta_{1}(\mathfrak{F}(p)))\\
 & =-\frac{\partial^{2}(\eta_{1}\circ\mathfrak{F})(p)}{\partial x_{0}^{2}}-R((\eta_{1}\circ\mathfrak{F})(p),X_{0})X_{0}.
\end{align*}
So 
\[
I_{1}^{2}\tilde{\mathcal{J}}_{A}(\eta_{1})(q)=(-\frac{\partial^{2}(\eta_{1}\circ\mathfrak{F})(p)}{\partial x_{0}^{2}}-R((\eta_{1}\circ\mathfrak{F})(p),X_{0})X_{0})\circ\mathfrak{F}^{-1}(q).
\]
It is obvious that $C^{-1}\|\eta\circ\mathfrak{F}\|_{C_{x_{0}}^{2}}\leq\|\eta\|_{C_{y_{0}}^{2}}\leq C\|\eta\circ\mathfrak{F}\|_{C_{x_{0}}^{2}}$.
So we can prove (\ref{Ja tilda estimate}). 

\end{proof}

\begin{Def}

The Green operator $G(y_{0},z_{0}):N\Gamma|_{z_{0}}\rightarrow N\Gamma|_{y_{0}}$
is a linear map for given $y_{0},z_{0}$ which satisfies
\begin{enumerate}
\item For fixed $z_{0},$ and $\nu\in N\Gamma|_{z_{0}}$, $\tilde{\mathcal{J}}_{A}(G(y_{0},z_{0})\nu)=0,y_{0}\ne z_{0}$,
where $\tilde{\mathcal{J}}_{A}$ acts on $y_{0}.$ 
\item For fixed $z_{0},$ $\nu\in N\Gamma|_{z_{0}}$, $G(y_{0},z_{0})\nu$
can be extended to be continuous at $y_{0}=z_{0}$ and $\frac{\partial G(z_{0}^{+},z_{0})\nu}{\partial y_{0}}-\frac{\partial G(z_{0}^{-},z_{0})\nu}{\partial y_{0}}=\nu.$
\end{enumerate}
\end{Def}

It is obvious that $G_{z_{0}}(y_{0},z_{0})$ and $G_{y_{0}}(y_{0},z_{0})$
make sense in light of the notion of connection of $N\Gamma$. 

\begin{lem}\label{estimate on fundamental solution of 1st mode}

The Green operator exists. If $\tilde{\mathcal{J}}_{A}\eta=f$, then
\begin{equation}
\eta(y_{0})=\int_{\Gamma}G(y_{0},z_{0})f(z_{0})dz_{0}.\label{green operator}
\end{equation}
 Moreover, $G(y_{0},z_{0})$ satisfies that 
\begin{enumerate}
\item For fixed $z_{0},$ let $\nu\in N\Gamma|_{z_{0}}$, $|\nu|=1$. Then
$G(y_{0},z_{0})\nu\in C_{y_{0}}^{\infty}(\Gamma\backslash\{z_{0}\},N\Gamma)$
and can be extended to be in $C_{y_{0}}^{0,1}(\Gamma,N\Gamma)$ and
the $C_{y_{0}}^{0,1}$ norm is bounded independent of $z_{0}$ and
$\nu$. 
\item Let $\nu\in N\Gamma|_{z_{0}}$, $|\nu|=1$. $G_{z_{0}}(y_{0},z_{0})\nu,G_{z_{0}y_{0}}(y_{0},z_{0})\nu,y_{0}\neq z_{0}$
are bounded independent of $y_{0},z_{0}$. 
\end{enumerate}
\end{lem}

\begin{proof}

For fixed $z_{0}$, we choose orthonormal basis $E_{1},E_{2}\in N\Gamma|_{z_{0}}.$
Let $\mathfrak{J}$ be the Jacobi field which is orthonormal to $\gamma$,
with $\mathfrak{J}(z_{0}^{+})=aE_{1}+bE_{2},\mathfrak{J}'(z_{0}^{+})=cE_{1}+dE_{2}.$
We extend $\mathfrak{J}$ along $\gamma$ to $z_{0}^{-}$. And It
is obvious that there is a $4\times4$ matrix $A(z_{0})$ which does
not depends on $\mathfrak{J}$, such that under the basis $E_{1},E_{2}$
\[
\left(\begin{array}{c}
\mathfrak{J}(z_{0}^{-})\\
\mathfrak{J}'(z_{0}^{-})
\end{array}\right)=A(z_{0})\left(\begin{array}{c}
a\\
b\\
c\\
d
\end{array}\right).
\]
 As $\tilde{\mathcal{J}}_{A}$ is nondegenerate, $1$ is not the eigenvalue
of $A(z_{0})$. So we can solve
\[
(A(z_{0})-I)\left(\begin{array}{c}
a_{1}\\
b_{1}\\
c_{1}\\
d_{1}
\end{array}\right)=\left(\begin{array}{c}
0\\
0\\
-1\\
0
\end{array}\right)\,\,{\rm and}\,\,(A(z_{0})-I)\left(\begin{array}{c}
a_{2}\\
b_{2}\\
c_{2}\\
d_{2}
\end{array}\right)=\left(\begin{array}{c}
0\\
0\\
0\\
-1
\end{array}\right).
\]
We let the Jacobi field $\mathfrak{J}_{i}^{z_{0}}(y_{0})$ satisfy
$\mathfrak{J}_{i}^{z_{0}}(z_{0}^{+})=a_{i}E_{1}+b_{i}E_{2},\frac{\partial}{\partial y_{0}}\mathfrak{J}_{i}^{z_{0}}(z_{0}^{+})=c_{i}E_{1}+d_{i}E_{2}.$
So $\mathfrak{J}_{i}^{z_{0}}(z_{0}^{+})=\mathfrak{J}_{i}^{z_{0}}(z_{0}^{-})$
and $\frac{\partial}{\partial y_{0}}\mathfrak{J}_{i}^{z_{0}}(z_{0}^{+})=\frac{\partial}{\partial y_{0}}\mathfrak{J}_{i}^{z_{0}}(z_{0}^{-})+E_{i}.$
$G(y_{0},z_{0})$ is defined by $G(y_{0},z_{0})E_{i}=\mathfrak{J}_{i}^{z_{0}}(y_{0}).$
One can check that $G(y_{0},z_{0})$ is the Green operator and it
is the inverse of $\tilde{\mathcal{J}}_{A}$. 

Now we prove the first item. It is obvious that $\mathfrak{J}_{i}^{z_{0}}(y_{0})$
is a continuous function of $(y_{0},z_{0})\in\Gamma\times\Gamma.$
So it is bounded. From the Jacobi equation, the $\frac{\partial^{2}}{\partial y_{0}^{2}}\mathfrak{J}_{i}^{z_{0}}(y_{0})$
is uniformly bounded when $y_{0}\neq z_{0}.$ As $(A(z_{0})-I)^{-1}$
is uniformly bounded, we know $\frac{\partial}{\partial y_{0}}\mathfrak{J}_{i}^{z_{0}}(z_{0}^{\pm})$
is uniformly bounded. So $\frac{\partial}{\partial y_{0}}\mathfrak{J}_{i}^{z_{0}}(y_{0})$
is uniformly bounded in $\Gamma\times\Gamma\backslash\{(p,p)\}$.
So the first item is clear. 

For the second item, first we know that for fixed $y_{0}$ $G(y_{0},z_{0})$
is smooth in $z_{0}\neq y_{0}$ and $G_{z_{0}}(y_{0},z_{0})E_{i}$
is also a Jacobi field. As $\frac{d}{dz_{0}}(G(z_{0},z_{0})E_{i})=\frac{d}{dz_{0}}\mathfrak{J}_{i}^{z_{0}}(z_{0})$
is uniformly bounded and $\frac{d}{dz_{0}}(G(z_{0},z_{0})E_{i})=G_{y_{0}}(z_{0}^{\pm},z_{0})E_{i}+G_{z_{0}}(z_{0}^{\pm},z_{0})E_{i}$,
we know $G_{z_{0}}(z_{0}^{\pm},z_{0})E_{i}$ is uniformly bounded.
Again as $\frac{d}{dz_{0}}(G_{y_{0}}(z_{0},z_{0})E_{i})=\frac{d}{dz_{0}}\frac{\partial}{\partial y_{0}}\mathfrak{J}_{i}^{z_{0}}(z_{0})$
is bounded, and $\frac{d}{dz_{0}}(G_{y_{0}}(z_{0},z_{0})E_{i})=G_{y_{0}y_{0}}(z_{0}^{\pm},z_{0})E_{i}+G_{z_{0}y_{0}}(z_{0}^{\pm},z_{0})E_{i}$,
we know $G_{z_{0}y_{0}}(z_{0}^{\pm},z_{0})E_{i}$ is uniformly bounded.
So $G_{z_{0}}(y_{0},z_{0})E_{i}$ is uniformly bounded from the comparison
theorem of Jacobi fields. Also $G_{z_{0}y_{0}}(y_{0},z_{0})E_{i}$
is uniformly bounded from the Jacobi equation. 

\end{proof}

\begin{lem}\label{better estimate in 1st mode}

Suppose that $F(\phi,\phi_{\psi})$ is a smooth function of $\phi$
and $\phi_{\psi}$ and it has the cancellation property in $y_{0}$
coordinate, i.e. 
\[
\int_{y_{1}}^{y_{2}}F(\phi,\phi_{\psi})dy_{0}=0
\]
where $[y_{1},y_{2}]$ is one period of $\phi(y_{0}).$ Let $R$ be
a $C_{x_{0}}^{1}$ section of $N\Gamma$. Then 
\[
h(y_{0})=\int_{\Gamma}G(y_{0},z_{0})R(z_{0})F(\phi,\phi_{\psi})(z_{0})dz_{0}
\]
satisfies
\[
\|h(y_{0})\|_{C_{y_{0}}^{1}}\leq C\varepsilon\|R\|_{C_{y_{0}}^{1}}\leq C\varepsilon\|R\|_{C_{x_{0}}^{1}},
\]
where $C$ depends on $\tau_{0}.$ 

\end{lem}

\begin{proof}Suppose $\chi(y_{0})$ is the primitive function of
$F(\phi,\phi_{\psi})$ in $y_{0}$ coordinate. Because $F(\phi,\phi_{\psi})$
has cancellation property, $\chi(y_{0})$ is a global periodic function
on $\Gamma.$ We may add a constant to $\chi(y_{0})$ such that it
also has cancellation property in $y_{0}$ coordinate. It is easy
verified that 
\begin{equation}
\|\chi(y_{0})\|_{C^{0}}\leq C(\tau_{0})\varepsilon.\label{Chi estimate}
\end{equation}
We have
\[
h(y_{0})=-\int_{\Gamma}\chi(z_{0})(G_{z_{0}}R+GR_{z_{0}})dz_{0},
\]
and 
\begin{align*}
h'(y_{0}) & =\int_{\Gamma}G_{y_{0}}R(z_{0})F(\phi,\phi_{\psi})(z_{0})dz_{0}\\
 & =\int_{0}^{L_{\Gamma}}G_{y_{0}}(y_{0},y_{0}+z)R(y_{0}+z)F(\phi,\phi_{\psi})(y_{0}+z)dz\\
 & =G_{y_{0}}(y_{0},y_{0}+z)R(y_{0}+z)\chi(y_{0}+z)|_{0}^{L_{\Gamma}}-\int_{0}^{L_{\Gamma}}\chi(y_{0}+z)(G_{y_{0}z}R+G_{y_{0}}R_{z})dz.
\end{align*}
From $|R_{z_{0}}|\leq C|R_{x_{0}}|\leq C$, (\ref{Chi estimate})
and Lemma \ref{estimate on fundamental solution of 1st mode}, we
can prove this lemma.

\end{proof}

\subsubsection{The main theorem of 1st mode}

Now we can discuss the operator $\mathcal{J}$. First it is clear
that, there is a map $P_{1}:L^{2}(SN\Gamma)\rightarrow L^{2}(N\Gamma)$
such that for any $f\in L^{2}(SN\Gamma)$
\[
\Pi_{1}(f)=<P_{1}(f),\Upsilon>.
\]
We have 

\begin{thm}\label{main thm 1st mode}

There is $\delta>0$ such that when $0<\varepsilon<\delta$, we have
\begin{enumerate}
\item For each $f\in C_{\varepsilon}^{\alpha}(N\Gamma)$, there exists a
unique $\eta\in C_{x_{0},\varepsilon}^{2,\alpha}(N\Gamma)$ such that
\begin{equation}
\mathcal{J}\eta=f\label{J eta=00003Df}
\end{equation}
and for some uniform $C$ which does not depend on $\varepsilon$
\[
\|\eta\|_{C_{x_{0},\varepsilon}^{2,\alpha}}\leq C\|f\|_{C_{\varepsilon}^{\alpha}}.
\]
\item The equation
\[
\mathcal{J}\eta=-\frac{2}{3}\varepsilon\dot{\phi}P_{1}(R(\Upsilon_{\theta},X_{0},\Upsilon,\Upsilon_{\theta}))
\]
has a unique solution $\eta$ such that for some uniform $C$ 
\[
\|\eta\|_{C_{x_{0},\varepsilon}^{2,\alpha}}\leq C\varepsilon^{2}.
\]
\end{enumerate}
\end{thm}

\begin{proof}The proof of the first item. Consider
\begin{align}
\frac{\dot{\psi}^{3}}{\phi}\mathcal{J}\eta(q)= & -\frac{\partial^{2}\eta}{\partial y_{0}^{2}}(q)-\Psi_{1}(\phi,\phi_{\psi})R|_{q}(\eta(q),X_{0})X_{0}\nonumber \\
= & \tilde{\mathcal{J}}_{A}\eta+(I_{2}-\Psi_{1}(\phi,\phi_{\psi})|_{q})(P_{\mathfrak{F}^{-1}(q)}^{q}R)(\eta(q),X_{0})X_{0}\nonumber \\
 & +\Psi_{1}(\phi,\phi_{\psi})|_{q}((P_{\mathfrak{F}^{-1}(q)}^{q}R)-R|_{q})(\eta(q),X_{0})X_{0}\nonumber \\
= & \tilde{\mathcal{J}}_{A}\eta+(I_{2}-\Psi_{1}(\phi,\phi_{\psi})|_{q})(P_{\mathfrak{F}^{-1}(q)}^{q}R)(\eta(q),X_{0})X_{0}+O(\varepsilon)L(\eta)\nonumber \\
= & \frac{\dot{\psi}^{3}}{\phi}f.\label{1st-mode one ode}
\end{align}
First we solve
\[
\tilde{\mathcal{J}}_{A}\eta_{1}=\frac{\dot{\psi}^{3}}{\phi}f.
\]
 From the invertibility of $\tilde{\mathcal{J}}_{A}$ and Lemma \ref{estimate on fundamental solution of 1st mode}
we know there is one unique solution $\eta_{1}$ 
\begin{align*}
\|\eta_{1}\|_{C_{y_{0}}^{2}} & \leq\|\int_{\Gamma}G(y_{0},z_{0})\frac{\dot{\psi}^{3}}{\phi}fdz_{0}\|_{C^{0}}\\
 & \leq\|f\|_{C^{0}}.
\end{align*}
Then for each $i\geq1$ we solve
\begin{align}
 & \tilde{\mathcal{J}}_{A}(\eta_{i+1}-\eta_{i})(q)\nonumber \\
= & (\Psi_{1}(\phi,\phi_{\psi})|_{q}-I_{2})(P_{\mathfrak{F}^{-1}(q)}^{q}R)(\eta_{i}(q)-\eta_{i-1}(q),X_{0})X_{0}+O(\varepsilon)L(\eta_{i}-\eta_{i-1})\label{J tilda iteration}
\end{align}
where $\eta_{0}=0.$ From (\ref{Psi1 average I2}) we know $\Psi_{1}(\phi,\phi_{\psi})|_{x_{0}}-I_{2}$
has $0$ average in one period in $y_{0}$ coordinate. From Lemma
\ref{better estimate in 1st mode}, we know
\begin{align*}
\|\eta_{i+1}-\eta_{i}\|_{C_{y_{0}}^{1}} & \leq C\varepsilon\|\eta_{i}-\eta_{i-1}\|_{C_{y_{0}}^{1}}
\end{align*}
where $C$ depends on $\tau_{0}$ and the norm of the curvature. So
we can choose $\delta$ such that $0<C\delta<\frac{1}{2}$. When $\varepsilon<\delta$,
$0<C\varepsilon<\frac{1}{2}$. So $\eta_{i}$ converges to some $\eta$
in $C_{y_{0}}^{1}.$ From (\ref{J tilda iteration})
\[
\|\eta_{i+1}-\eta_{i}\|_{C_{y_{0}}^{2}}\leq C\|\eta_{i}-\eta_{i-1}\|_{C_{y_{0}}^{1}},
\]
we know $\partial_{y_{0}}^{2}\eta_{i}\rightarrow\partial_{y_{0}}^{2}\eta$
in $C^{0}$ sense. So we get a solution for $\mathcal{J}\eta=f$ and
we have the estimate 
\[
\|\eta\|_{C_{y_{0}}^{1}}\leq\|f\|_{C^{0}}.
\]
So we have 
\[
\|\eta\|_{C_{x_{0}}^{1}}\leq\|f\|_{C^{0}}
\]
and hence from (\ref{J eta=00003Df})
\[
\|\eta\|_{C_{x_{0},\varepsilon}^{2,\alpha}}\leq\|f\|_{C_{\varepsilon}^{\alpha}}.
\]
Now we prove the second item. In (\ref{1st-mode one ode}), we let
$f=-\frac{2}{3}\varepsilon\dot{\phi}P_{1}(R(\Upsilon_{\theta},X_{0},\Upsilon,\Upsilon_{\theta})).$
It is obvious that $\frac{\dot{\psi}^{3}}{\phi}\dot{\phi}$ has 0
average in one period. Apply Lemma \ref{better estimate in 1st mode}
again, we can draw the conclusion.

\end{proof}

\subsection{\label{subsec:The-sketch-of}The sketch of the proof and main difficulties}

Now we are in a position to sketch the whole proof and state the main
difficulities. It would be helpful to analyze the three modes together.
Basically, we would like to use an iteration method to solve the following
system. 
\[
\begin{cases}
\mathcal{L}_{0}w_{0}= & -\varepsilon^{2}F_{1}(\phi,\phi_{\psi})\star\Pi_{0}(R_{1})-\varepsilon\Pi_{0}(E)\\
 & -\varepsilon F_{3}(\phi,\phi_{\psi})\star\Pi_{0}(R_{3}(\eta))\\
 & -\varepsilon\Pi_{0}(\varepsilon L(w_{0}+\tilde{w},\eta)+\varepsilon^{-1}Q(w_{0}+\tilde{w},\eta)),\\
<\mathcal{J}\eta,\varUpsilon>= & -\varepsilon F_{2}(\phi,\phi_{\psi})\star\Pi_{1}(R_{2})-\Pi_{1}(E)\\
 & -\Pi_{1}(\varepsilon L(w_{0}+\tilde{w},\eta)+\varepsilon^{-1}Q(w_{0}+\tilde{w},\eta)),\\
\tilde{\mathcal{L}}\tilde{w}= & -\varepsilon^{2}F_{1}(\phi,\phi_{\psi})\star\tilde{\Pi}(R_{1})-\varepsilon^{2}F_{2}(\phi,\phi_{\psi})\star\tilde{\Pi}(R_{2})\\
 & -\varepsilon\tilde{\Pi}(E)-\varepsilon F_{3}(\phi,\phi_{\psi})\star\tilde{\Pi}(R_{3}(\eta))\\
 & -\varepsilon\tilde{\Pi}(\varepsilon L(w_{0}+\tilde{w},\eta)+\varepsilon^{-1}Q(w_{0}+\tilde{w},\eta)).
\end{cases}
\]

The two fundamental solutions for $\mathcal{L}_{0}w_{0}=0$ are $W_{1}=h_{1}(\tau_{0})\phi_{\psi}$
and $W_{2}=h_{2}(\tau_{0})\psi\phi_{\psi}+v$, where $v$ is a periodic
function. $W_{1}$ is periodic and $W_{2}$ has linear growth. $w_{0}$
can be expressed using $\mathcal{L}_{0}w_{0}$ by
\begin{align}
w_{0}(\psi)= & c_{1}W_{1}(\psi)+c_{2}W_{2}(\psi)\nonumber \\
 & -\int_{0}^{\psi}(W_{2}(\psi)W_{1}(t)-W_{1}(\psi)W_{2}(t))\sigma(t)^{-1}(1+\phi_{\psi}^{2})^{-\frac{3}{2}}\mathcal{L}_{0}w_{0}d\psi,\label{w0 expression}
\end{align}
where $\sigma(\psi)=W_{2}'(\psi)W_{1}(\psi)-W_{1}'(\psi)W_{2}(\psi).$

There are three difficulties in solving the above equation system. 
\begin{enumerate}
\item The operator $\mathcal{L}_{0}$ has kernel $\frac{\dot{\phi}}{\dot{\psi}}=\phi_{\psi}$
which corresponds to the translation of the surface along the geodesic.
So one may only solve it up to the kernel. We have to add one term
$\varepsilon^{3}\omega\phi^{-1}\phi_{\psi}$ to the right hand side
of the first equation. Once we have a solution $w_{0}$, which is
smooth at any $p\in\Gamma$, we get infinite many solutions as we
have freedom to choose $c_{1}.$ We let $c_{1}=0$ to have a unique
solution which satisfies $w_{0}'(0)=0.$ To get such a solution, first
we adjust $\omega$ such that $w_{0}(0)=w_{0}(\frac{L_{\Gamma}}{\varepsilon}).$
Then we adjust $c_{2}$ such that $w_{0}'(0^{+})=w_{0}'(\frac{L_{\Gamma}}{\varepsilon}^{-}).$
\item From (\ref{w0 expression}), we generally expect $w_{0}$ to be of
the size $\frac{1}{\varepsilon^{2}}\mathcal{L}_{0}w_{0}$. Then as
$-\varepsilon^{2}F_{1}(\phi,\phi_{\psi})\star\Pi_{0}(R_{1})$ is of
size $\varepsilon^{2}$ and there is one term $(\partial_{s}w)(\partial_{s}^{2}w)$
in $\Pi_{0}(Q(w,\eta))$, it is impossible to solve the first equation
using an iteration method, even given $\tilde{w}=0$ and $\eta=0$.
Actually there is a cancellation property for $-\varepsilon^{2}F_{1}(\phi,\phi_{\psi})\star\Pi_{0}(R_{1})$
so that it looks as if it has order $O(\varepsilon^{3}).$ Even though,
the iteration still breaks down. 
\item The term $F_{3}(\phi,\phi_{\psi})\star R_{3}(\eta)$ also causes problem
because there is a term of the type $\frac{1}{\varepsilon}w\cdot\partial_{x_{0}}\eta$
in $\Pi_{1}(Q(w,\eta))$. Once $w_{0}$ has a variation of $O(\varepsilon).$
From Theorem \ref{main thm 1st mode}, $\eta$ will have a variation
of size $O(\varepsilon^{2}).$ Now if we go back to $0$th mode, $w_{0}$
will again have a variation of size $O(\varepsilon)$. So even given
$\tilde{w}=0$, it is impossible to solve the first two equations
together. 
\end{enumerate}
The second difficulty appears because of the shortcoming of 
\begin{align*}
\mathcal{L}_{0}w_{0}= & -\varepsilon^{2}F_{1}(\phi,\phi_{\psi})\star\Pi_{0}(R_{1})-\varepsilon\Pi_{0}(E)-\varepsilon F_{3}(\phi,\phi_{\psi})\star\Pi_{0}(R_{3}(\eta))\\
 & -\varepsilon\Pi_{0}(\varepsilon L(w_{0}+\tilde{w},\eta)+\varepsilon^{-1}Q(w_{0}+\tilde{w},\eta)),
\end{align*}
that is, $Q(w_{0}+\tilde{w},\eta)$ involves special structures, which
are hard to analyze and use. Now we consider 0th mode in a nonlinear
way. Notice that if $\phi_{\tau_{0}}$ satisfies (\ref{eq:Delannay Definition directly.}),
\begin{align*}
\Pi_{0}(H(\mathcal{D}_{\phi_{\tau_{0}},p_{0},\varepsilon}(w,\eta))-\frac{2}{\varepsilon})= & \varepsilon F_{1}(\phi_{\tau_{0}},(\phi_{\tau_{0}})_{\psi})\star\Pi_{0}(R_{1})+\Pi_{0}(E)\\
 & +F_{3}(\phi_{\tau_{0}},(\phi_{\tau_{0}})_{\psi})\star\Pi_{0}(R_{3}(\eta))\\
 & +\Pi_{0}(\varepsilon L(w_{0}+\tilde{w},\eta)+\varepsilon^{-1}Q(w_{0}+\tilde{w},\eta)).
\end{align*}
The key idea is we substitute the solution $\phi_{\xi,\mu}(\psi)$
of the following nonlinear ODE for $\phi_{\tau_{0}}$. We let $\psi(p_{0})=0,\mod\,\frac{L_{\Gamma}}{\varepsilon}.$
Consider
\begin{equation}
\begin{cases}
\phi_{\psi\psi} & -\phi^{-1}(1+\phi_{\psi}^{2})+(2+\rho)(1+\phi_{\psi}^{2})^{\frac{3}{2}}=0,\\
\phi_{\psi}(0) & =\phi_{\psi}(\frac{L_{\Gamma}}{\varepsilon})=0,\\
\phi(0) & =\phi(\frac{L_{\Gamma}}{\varepsilon})=\frac{1-\sqrt{1-4\tau(0)}}{2},\\
\rho & =-\varepsilon^{2}F_{1}(\phi,\phi_{\psi})\star\Pi_{0}(R_{1})+\varepsilon F_{4}(\phi,\phi_{\psi})\xi(x_{0})\\
 & +\varepsilon^{3}\mu(\psi)+\varepsilon^{3}\omega\phi^{-1}\phi_{\psi},
\end{cases}\label{0th mode ODE}
\end{equation}
where $\tau(0)=\tau(\phi(0),0)$ and
\begin{equation}
\|\xi\|_{C_{x_{0}}^{1}}\leq C_{1}\varepsilon^{2},\|\mu\|_{C_{\varepsilon}^{\alpha}}\leq C_{2},|\omega|\leq C_{3},|\phi(0)-\frac{1-\sqrt{1-4\tau_{0}}}{2}|\leq C(\tau_{0})\varepsilon,\label{C1C2C3C(tau0)}
\end{equation}
where 
\[
C(\tau_{0})=\frac{1}{10}\min\{\frac{\sqrt{1-4\tau_{0}}}{2},\frac{1-\sqrt{1-4\tau_{0}}}{2},\frac{1}{4}-\tau_{0},\tau_{0}\}.
\]
 Among $\xi,\mu,\omega,\tau(0)$, only $\xi,\mu$ have freedom because
this equation is overdetermined. Once we have prescribed $\xi,\mu$,
we need to adjust $\omega$ and $\tau(0)$ to get a global smooth
solution. Theorem \ref{0th mode nonlinear ode solution} contains
the main results of $0$th mode, one may refer to which for the existence,
uniqueness and estimates of the solution $\phi_{\xi,\mu}$ to (\ref{0th mode ODE}). 

In Subsection \ref{subsec:A-fixed-point} we calculate $H(\mathcal{D}_{\phi_{\xi,\mu},p_{0},\varepsilon}(\tilde{w},\eta))$
instead, where $\Pi_{0}(\tilde{w})=0.$ The expression for $H(\mathcal{D}_{\phi_{\xi,\mu},p_{0},\varepsilon}(\tilde{w},\eta))$
is (\ref{Perturbed mean curvature}). 

After we prescribed $\xi$ and $\mu$ and got $\phi_{\xi,\mu}$, the
coefficients of $1$st mode and high mode change accordingly, while
from Theorem \ref{Perturbed high mode and 1st mode}, the first mode
and high mode are still solvable and the solutions enjoy similar estimates
as in Subsection \ref{subsec:High-mode} and Theorem \ref{main thm 1st mode}.
We denote the corresponding solutions as $\tilde{w}_{\xi,\mu}$ and
$\eta_{\xi,\mu}.$ 

We have 
\begin{align*}
\Pi_{0}(H(\mathcal{D}_{\phi_{\xi,\mu},p_{0},\varepsilon}(\tilde{w}_{\xi,\mu},\eta_{\xi,\mu}))-\frac{2}{\varepsilon})= & \Pi_{0}((\varepsilon^{2}+\varepsilon\rho)E_{\xi,\mu}+(\varepsilon+\varepsilon^{-1}\rho)L_{\xi,\mu}(\tilde{w}_{\xi,\mu},\eta_{\xi,\mu})\\
 & +\varepsilon^{-1}(1+\rho)Q_{\xi,\mu}(\tilde{w}_{\xi,\mu},\eta_{\xi,\mu}))+\varepsilon^{2}\mu\\
 & +\Pi_{0}(F_{3}(\phi_{\xi,\mu},\zeta_{\xi,\mu})\star R_{3}(\eta_{\xi,\mu}))\\
 & +F_{4}(\phi_{\xi,\mu},\zeta_{\xi,\mu})\star\xi+\varepsilon^{2}\omega\phi_{\xi,\mu}^{-1}\frac{\partial\phi_{\xi,\mu}}{\partial\psi},
\end{align*}
where $\zeta_{\xi,\mu}=\partial_{\psi}\phi_{\xi,\mu}.$

Note that $\Pi_{0}(F_{3}(\phi_{\xi,\mu},\zeta_{\xi,\mu})\star R_{3}(\eta_{\xi,\mu}))=F_{4}(\phi_{\xi,\mu},\zeta_{\xi,\mu})\star\Pi_{0}(R_{3}(\eta_{\xi,\mu}))$,
because 
\[
\Pi_{0}(R(\Upsilon,X_{0},\eta_{\xi,\mu},\Upsilon)-R(\Upsilon_{\theta},X_{0},\eta_{\xi,\mu},\Upsilon_{\theta}))=0.
\]
 We want to choose $\xi$ and $\mu$ such that 
\[
\begin{cases}
(\varepsilon^{2}+\varepsilon\rho)E_{\xi,\mu}+(\varepsilon+\varepsilon^{-1}\rho)L_{\xi,\mu}(\tilde{w}_{\xi,\mu},\eta_{\xi,\mu}) & +\varepsilon^{-1}(1+\rho)Q_{\xi,\mu}(\tilde{w}_{\xi,\mu},\eta_{\xi,\mu})+\varepsilon^{2}\mu=0,\\
\Pi_{0}(R_{3}(\eta_{\xi,\mu}))+\xi=0.
\end{cases}
\]
From Definition \ref{new L Q E}, Lemma \ref{Delta mu delta xi influence on w and eta},
\ref{fixed point argument}, we know a fixed point argument works
and we will find the solution $(\hat{\xi},\hat{\mu})$ such that
\[
H(\mathcal{D}_{\phi_{\hat{\xi},\hat{\mu}},p_{0},\varepsilon}(\tilde{w}_{\hat{\xi},\hat{\mu}},\eta_{\hat{\xi},\hat{\mu}}))=\frac{2}{\varepsilon}+\varepsilon^{2}\omega_{\hat{\xi},\hat{\mu}}\phi_{\hat{\xi},\hat{\mu}}^{-1}\frac{\partial\phi_{\hat{\xi},\hat{\mu}}}{\partial\psi}.
\]
 In Subsection \ref{subsec:The-energy-of}, we remove the term $\varepsilon^{2}\omega_{\hat{\xi},\hat{\mu}}\phi_{\hat{\xi},\hat{\mu}}^{-1}\frac{\partial\phi_{\hat{\xi},\hat{\mu}}}{\partial\psi}$,
by considering the energy of the surface. Letting $p_{0}$ vary on
the geodesic, we find the critical points of the energy functional,
which correspond to the case $\omega_{\hat{\xi},\hat{\mu}}=0$. So
we can overcome the first difficulty. 

The third difficulty exists even when we consider nonlinear ODE in
$0$th mode. Generally $\eta=O(\varepsilon^{2})$ is too big to solve
the whole system. 

The key to control the size of $\phi_{\xi,\mu}$ is the Delaunay parameter
function $\tau=-\phi^{2}+\frac{\phi}{\sqrt{1+\phi_{\psi}^{2}}}.$
For the solution $\phi_{\xi,\mu}$ of (\ref{0th mode ODE}), $F_{4}\xi$
influence the size of $\phi_{\xi,\mu}$ through the following equation
\[
\frac{d\tau}{d\psi}=\phi\phi_{\psi}\rho.
\]
Here $\xi$ represents $R_{4}(\eta)$ whose norm is $C_{x_{0}}^{1}.$
For $F_{4}$, we have

\begin{lem}(``average 0'' lemma)\label{average 0 lemma} Suppose
$\phi_{\tau_{0}}$ is the solution to (\ref{eq:Delannay Definition directly.})
and $\psi\in[a_{\tau_{0}},b_{\tau_{0}}]$ is one period of $\phi_{\tau_{0}}(\psi)$.
Then we have 
\[
\int_{a_{\tau_{0}}}^{b_{\tau_{0}}}\phi_{\tau_{0}}(\frac{\partial\phi_{\tau_{0}}}{\partial\psi})F_{4}(\phi_{\tau_{0}},\frac{\partial\phi_{\tau_{0}}}{\partial\psi})d\psi=0.
\]

\end{lem}

We prove this lemma in Appendix \ref{(APP)-Average1 and 0}.

Using this lemma, we can overcome the third difficulty. Indeed, if
in the $C_{x_{0}}^{1}$ norm, $\eta$ has a variation of size $O(\varepsilon^{2})$,
because of Lemma \ref{average 0 lemma}, its influence on 0th mode
is only as if it were $O(\varepsilon^{3})$. The function $\phi_{\tau_{0}}(\frac{\partial\phi_{\tau_{0}}}{\partial\psi})F_{4}(\phi_{\tau_{0}},\frac{\partial\phi_{\tau_{0}}}{\partial\psi})$
is in no sense an ``odd'' function. So it is nontrivial that such
a cancellation result holds.

\subsection{0th mode\label{subsec:0th-mode}}

\subsubsection{The main theorem of 0th mode}

\begin{thm}\label{0th mode nonlinear ode solution} For fixed $\tau_{0}\in(0,\frac{1}{4})$,
$C_{1},C_{2}>0$, we can choose $C_{3},C_{4}>0$ and $\delta_{0}>0$,
such that when $\varepsilon\leq\delta_{0}$ and $\varepsilon\in{\rm PS}(L_{\Gamma},\tau_{0})$,
for every $\|\xi\|_{C_{x_{0}}^{1}}\leq C_{1}\varepsilon^{2},\|\mu\|_{C_{\varepsilon}^{\alpha}}\leq C_{2},$
we can find unique
\[
|\omega_{\xi,\mu}|\leq C_{3},|\phi(0)_{\xi,\mu}-\frac{1-\sqrt{1-4\tau_{0}}}{2}|\leq C_{4}\varepsilon^{2}
\]
 and $\phi_{\xi,\mu,\omega_{\xi,\mu},\phi(0)_{\xi,\mu}}(\psi)$ which
solves (\ref{0th mode ODE}) with $\omega=\omega_{\xi,\mu},\phi(0)=\phi(0)_{\xi,\mu}.$ 

For $\phi_{\xi,\mu,\omega_{\xi,\mu},\phi(0)_{\xi,\mu}}(\psi)$ we
have a $C^{1}$ map 
\[
\tilde{\Phi}=\tilde{\Phi}_{\xi,\mu,\tau_{0}}:\Gamma\rightarrow\Gamma
\]
 such that 
\begin{align*}
 & |\phi_{\xi,\mu,\omega_{\xi,\mu},\phi(0)_{\xi,\mu}}(\psi)-\phi_{\tau_{0}}(\tilde{\Phi}(\psi))|+|\frac{\partial\phi_{\xi,\mu,\omega_{\xi,\mu},\phi(0)_{\xi,\mu}}}{\partial\psi}(\psi)-\frac{\partial\phi_{\tau_{0}}}{\partial\tilde{\Phi}(\psi)}(\tilde{\Phi}(\psi))|\\
\leq & C(\tau_{0},C_{1},C_{2})\varepsilon^{2}
\end{align*}
 and
\[
\begin{cases}
|\tilde{\Phi}(\psi)-\psi| & \leq C(\tau_{0},C_{1},C_{2})\varepsilon,\\
|\tilde{\Phi}'(\psi)-1| & \leq C(\tau_{0},C_{1},C_{2})\varepsilon^{2}.
\end{cases}
\]
Moreover suppose $\psi_{i}$ is the $i$th local minimum point of
$\phi_{\xi,\mu,\omega_{\xi,\mu},\phi(0)_{\xi,\mu}}$, then $\tilde{\Phi}(\psi_{i})$
is the $i$th local minimum point of $\phi_{\tau_{0}}$ . In particular
$\tilde{\Phi}(0)=0,\tilde{\Phi}(L_{\Gamma})=L_{\Gamma}.$

Moreover, 
\begin{align*}
|\omega_{\xi_{2},\mu_{2}}-\omega_{\xi_{1},\mu_{1}}| & \leq C(\tau_{0},C_{1},C_{2})(\|\mu_{2}-\mu_{1}\|_{C_{\varepsilon}^{\alpha}}+\frac{1}{\varepsilon}\|\xi_{2}-\xi_{1}\|_{C_{x_{0}}^{1}}),\\
|\phi(0)_{\xi_{2},\mu_{2}}-\phi(0)_{\xi_{1},\mu_{1}}| & \leq C(\tau_{0},C_{1},C_{2})(\varepsilon^{2}\|\mu_{2}-\mu_{1}\|_{C_{\varepsilon}^{\alpha}}+\varepsilon\|\xi_{2}-\xi_{1}\|_{C_{x_{0}}^{1}}),
\end{align*}
and 
\begin{align}
 & \|\phi_{\xi_{2},\mu_{2},\omega_{\xi_{2},\mu_{2}},\tau(0)_{\xi_{2},\mu_{2}}}(\psi)-\phi_{\xi_{1},\mu_{1},\omega_{\xi_{1},\mu_{1}},\tau(0)_{\xi_{1},\mu_{1}}}(\psi)\|_{C_{\varepsilon}^{2,\alpha}}\nonumber \\
\leq & C(\tau_{0},C_{1},C_{2})(\varepsilon\|\mu_{2}-\mu_{1}\|_{C_{\varepsilon}^{\alpha}}+\|\xi_{2}-\xi_{1}\|_{C_{x_{0}}^{1}}).\label{main estimates for 0th mode-1}
\end{align}

\end{thm}

The proof of Theorem \ref{0th mode nonlinear ode solution} is the
most technical part of this paper.  We prove it in six steps. 

From Step 1 to Step 4, we will study the following initial value problem
instead of (\ref{0th mode ODE}). 

\begin{equation}
\begin{cases}
\phi_{\psi\psi} & -\phi^{-1}(1+\phi_{\psi}^{2})+(2+\rho)(1+\phi_{\psi}^{2})^{\frac{3}{2}}=0,\\
\phi_{\psi}(0) & =0,\\
\phi(0) & =\frac{1-\sqrt{1-4\tau(0)}}{2},
\end{cases}\label{0th mode ODE-1}
\end{equation}
where $\rho$ is defined in (\ref{0th mode ODE}). In Step 1, we prove
the local existence of (\ref{0th mode ODE-1}); In Step 2, we prove
the existence of the solution of (\ref{0th mode ODE-1}) on the whole
interval $[0,\frac{L_{\Gamma}}{\varepsilon}]$. In Step 3, we do some
basic estimates for the solution. In Step 4, we analyze the linearized
equations of (\ref{0th mode ODE-1}), which will be used in the next
two steps. In Step 5, we adjust $\omega$ and $\tau(0)$ to match
the boundary data on both sides of $\psi=0$ (or $\psi=\frac{L_{\Gamma}}{\varepsilon}$)
and get the global solution to (\ref{0th mode ODE}). In Step 6, we
derive the estimates in Theorem \ref{0th mode nonlinear ode solution}.

\paragraph{Step 1. Local existence and uniqueness of ODE (\ref{0th mode ODE})}

Fix $C_{1},C_{2},C_{3}$ and $C(\tau_{0})$ in (\ref{C1C2C3C(tau0)}).
Choose $A_{1},A_{2}$,$B_{1},B_{2}$, $K_{1},K_{2}$ that only depend
on $\tau_{0}$ such that 
\[
0<A_{1}<A_{2}<\frac{1-\sqrt{1-4\tau(0)}}{2}<\frac{1+\sqrt{1-4\tau(0)}}{2}<B_{2}<B_{1}<1,
\]
and 
\[
K_{1}-1=K_{2}=1+\sup\{|\phi_{\psi}|;-\phi^{2}+\frac{\phi}{\sqrt{1+\phi_{\psi}}}=\tau_{0}-C(\tau_{0})\}.
\]
 Define $C_{A_{i},B_{i},K_{i}}^{1}([0,T])=\{\phi(\psi)\in C^{1}([0,T]);A_{i}\leq\phi(\psi)\leq B_{i},|\phi_{\psi}|\leq K_{i}\},i=1,2.$
We have

\begin{lem}\label{local existence of 0th mode} If $\phi(\psi)\in C_{A_{2},B_{2},K_{2}}^{1}([0,T])$
solves (\ref{0th mode ODE-1}) for $T\geq0$. Then for some $\delta>0$
which only depends on $\tau_{0}$, this solution can be uniquely extended
to $\phi(\psi)\in C_{A_{1},B_{1},K_{1}}^{1}([0,T+\delta]).$

\end{lem}

\begin{proof} Denote $\frac{\partial\phi}{\partial\psi}$ by $\zeta.$
Then $(\phi,\zeta)$ satisfies the following system

\[
\begin{cases}
\phi_{\psi}=\zeta, & \phi(0)=\frac{1-\sqrt{1-4\tau(0)}}{2},\\
\zeta_{\psi}=\phi^{-1}(1+\zeta^{2})-(2+\rho)(1+\zeta^{2})^{\frac{3}{2}}, & \zeta(0)=0,
\end{cases}
\]
The right hand side is uniformly bounded in $(\phi,\zeta,\psi)$ and
is a Lipschitz function with respect to $(\phi,\zeta)$ in the domain
\[
\begin{cases}
|\zeta|<K_{1}\\
A_{1}<\phi<B_{1}\\
0<\psi<\frac{L_{\Gamma}}{\varepsilon}.
\end{cases}
\]
The $C^{0}$ norm and Lipschitz constant of the right hand side only
depend on $\tau_{0}.$ So from standard theory of ODE we get the conclusion.

\end{proof}

\paragraph{Step 2. Existence on $[0,\frac{L_{\Gamma}}{\varepsilon}]$}

To get the existence of the solution on $[0,\frac{L_{\Gamma}}{\varepsilon}]$,
we need to do apriori estimates for the solution. The key to do this
is the first integral 
\begin{eqnarray*}
\tau(\phi,\phi_{\psi}) & = & -\phi^{2}+\frac{\phi}{\sqrt{1+\phi_{\psi}^{2}}}.
\end{eqnarray*}
The following graph is the phase space of (\ref{0th mode ODE}). We
choose $\tau_{0}\in(0,\frac{1}{4})$ and $\delta_{1}$ small. We denote
the closure of the domain between the outside and inside circle by
$A(\tau_{0},\delta_{1}).$ And we denote $\{\phi(\psi)\in C^{1}([0,T_{0}])|(\phi,\phi_{\psi})\in A(\tau_{0},\delta_{1})\}$
as $C_{A(\tau_{0},\delta_{1})}^{1}([0,T_{0}])$. 

\includegraphics[scale=0.4]{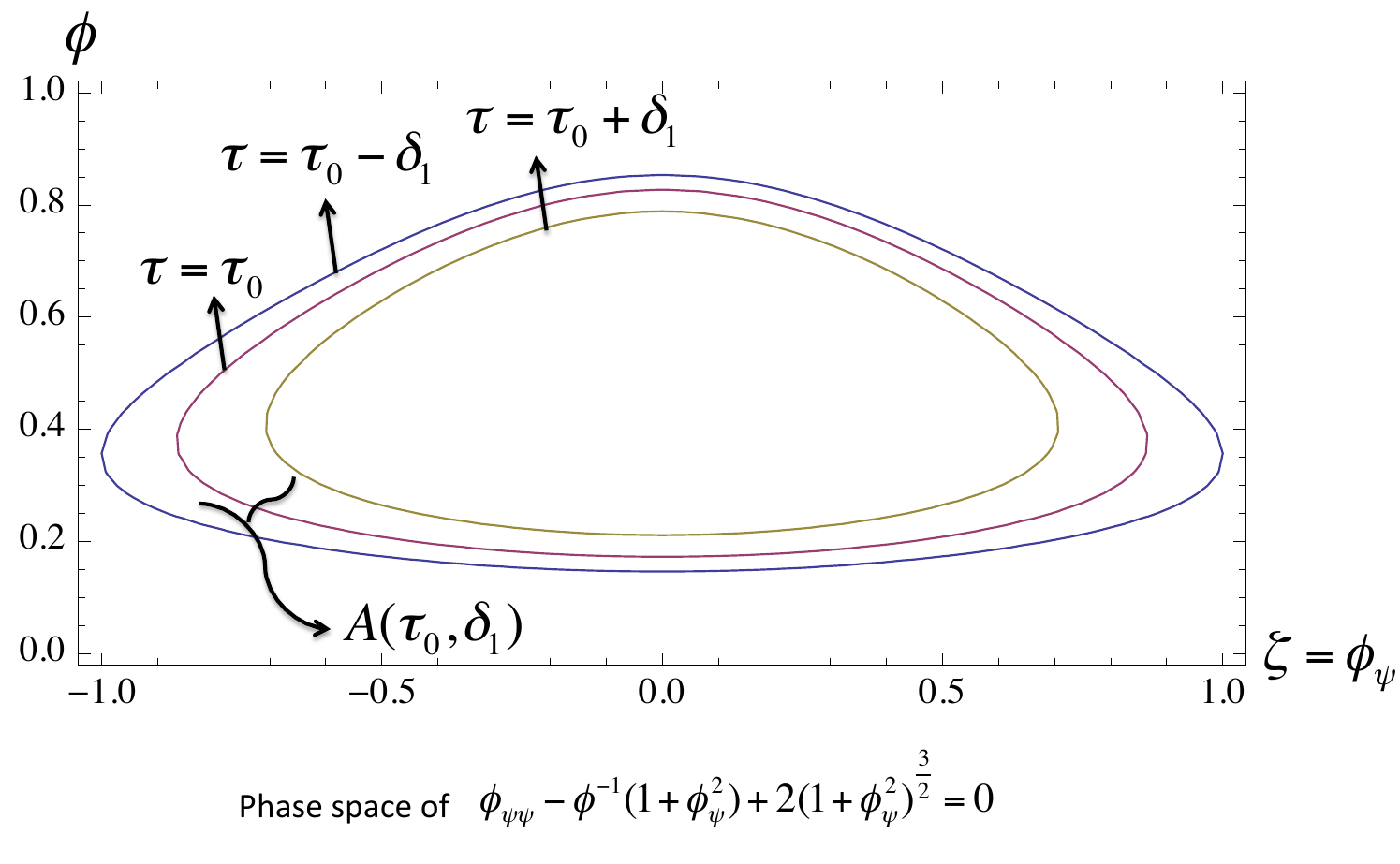}

We have, 

\begin{lem}( Existence on $[0,\frac{L_{\Gamma}}{\varepsilon}]$)
For fixed $\xi,\mu,\omega,\tau(0)$ which satisfy (\ref{C1C2C3C(tau0)}),
there is $\tilde{\delta}>0$ such that when $\tilde{\delta}>\varepsilon\in{\rm PS}(L_{\Gamma},\tau_{0})$,
there is a unique solution $\phi_{\xi,\mu,\omega,\tau(0)}(\psi)\in C_{A(\tau_{0},\delta_{1})}^{1}([0,\frac{1}{\varepsilon}L_{\Gamma}])$
to (\ref{0th mode ODE-1}).

\end{lem}

\begin{proof}

First one can choose $\tilde{\delta}$ small that $|\tau(0)-\tau_{0}|\leq C(\tau_{0})\tilde{\delta}<\delta_{1}.$
Suppose for contradiction the lemma were false. From Step 1, we assume
$T<\frac{L_{\Gamma}}{\varepsilon}$ is the maximal value that the
solution $\phi(\psi)$ can be extended in $C_{A(\tau_{0},\delta_{1})}^{1}([0,T]).$
So we must have $(\phi(T),\zeta(T))\in\partial A(\tau_{0},\delta_{1})$,
that is to say $\tau(\phi(T),\zeta(T))=\tau_{0}\pm\delta_{1}.$ 

From easy calculations, 
\[
\frac{d\tau}{d\psi}=\rho(\psi)\phi\phi_{\psi},
\]
where $\|\rho(\psi)\|_{C^{0}}\leq C(\varepsilon^{2}+\varepsilon\|\xi\|_{C^{0}}+\varepsilon^{3}\|\mu\|_{C^{0}}+\varepsilon^{3}|\omega|).$
So 
\begin{align}
 & |\tau(\phi(T),\zeta(T))-\tau(\phi(0),0)|\nonumber \\
= & |\int_{0}^{T}\rho(\psi)\phi\phi_{\psi}d\psi|\nonumber \\
\leq & CT(\varepsilon^{2}+\varepsilon\|\xi\|_{C^{0}}+\varepsilon^{3}\|\mu\|_{C^{0}}+\varepsilon^{3}|\omega|)\label{eq:rough estimate for tau-rho-1}\\
\leq & C(\varepsilon+\|\xi\|_{C^{0}}+\varepsilon^{2}\|\mu\|_{C^{0}}+\varepsilon^{2}|\omega|)\nonumber \\
\leq & C(\varepsilon+(C_{1}+C_{2}+C_{3})\varepsilon^{2}).\nonumber 
\end{align}
So we have
\begin{align}
 & |\tau(\phi(T),\zeta(T))-\tau_{0}|\nonumber \\
\leq & |\tau(\phi(T),\zeta(T))-\tau(\phi(0),0)|+|\tau(\phi(0),0)-\tau_{0}|\label{eq:rough estimate for tau-rho}\\
\leq & C(\varepsilon+(C_{1}+C_{2}+C_{3})\varepsilon^{2})+C(\tau_{0})\varepsilon.\nonumber 
\end{align}
So we can choose $\tilde{\delta}$ small enough, such that, when $\varepsilon\leq\tilde{\delta}$
\[
|\tau(\phi(T),\zeta(T))-\tau_{0}|\leq\frac{\delta_{1}}{2}
\]
which makes $\tau(\phi(T),\zeta(T))=\tau_{0}\pm\delta_{1}$ impossible.
So when $\varepsilon\leq\tilde{\delta}$, for any prescription of
$\xi,\mu,\omega,\tau(0)$ which satisfies (\ref{C1C2C3C(tau0)}),
the solution $\phi_{\xi,\mu,\omega,\tau(0)}(\psi)$ exists for $\psi\in[0,\frac{L_{\Gamma}}{\varepsilon}].$

\end{proof}

\paragraph{Step 3. Estimates of $\phi_{\xi,\mu,\omega,\tau(0)}(\psi)$}

For simplicity we denote $(\phi_{\xi,\mu,\omega,\tau(0)}(\psi),\frac{\partial\phi_{\xi,\mu,\omega,\tau(0)}(\psi)}{\partial\psi})$
as $(\phi(\psi),\zeta(\psi))$. From (\ref{eq:rough estimate for tau-rho-1})
we know, $\tau(\phi(\psi),\zeta(\psi))$ will keeps in $C(C_{1},C_{2},C_{3},\tau_{0})\varepsilon$
neighborhood of $\tau(0)=\tau(\phi(0),0).$ Actually we can improve
this estimate to $C\varepsilon^{2}$ neighborhood. This result comes
from a simple observation. Note that
\begin{eqnarray}
 &  & |\tau(\phi(T),\zeta(T))-\tau(\phi(0),0)|\nonumber \\
 & = & |\int_{0}^{T}\rho(\psi)\phi\phi_{\psi}d\psi|\label{tau(T0)-tau0}\\
 & = & |\int_{0}^{T}\phi\phi_{\psi}(-\varepsilon^{2}F_{1}(\phi,\phi_{\psi})\star\Pi_{0}(R_{1})+\varepsilon F_{4}(\phi,\phi_{\psi})\xi\nonumber \\
 &  & +\varepsilon^{3}\mu+\varepsilon^{3}\omega\phi^{-1}\phi_{\psi})d\psi|.\nonumber 
\end{eqnarray}
First $\varepsilon\xi=O(\varepsilon^{3}),\varepsilon^{3}\mu=O(\varepsilon^{3}),\varepsilon^{3}\omega\phi^{-1}\phi_{\psi}=O(\varepsilon^{3})$.
The integral of these three terms is $O(\varepsilon^{2})$ because
$0\leq T\leq\frac{L_{\Gamma}}{\varepsilon}$. At first glance the
integral of $\phi\frac{\partial\phi}{\partial\psi}(-\varepsilon^{2})F_{1}(\phi,\phi_{\psi})\star\Pi_{0}(R_{1})$
is of order $O(\varepsilon).$ However there is a cancellation property.
Note that $F_{1}(\phi,\phi_{\psi})\phi\frac{\partial\phi}{\partial\psi}$
nearly has 0 average in one period (up to an error of order $O(\varepsilon$)).
Also we know $\nabla_{\psi}R(\Upsilon,\Upsilon_{\theta},\Upsilon,\Upsilon_{\theta})=O(\varepsilon).$
So in one period the integral is only of size $\varepsilon^{3}$.
When $T$ is as large as $\frac{L_{\Gamma}}{\varepsilon}$, the integral
is of size $\varepsilon^{2}$. 

To be precise, we have

\begin{lem}\label{improvement of the estimates of 0th ODE} There
exists $C=C(C_{1},C_{2},C_{3},\tau_{0})$ which doesn't depend on
$\varepsilon$ such that for $\psi\in[0,\frac{L_{\Gamma}}{\varepsilon}],$
$(\phi(\psi),\zeta(\psi))\in A(\tau(0),C\varepsilon^{2}).$

\end{lem}

\begin{proof}We know that for each $\psi\in[0,\frac{L_{\Gamma}}{\varepsilon}]$,
$|\tau(\phi(\psi),\zeta(\psi))-\tau(\phi(0),0)|\leq C(\varepsilon+\|\xi\|_{C^{0}}+\varepsilon^{2}\|\mu\|_{C^{0}}+\varepsilon^{2}|\omega|)$
from (\ref{eq:rough estimate for tau-rho-1}). Suppose $\phi_{\tau(0)}(\psi)$
defines the standard Delaunay surface with parameter $\tau\equiv\tau(\phi(0),0)$
and $\zeta_{\tau(0)}(\psi)=\frac{d}{d\psi}\phi_{\tau(0)}(\psi)$.
We assume the arc length parameter of the curve $(\phi_{\tau(0)}(\psi),\zeta_{\tau(0)}(\psi))$
in the phase space is $s_{0}$ ($s_{0}$ is a multi-valued function
on the curve $(\phi_{\tau(0)},\zeta_{\tau(0)})$). We can extend $s_{0}$
to a neighborhood of $(\phi_{\tau(0)}(\psi),\zeta_{\tau(0)}(\psi))$
such that $\frac{\partial}{\partial\tau}\perp\frac{\partial}{\partial s_{0}}$
holds everywhere in this neighborhood. We know that in this neighborhood
\[
d\phi^{2}+d\zeta^{2}=<\partial_{s_{0}},\partial_{s_{0}}>ds_{0}^{2}+<\partial_{\tau},\partial_{\tau}>d\tau^{2}
\]
where $<,>$ denotes the inner product of the metric $d\phi^{2}+d\zeta^{2}.$
We know that $<\partial_{s_{0}},\partial_{s_{0}}>=1$ on $(\phi_{\tau(0)}(\psi),\zeta_{\tau(0)}(\psi)).$
Regard $(s_{0},\tau)$ as new local coordinate of $A(\tau(0),C\varepsilon)$.
We define a continuous map $\Phi=\Phi_{\xi,\mu,\omega,\tau(0)}:[0,\frac{L_{\Gamma}}{\varepsilon}]\rightarrow\mathbb{R}$
such that 
\begin{equation}
s_{0}(\phi(\psi),\zeta(\psi))=s_{0}(\phi_{\tau(0)}(\Phi(\psi)),\zeta_{\tau(0)}(\Phi(\psi))).\label{Phi definition}
\end{equation}
Here is a graph which illustrates the definition of $\Phi.$ 

\includegraphics[scale=0.4]{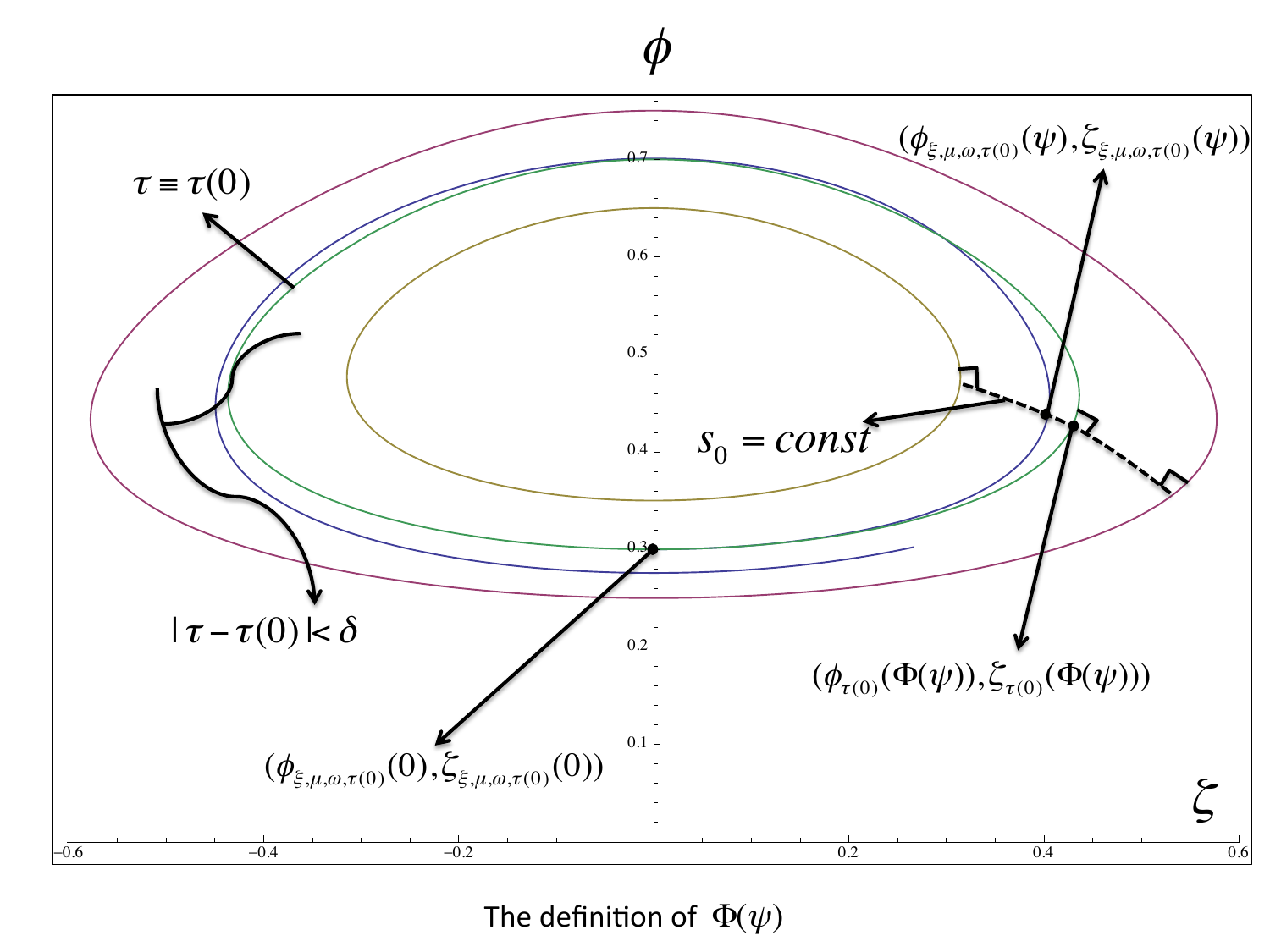}\begin{rmk}From
(\ref{Phi definition}) we see, if $\zeta(\psi)=0$, then $\zeta_{\tau(0)}(\Phi(\psi))=0.$
If $\psi$ is a local minimum point of $\phi$, $\Phi(\psi)$ is a
local minimum point of $\phi_{\tau(0)}$. 

\end{rmk}So we have 
\begin{equation}
\begin{cases}
|\phi(\psi)-\phi_{\tau(0)}(\Phi(\psi))| & \leq C(\tau_{0})(\varepsilon+(C_{1}+C_{2}+C_{3})\varepsilon^{2}),\\
|\zeta(\psi)-\zeta_{\tau(0)}(\Phi(\psi))| & \leq C(\tau_{0})(\varepsilon+(C_{1}+C_{2}+C_{3})\varepsilon^{2}).
\end{cases}\label{eq:rough estimate for phi and zeta}
\end{equation}
\begin{lem}\label{Phi-estimates}
\begin{equation}
\begin{cases}
\Phi(0)=0,\\
|\Phi(\psi)-\psi|\leq\frac{C(\tau_{0})}{\varepsilon}(\varepsilon+(C_{1}+C_{2}+C_{3})\varepsilon^{2}),\\
|\Phi'(\psi)-1|\leq C(\tau_{0})(\varepsilon+(C_{1}+C_{2}+C_{3})\varepsilon^{2}).
\end{cases}\label{eq:first estimate for Phi}
\end{equation}
In particular, $\Phi:[0,\frac{L_{\Gamma}}{\varepsilon}]\rightarrow\Phi([0,\frac{L_{\Gamma}}{\varepsilon}])$
is invertible. 

\end{lem}

We prove Lemma \ref{Phi-estimates} in Appendix \ref{(APP):Proof-of-Lemma Phi-estimate}. 

Now we continue to prove Lemma 3.9. Let $[\tilde{\psi}_{i},\tilde{\psi}_{i+1}]$
be the $i$th period of $\phi_{\tau(0)}(\psi)$. We assume $\psi_{i}=\Phi^{-1}(\tilde{\psi}_{i}).$
We know $\phi_{\psi}(\psi_{i})=0$ . From (\ref{eq:rough estimate for tau-rho-1})(\ref{eq:rough estimate for phi and zeta}),
we have

\begin{align*}
 & |\int_{\psi_{i}}^{\psi_{i+1}}\phi\phi_{\psi}F_{1}(\phi,\phi_{\psi})\star\Pi_{0}(R_{1})d\psi|\\
\leq & |\int_{\tilde{\psi}_{i}}^{\tilde{\psi}_{i+1}}\phi_{\tau(0)}\frac{\partial\phi_{\tau(0)}}{\partial\psi}F_{1}(\phi_{\tau(0)},\frac{\partial\phi_{\tau(0)}}{\partial\psi})|_{\varsigma}\star\Pi_{0}(R_{1})|_{\Phi^{-1}(\varsigma)}d\varsigma|\\
 & +C(\tau_{0})(\varepsilon+\|\xi\|_{C^{0}}+\varepsilon^{2}\|\mu\|_{C^{0}}+\varepsilon^{2}|\omega|)\\
= & |\int_{\tilde{\psi}_{i}}^{\tilde{\psi}_{i+1}}\tilde{\chi}(\varsigma)\star\frac{\partial}{\partial\varsigma}\Pi_{0}(R_{1})|_{\Phi^{-1}(\varsigma)}d\varsigma|\\
 & +C(\tau_{0})(\varepsilon+\|\xi\|_{C^{0}}+\varepsilon^{2}\|\mu\|_{C^{0}}+\varepsilon^{2}|\omega|)\\
= & C(\tau_{0})(\varepsilon+(C_{1}+C_{2}+C_{3})\varepsilon^{2}),
\end{align*}
where $\tilde{\chi}(\varsigma)$ is the primitive of $\phi_{\tau(0)}\frac{\partial\phi_{\tau(0)}}{\partial\psi}F_{1}(\phi_{\tau(0)},\frac{\partial\phi_{\tau(0)}}{\partial\psi})$
and $|\tilde{\chi}(s)|_{C^{0}}\leq C(\tau_{0})\varepsilon.$ Now by
dividing $[0,\frac{L_{\Gamma}}{\varepsilon}]$ into ``periods''
of $\phi(\psi)$, it is easy to prove that, for any $\bar{\psi}\in[0,\frac{L_{\Gamma}}{\varepsilon}]$
\begin{align*}
 & |\int_{0}^{\bar{\psi}}\phi\phi_{\psi}\varepsilon^{2}F_{1}(\phi,\phi_{\psi})\star\Pi_{0}(R_{1})d\psi|\\
\leq & C(\tau_{0})\varepsilon(\varepsilon+(C_{1}+C_{2}+C_{3})\varepsilon^{2}).
\end{align*}

So we can get better estimate for $\tau,$ 
\begin{eqnarray}
|\tau(\phi(T),\zeta(T))-\tau(\phi(0),0)| & \leq & C(\tau_{0})(1+C_{1}+C_{2}+C_{3})\varepsilon^{2}\label{tau-better estimate}
\end{eqnarray}
for $T\in[0,\frac{1}{\varepsilon}L_{\Gamma}].$ 

\end{proof}

Now we can get better estimates for $\Phi$ and $\phi(\psi)$, i.e.
\begin{equation}
\begin{cases}
|\phi(\psi)-\phi_{\tau(0)}(\Phi(\psi))| & \leq C(\tau_{0})(\varepsilon^{2}+(C_{1}+C_{2}+C_{3})\varepsilon^{2}),\\
|\zeta(\psi)-\zeta_{\tau(0)}(\Phi(\psi))| & \leq C(\tau_{0})(\varepsilon^{2}+(C_{1}+C_{2}+C_{3})\varepsilon^{2}).
\end{cases}\label{eq:phi and phi_0}
\end{equation}
and 
\begin{equation}
\begin{cases}
\Phi(0)=0,\\
|\Phi(\psi)-\psi|\leq\frac{C(\tau_{0})}{\varepsilon}(\varepsilon^{2}+\|\xi\|_{C^{0}}+\varepsilon^{2}\|\mu\|_{C^{0}}+\varepsilon^{2}|\omega|),\\
|\Phi^{\prime}(\psi)-1|\leq C(\tau_{0})(\varepsilon^{2}+\|\xi\|_{C^{0}}+\varepsilon^{2}\|\mu\|_{C^{0}}+\varepsilon^{2}|\omega|).
\end{cases}\label{eq:estimates for Phi}
\end{equation}

\begin{cor}\label{phi(psi)-phi_0(psi)}
\begin{eqnarray*}
 &  & \|\phi(\psi)-\phi_{\tau(0)}(\psi)\|_{C_{\varepsilon}^{2,\alpha}}\\
 & \leq & C(\tau_{0})\frac{1}{\varepsilon}(\varepsilon^{2}+\|\xi\|_{C_{x_{0}}^{1}}+\varepsilon^{2}\|\mu\|_{C_{\varepsilon}^{\alpha}}+\varepsilon^{2}|\omega|)\\
 & \leq & C(\tau_{0})(1+C_{1}+C_{2}+C_{3})\varepsilon.
\end{eqnarray*}

\end{cor}

\begin{proof}From (\ref{eq:phi and phi_0}), (\ref{eq:estimates for Phi})
and $\zeta_{\tau(0)}$, $\frac{\partial^{2}\phi_{\tau(0)}}{\partial\psi^{2}}$
are uniformly bounded by $C(\tau_{0})$, we have 
\begin{eqnarray}
|\phi(\psi)-\phi_{\tau(0)}(\psi)| & \leq & |\phi(\psi)-\phi_{\tau(0)}(\Phi(\psi))|+|\phi_{\tau(0)}(\Phi(\psi))-\phi_{\tau(0)}(\psi)|\nonumber \\
 & \leq & C(\tau_{0})(\varepsilon^{2}+\|\xi\|_{C^{0}}+\varepsilon^{2}\|\mu\|_{C^{0}}+\varepsilon^{2}|\omega|)\nonumber \\
 &  & +\sup|\zeta_{\tau(0)}|\cdot|\Phi(\psi)-\psi|\nonumber \\
 & \leq & \frac{C(\tau_{0})}{\varepsilon}(\varepsilon^{2}+\|\xi\|_{C^{0}}+\varepsilon^{2}\|\mu\|_{C^{0}}+\varepsilon^{2}|\omega|)\label{phi,phi tao0}
\end{eqnarray}
and in the same way 
\begin{eqnarray}
|\zeta(\psi)-\zeta_{\tau(0)}(\psi)| & \leq & \frac{C(\tau_{0})}{\varepsilon}(\varepsilon^{2}+\|\xi\|_{C^{0}}+\varepsilon^{2}\|\mu\|_{C^{0}}+\varepsilon^{2}|\omega|).\label{zeta zeta tau 0}
\end{eqnarray}
$|\phi^{-1}-\phi_{\tau(0)}^{-1}|$ also has similar estimates as $\phi$
and $\phi_{\tau(0)}$ has positive lower bounds. Finally from (\ref{0th mode ODE}),
the definition of $\phi_{\tau(0)}$ and (\ref{eq:phi and phi_0})(\ref{zeta zeta tau 0})
we have 
\begin{eqnarray*}
\|\frac{\partial^{2}\phi}{\partial\psi^{2}}(\psi)-\frac{\partial^{2}\phi_{\tau(0)}}{\partial\psi^{2}}(\psi)\|_{C_{\varepsilon}^{\alpha}} & \leq & \frac{C(\tau_{0})}{\varepsilon}(\varepsilon^{2}+\|\xi\|_{C_{x_{0}}^{1}}+\varepsilon^{2}\|\mu\|_{C_{\varepsilon}^{\alpha}}+\varepsilon^{2}|\omega|).
\end{eqnarray*}

\end{proof}

\begin{cor}\label{F(phi)-F(phi0)} Suppose $F(\theta_{1},\theta_{2})$
is a function with the following property. For any $\tilde{C}_{1}(\tau_{0}),\tilde{C}_{2}(\tau_{0})$,
there exists $\widetilde{C}_{3}(\tilde{C}_{1},\tilde{C}_{2})$ such
that 
\[
\sup_{|\theta_{1}|\leq\tilde{C}_{1},|\theta_{2}|\leq\tilde{C}_{2}}|F|+|\partial_{\theta_{1}}F|+|\partial_{\theta_{2}}F|\leq\tilde{C}_{3}.
\]
 Then for any 
\[
\|R(\psi)\|_{C_{x_{0}}^{1}(\Gamma)}\leq C
\]
we have 
\begin{align*}
 & |\int_{0}^{\psi}F(\phi,\phi_{\psi})R(\psi)d\psi-\int_{0}^{\Phi(\psi)}F(\phi_{\tau(0)}(\varsigma),\frac{\partial\phi_{\tau(0)}}{\partial\varsigma})R(\varsigma)d\varsigma|\\
\leq & \frac{C(\tau_{0})}{\varepsilon}(\varepsilon^{2}+\|\xi\|_{C^{0}}+\varepsilon^{2}\|\mu\|_{C^{0}}+\varepsilon^{2}|\omega|)\|R\|_{C_{x_{0}}^{1}}.
\end{align*}
In particular, when $R(\psi)\equiv1$, 
\begin{align*}
 & |\int_{0}^{\psi}F(\phi,\phi_{\psi})d\psi-\int_{0}^{\Phi(\psi)}F(\phi_{\tau(0)}(\varsigma),\frac{\partial\phi_{\tau(0)}}{\partial\varsigma})d\varsigma|\\
\leq & \frac{C(\tau_{0})}{\varepsilon}(\varepsilon^{2}+\|\xi\|_{C^{0}}+\varepsilon^{2}\|\mu\|_{C^{0}}+\varepsilon^{2}|\omega|).
\end{align*}

\end{cor}

\begin{proof}
\begin{align*}
 & |\int_{0}^{\psi}F(\phi,\phi_{\psi})R(\psi)d\psi-\int_{0}^{\Phi(\psi)}F(\phi_{\tau(0)}(\varsigma),\frac{\partial\phi_{\tau(0)}}{\partial\varsigma})R(\varsigma)d\varsigma|\\
\leq & \int_{0}^{\Phi(\psi)}|F(\phi,\phi_{\psi})(\Phi^{-1}(\varsigma))(\Phi^{-1})^{\prime}R(\Phi^{-1}(\varsigma))-F(\phi_{\tau(0)}(\varsigma),\frac{\partial\phi_{\tau(0)}}{\partial\varsigma})R(\varsigma)|d\varsigma\\
\leq & \frac{C(\tau_{0})}{\varepsilon}(|\phi(\Phi^{-1}(\varsigma))-\phi_{\tau(0)}(\varsigma)||R|+|\phi_{\psi}(\Phi^{-1}(\varsigma))-\frac{\partial\phi_{\tau(0)}(\varsigma)}{\partial\varsigma}||R|\\
 & +|(\Phi^{-1})'-1||R|+|\partial_{\psi}R||\Phi^{-1}(\varsigma)-\varsigma|)\\
\leq & \frac{C(\tau_{0})}{\varepsilon}(\varepsilon^{2}+\|\xi\|_{C^{0}}+\varepsilon^{2}\|\mu\|_{C^{0}}+\varepsilon^{2}|\omega|)\|R\|_{C_{x_{0}}^{1}}.
\end{align*}

\end{proof}

\paragraph{Step 4. The linearized equations}

In this step we analyze the linearized equations of (\ref{0th mode ODE-1}).
In Step 2, we have got a solution $\phi(\psi)=\phi_{\xi,\mu,\omega,\tau(0)}(\psi),\psi\in[0,\frac{L_{\Gamma}}{\varepsilon}].$
We can linearize the equation in different ways. First, we fix $\xi,\mu,\omega$
and perturb the initial values. Then we fixed the initial values and
perturb $\xi,\mu,\omega$.

For fixed $\xi,\mu,\omega$ we suppose $\phi_{t}$ is a class of solutions
to 
\[
\frac{\partial^{2}\phi}{\partial\psi^{2}}-\phi^{-1}(1+(\frac{\partial\phi}{\partial\psi})^{2})+(2+\rho)(1+(\frac{\partial\phi}{\partial\psi})^{2})^{\frac{3}{2}}=0,
\]
(with different initial values) and $\frac{d}{dt}\phi_{t}|_{t=0}=\beta(\psi)$.
Then we can calculate

\begin{eqnarray}
 &  & \mathcal{L}_{\xi,\mu,\omega,\tau(0)}\beta(\psi)\nonumber \\
 & = & \frac{\partial^{2}\beta}{\partial\psi^{2}}+(6(1+\phi_{\psi}^{2})^{\frac{1}{2}}\phi_{\psi}-2\phi^{-1}\phi_{\psi}+\bar{F}_{1})\frac{\partial\beta}{\partial\psi}+(\phi^{-2}(1+\phi_{\psi}^{2})+\bar{F}_{2})\beta\nonumber \\
 & = & 0\label{linearized operator}
\end{eqnarray}
where $\|\bar{F}_{1}\|_{C_{\varepsilon}^{\alpha}}+\|\bar{F}_{2}\|_{C_{\varepsilon}^{\alpha}}\leq\varepsilon^{2}(C+C_{1}+C_{2}+C_{3})$.
First we analyze the fundamental solutions to the linearized equation
(\ref{linearized operator}). Suppose $\beta_{1}(\psi),\beta_{2}(\psi)$
satisfy
\[
\mathcal{L}_{\xi,\mu,\omega,\tau(0)}\beta_{i}(\psi)=0,\psi\in[0,\frac{L_{\Gamma}}{\varepsilon}]
\]
and 
\[
\left(\begin{array}{cc}
\beta_{1}(0) & \beta_{2}(0)\\
\frac{\partial\beta_{1}}{\partial\psi}(0) & \frac{\partial\beta_{2}}{\partial\psi}(0)
\end{array}\right)=\left(\begin{array}{cc}
1 & 0\\
0 & 1
\end{array}\right).
\]

Let $R(\psi)=\left|\begin{array}{cc}
\beta_{1}(\psi) & \beta_{2}(\psi)\\
\beta_{1}'(\psi) & \beta_{2}'(\psi)
\end{array}\right|.$ Easy calculation yields
\[
\frac{d}{d\psi}\log R(\psi)=-(6(1+\phi_{\psi}^{2})^{\frac{1}{2}}\phi_{\psi}-2\phi^{-1}\phi_{\psi}+\bar{F}_{1}).
\]
Suppose $\phi_{\tau(0)}$ is the solution to (\ref{0th mode ODE-1}),
with $\rho$ replaced by $0.$ From Corollary \ref{F(phi)-F(phi0)},
and the fact that $6(1+(\frac{\partial\phi_{\tau(0)}}{\partial\varsigma})^{2})^{\frac{1}{2}}\frac{\partial\phi_{\tau(0)}}{\partial\varsigma}-2\phi_{\tau(0)}^{-1}\frac{\partial\phi_{\tau(0)}}{\partial\varsigma}$
has $0$ average in one period, we deduce that there exists $C=C(\tau_{0},C_{1},C_{2},C_{3})>0$
such that 
\begin{equation}
{\rm e}^{-C}\leq R(\psi)\leq{\rm e}^{C}.\label{R(t) estimate}
\end{equation}

We define an operator
\begin{eqnarray*}
 &  & \mathcal{L}_{\tau(0)}\beta(\varsigma)\\
 & = & \frac{\partial^{2}\beta}{\partial\varsigma^{2}}+(6(1+(\frac{\partial\phi_{\tau(0)}}{\partial\varsigma})^{2})^{\frac{1}{2}}\frac{\partial\phi_{\tau(0)}}{\partial\varsigma}-2\phi_{\tau(0)}^{-1}\frac{\partial\phi_{\tau(0)}}{\partial\varsigma})\frac{\partial\beta}{\partial\varsigma}\\
 &  & +\phi_{\tau(0)}^{-2}(1+(\frac{\partial\phi_{\tau(0)}}{\partial\varsigma})^{2})\beta.
\end{eqnarray*}
Using Lemma \ref{Phi-estimates} and Corollary \ref{phi(psi)-phi_0(psi)}
we can make comparison between $\mathcal{L}_{\xi,\mu,\omega,\tau(0)}$
and $\mathcal{L}_{\tau(0)}.$

We denote the $i$th local minimum of $\phi$ in $[0,\frac{L_{\Gamma}}{\varepsilon}]$
as $\psi_{i},i=0,1,2,\cdots$. So $\psi_{0}=0$ and $[\psi_{i-1},\psi_{i}]$
resembles the $i$th ``period'' of $\phi(\psi).$ $\tilde{\psi}_{i}=\Phi(\psi_{i})$
is the $i$th local minimum of $\phi_{\tau(0)}$. It is obvious that
$\tilde{\psi}_{i}=i\tilde{\psi}_{1}$. When $\lambda\in[0,1]$, let
$\psi=(1-\lambda)\psi_{i-1}+\lambda\psi_{i}$. From (\ref{eq:phi and phi_0}),(\ref{eq:estimates for Phi})
we have
\begin{align*}
 & |\phi((1-\lambda)\psi_{i-1}+\lambda\psi_{i})-\phi_{\tau(0)}((1-\lambda)\tilde{\psi}_{i-1}+\lambda\tilde{\psi}_{i})|\\
\leq & |\phi(\psi)-\phi_{\tau(0)}(\Phi(\psi))|+|\phi_{\tau(0)}(\Phi(\psi))-\phi_{\tau(0)}((1-\lambda)\tilde{\psi}_{i-1}+\lambda\tilde{\psi}_{i})|\\
\leq & C(\tau_{0})(\varepsilon^{2}+\|\xi\|_{C^{0}}+\varepsilon^{2}\|\mu\|_{C^{0}}+\varepsilon^{2}|\omega|),
\end{align*}
and 

\begin{align*}
 & |\zeta((1-\lambda)\psi_{i-1}+\lambda\psi_{i})-\zeta_{\tau(0)}((1-\lambda)\tilde{\psi}_{i-1}+\lambda\tilde{\psi}_{i})|\\
\leq & C(\tau_{0})(\varepsilon^{2}+\|\xi\|_{C^{0}}+\varepsilon^{2}\|\mu\|_{C^{0}}+\varepsilon^{2}|\omega|).
\end{align*}
Suppose $\beta_{1,i-1},\beta_{2,i-1}$ solves 
\begin{align*}
\mathcal{L}_{\xi,\mu,\omega,\tau(0)}\beta_{j,i-1}(\psi) & =0,\psi\in[\psi_{i-1},\psi_{i}],j=1,2,
\end{align*}
with
\[
\left(\begin{array}{cc}
\beta_{1,i-1}(\psi_{i-1}) & \beta_{2,i-1}(\psi_{i-1})\\
\beta_{1,i-1}^{\prime}(\psi_{i-1}) & \beta_{2,i-1}^{\prime}(\psi_{i-1})
\end{array}\right)=\left(\begin{array}{cc}
1 & 0\\
0 & 1
\end{array}\right).
\]

Suppose $W_{1,i-1},W_{2,i-1}$ solves
\[
\mathcal{L}_{\tau(0)}W_{j,i-1}=0,\psi\in[\tilde{\psi}_{i-1},\tilde{\psi}_{i}],j=1,2,
\]
with
\[
\left(\begin{array}{cc}
W_{1,i-1}(\tilde{\psi}_{i-1}) & W_{2,i-1}(\tilde{\psi}_{i-1})\\
W_{1,i-1}^{\prime}(\tilde{\psi}_{i-1}) & W_{2,i-1}^{\prime}(\tilde{\psi}_{i-1})
\end{array}\right)=\left(\begin{array}{cc}
1 & 0\\
0 & 1
\end{array}\right).
\]
We know 
\begin{align*}
W_{1,i-1}(\varsigma) & =h_{1}(\tau(0))(\varsigma-\tilde{\psi}_{i-1})\frac{\partial\phi_{\tau(0)}}{\partial\varsigma}+v_{\tau(0)}(\varsigma),\\
W_{2,i-1}(\varsigma) & =h_{2}(\tau(0))\frac{\partial\phi_{\tau(0)}}{\partial\varsigma},
\end{align*}
where $v_{\tau(0)}(\varsigma)$ has period $\tilde{\psi}_{1}.$

Comparing the coefficients of $\mathcal{L}_{\xi,\mu,\omega,\tau(0)}$
with that of $\mathcal{L}_{\tau(0)}$ we have
\begin{align}
 & |\beta_{j,i-1}((1-\lambda)\psi_{i-1}+\lambda\psi_{i})-W_{j,i-1}((1-\lambda)\tilde{\psi}_{i-1}+\lambda\tilde{\psi}_{i})|\nonumber \\
 & +|\beta_{j,i-1}'((1-\lambda)\psi_{i-1}+\lambda\psi_{i})-W_{j,i-1}'((1-\lambda)\tilde{\psi}_{i-1}+\lambda\tilde{\psi}_{i})|\nonumber \\
\leq & C(\tau_{0})(1+C_{1}+C_{2}+C_{3})\varepsilon^{2},\label{Beta piecewise estimate}
\end{align}
which implies the following lemma,

\begin{lem} \label{estimates of beta(j,i-1)}
\begin{align*}
 & \|\beta_{1,i-1}(\psi)-(h_{1}(\tau(0))(\psi-\psi_{i})\frac{\partial\phi}{\partial\psi}+v_{1}(\psi))\|_{C_{\varepsilon}^{1}([\psi_{i-1},\psi_{i}])}\\
 & +\|\beta_{2,i-1}(\psi)-h_{2}(\tau(0))\frac{\partial\phi}{\partial\psi}\|_{C_{\varepsilon}^{1}([\psi_{i-1},\psi_{i}])}\\
\leq & C(\tau_{0})(1+C_{1}+C_{2}+C_{3})\varepsilon^{2},
\end{align*}
where $v_{1}((1-\lambda)\psi_{i-1}+\lambda\psi_{i})=v_{\tau(0)}(\lambda(\tilde{\psi}_{i}-\tilde{\psi}_{i-1})).$ 

\end{lem}

From Lemma \ref{estimates of beta(j,i-1)} we have 
\begin{equation}
\left(\begin{array}{cc}
\beta_{1,i-1}(\psi_{i}) & \beta_{2,i-1}(\psi_{i})\\
\frac{\partial\beta_{1,i-1}}{\partial\psi}(\psi_{i}) & \frac{\partial\beta_{2,i-1}}{\partial\psi}(\psi_{i})
\end{array}\right)=\left(\begin{array}{cc}
1+e_{11}^{i} & e_{12}^{i}\\
\kappa+e_{21}^{i} & 1+e_{22}^{i}
\end{array}\right),\label{beta(i-1)(psi i)}
\end{equation}
where $|e_{jk}^{i}|\leq C(\tau_{0})(1+C_{1}+C_{2}+C_{3})\varepsilon^{2}.$

From the theory of linear ODE, we know 
\begin{align*}
\left(\begin{array}{cc}
\beta_{1}(\psi_{i}) & \beta_{2}(\psi_{i})\\
\frac{\partial\beta_{1}}{\partial\psi}(\psi_{i}) & \frac{\partial\beta_{2}}{\partial\psi}(\psi_{i})
\end{array}\right) & =\left(\begin{array}{cc}
1+e_{11}^{i} & e_{12}^{i}\\
\kappa+e_{21}^{i} & 1+e_{22}^{i}
\end{array}\right)\cdots\left(\begin{array}{cc}
1+e_{11}^{1} & e_{12}^{1}\\
\kappa+e_{21}^{1} & 1+e_{22}^{1}
\end{array}\right)\\
 & =\left(\begin{array}{cc}
A_{11}^{i} & A_{12}^{i}\\
A_{21}^{i} & A_{22}^{i}
\end{array}\right).
\end{align*}
By analysis of the matrix, we can prove

\begin{lem}\label{Aij estimate}

\begin{align*}
\exp(-C(\tau_{0})(1+C_{1}+C_{2}+C_{3}))\leq & A_{11}^{i}+A_{22}^{i}+\varepsilon A_{21}^{i}\leq\exp C(\tau_{0})(1+C_{1}+C_{2}+C_{3}),\\
|A_{12}^{i}|\leq & C(\tau_{0},C_{1},C_{2},C_{3})\varepsilon\exp C(\tau_{0})(1+C_{1}+C_{2}+C_{3}).
\end{align*}

\end{lem}

We will prove this result in Appendix \ref{(APP)Aij estimate}. So
from Lemma \ref{estimates of beta(j,i-1)} we have

\begin{lem}\label{estimate for the fundamental solution} For $\xi,\mu,\omega,\tau(0)$
which satisfy (\ref{C1C2C3C(tau0)}), there is $\delta>0$ such that
when $0<\varepsilon<\delta$, in the interval $[\psi_{i},\psi_{i+1}]$,
we have 
\begin{align*}
 & \|\beta_{1}(\psi)-(A_{11}^{i}[h_{1}(\tau(0))(\psi-\psi_{i})\frac{\partial\phi}{\partial\psi}+v_{i}(\psi)]+A_{21}^{i}h_{2}(\tau(0))\frac{\partial\phi}{\partial\psi})\|_{C_{\varepsilon}^{1}}\\
\leq & C(\tau_{0},C_{1},C_{2},C_{3})\varepsilon,\\
 & \|\beta_{2}(\psi)-(A_{12}^{i}[h_{1}(\tau(0))(\psi-\psi_{i})\frac{\partial\phi}{\partial\psi}+v_{i}(\psi)]+A_{22}^{i}h_{2}(\tau(0))\frac{\partial\phi}{\partial\psi})\|_{C_{\varepsilon}^{1}}\\
\leq & C(\tau_{0},C_{1},C_{2},C_{3})\varepsilon^{2}.
\end{align*}
where $v_{i}((1-\lambda)\psi_{i}+\lambda\psi_{i+1})=v_{\tau(0)}((1-\lambda)\Phi(\psi_{i})+\lambda\Phi(\psi_{i+1}))$
for $\lambda\in[0,1]$ and $A_{kl}^{i}$ satisfy the inequalities
in Lemma \ref{Aij estimate}. 

In particular, $\beta_{1}$ has linear growth and $\beta_{2}$ is
bounded.

\end{lem}

We denote 
\[
\begin{cases}
\frac{d}{dt}\phi_{\xi,\mu+t\Delta\mu,\omega,\tau(0)}(\psi)|_{t=0} & =\beta_{\mu}(\psi),\\
\frac{d}{dt}\phi_{\xi+t\Delta\xi,\mu,\omega,\tau(0)}(\psi)|_{t=0} & =\beta_{\xi}(\psi),\\
\frac{d}{dt}\phi_{\xi,\mu,\omega+t,\tau(0)}(\psi)|_{t=0} & =\beta_{\omega}(\psi).
\end{cases}
\]

Using (\ref{R(t) estimate}), Lemma \ref{Aij estimate}, \ref{estimate for the fundamental solution},
we have the following estimates

\begin{lem}\label{Linearization estimates}

\begin{align}
\|\beta_{\mu}(\psi)\|_{C_{\varepsilon}^{1}} & \leq C(\tau_{0},C_{1},C_{2},C_{3})\varepsilon\|\Delta\mu\|_{C^{0}},\label{beta-mu estimate}\\
\|\beta_{\xi}(\psi)\|_{C_{\varepsilon}^{1}} & \leq C(\tau_{0},C_{1},C_{2},C_{3})\|\Delta\xi\|_{C_{x_{0}}^{1}},\label{beta xi estimate}\\
\|\beta_{\omega}(\psi)\|_{C_{\varepsilon}^{1}} & \leq C(\tau_{0},C_{1},C_{2},C_{3})\varepsilon.\label{beta omega estimate}
\end{align}

Moreover if we consider 
\[
\frac{\partial(\Delta\tau,\zeta(\frac{L_{\Gamma}}{\varepsilon}))}{\partial(\omega,\phi(0))}
\]
where $\Delta\tau=\tau(\frac{L_{\Gamma}}{\varepsilon})-\tau(0)$ and
$\zeta=\phi_{\psi}.$ We can get
\begin{equation}
\begin{cases}
\frac{\partial\Delta\tau}{\partial\omega} & \geq C_{5}(\tau_{0})\varepsilon^{2}\\
|\frac{\partial\Delta\tau}{\partial\phi(0)}| & \leq K_{1}(\tau_{0},C_{1},C_{2},C_{3})\varepsilon\\
|\frac{\partial\zeta(\frac{L_{\Gamma}}{\varepsilon})}{\partial\omega}| & =|\beta_{\omega}^{\prime}(\frac{L_{\Gamma}}{\varepsilon})|\leq K_{2}(\tau_{0},C_{1},C_{2},C_{3})\varepsilon\\
\frac{\partial\zeta(\frac{L_{\Gamma}}{\varepsilon})}{\partial\phi(0)} & =\beta_{1}^{\prime}(\frac{L_{\Gamma}}{\varepsilon})\geq\exp(-C(\tau_{0})(1+C_{1}+C_{2}+C_{3}))\frac{1}{\varepsilon}.
\end{cases}\label{estimates for Jacobi matrix}
\end{equation}

\end{lem}

We will prove this lemma in Appendix \ref{(APP)Proof-of-Lemma Linearization}.
We note that the proof of (\ref{beta xi estimate}) relies on Lemma
\ref{average 0 lemma}, the ``average $0$'' lemma. 

\paragraph{Step 5. Match the boundary value}

\begin{lem}\label{match the boundary value}

For $C_{1},C_{2}>0$, if $\|\xi\|_{C_{x_{0}}^{1}}\leq C_{1}\varepsilon^{2}$
and $\|\mu\|_{C_{\varepsilon}^{\alpha}}\leq C_{2}$, we can choose
$C_{3},C_{4}$ and $\delta$ such that when $0<\varepsilon<\delta$,
there are unique $\mbox{\ensuremath{\omega_{\xi,\mu}}}$ and $\phi(0)_{\xi,\mu}$
(or $\tau(0)_{\xi,\mu}$) such that 
\[
\begin{cases}
\phi_{\xi,\mu,\mbox{\ensuremath{\omega_{\xi,\mu}},}\phi(0)_{\xi,\mu}}(0)=\phi_{\xi,\mu,\mbox{\ensuremath{\omega_{\xi,\mu}},}\phi(0)_{\xi,\mu}}(\frac{L_{\Gamma}}{\varepsilon})=\phi(0)_{\xi,\mu},\\
\phi_{\xi,\mu,\mbox{\ensuremath{\omega_{\xi,\mu}},}\phi(0)_{\xi,\mu}}^{\prime}(0)=\phi_{\xi,\mu,\mbox{\ensuremath{\omega_{\xi,\mu}},}\phi(0)_{\xi,\mu}}^{\prime}(\frac{L_{\Gamma}}{\varepsilon})=0
\end{cases}
\]
and 
\begin{eqnarray*}
|\mbox{\ensuremath{\omega_{\xi,\mu}}}| & \leq & C_{3},\\
|\phi(0)_{\xi,\mu}-\frac{1-\sqrt{1-4\tau_{0}}}{2}| & \leq & C_{4}\varepsilon^{2}.
\end{eqnarray*}

\end{lem}

\begin{proof} We begin with $\omega=0,\tau(0)=\tau_{0}.$ From (\ref{tau-better estimate})
and Corollary \ref{phi(psi)-phi_0(psi)} we have for a particular
$\tilde{C}=\tilde{C}(\tau_{0})$ 
\begin{align*}
|\Delta\tau| & =|\tau(\frac{L_{\Gamma}}{\varepsilon})-\tau(0)|\leq\tilde{C}(\tau_{0})\varepsilon^{2}(1+C_{1}+C_{2})\\
|\zeta(\frac{L_{\Gamma}}{\varepsilon})| & \leq\tilde{C}(\tau_{0})\varepsilon(1+C_{1}+C_{2}).
\end{align*}
We choose 
\[
C_{3}=\frac{4\tilde{C}(\tau_{0})(1+C_{1}+C_{2})}{C_{5}(\tau_{0})}
\]
 and 
\[
C_{4}=4\tilde{K}_{2}C_{5}^{-1}\tilde{C}(\tau_{0})(1+C_{1}+C_{2})\exp(C(\tau_{0})(1+C_{1}+C_{2}+C_{3}))
\]
where $\tilde{K}_{2}=\tilde{K}_{2}(\tau_{0},C_{1},C_{2},C_{3},C_{5})=K_{2}(\tau_{0},C_{1},C_{2},C_{3})+C_{5}$.
We choose 
\[
0<\varepsilon\leq C_{6}=\min\{\frac{C(\tau_{0})}{C_{4}},\frac{C_{5}\exp(-C(\tau_{0})(1+C_{1}+C_{2}+C_{3}))}{2|K_{1}\tilde{K}_{2}|+1}\}.
\]
 Note that first $C_{4}\varepsilon^{2}\leq C(\tau_{0})\varepsilon$,
so one can work with a uniform constant $C_{5}(\tau_{0})>0$ such
that 
\[
\frac{\partial\Delta\tau}{\partial\omega}\geq C_{5}\varepsilon^{2}.
\]
First we prove the uniqueness. For $\omega_{1},\omega_{2}\in[-C_{3},C_{3}],\phi(0)_{1},\phi(0)_{2}\in[\frac{1-\sqrt{1-4\tau_{0}}}{2}-C_{4}\varepsilon^{2},\frac{1-\sqrt{1-4\tau_{0}}}{2}+C_{4}\varepsilon^{2}],$
if $\Delta\tau$ and $\zeta(\frac{L_{\Gamma}}{\varepsilon})$ takes
the same value at $(\omega_{1},\phi(0)_{1})$ and $(\omega_{2},\phi(0)_{2})$
then there exists $\omega_{3},\omega_{4}$ lying between $\omega_{1},\omega_{2}$
and $\phi(0)_{3},\phi(0)_{4}$ lying between $\phi(0)_{1},\phi(0)_{2}$
such that 
\[
\left[\begin{array}{cc}
\frac{\partial\Delta\tau}{\partial\omega}|_{(\omega_{3},\phi(0)_{1})} & \frac{\partial\Delta\tau}{\partial\phi(0)}|_{(\omega_{2},\phi(0)_{3})}\\
\frac{\partial\zeta(\frac{L_{\Gamma}}{\varepsilon})}{\partial\omega}|_{(\omega_{4},\phi(0)_{1})} & \frac{\partial\zeta(\frac{L_{\Gamma}}{\varepsilon})}{\partial\phi(0)}|_{(\omega_{2},\phi(0)_{4})}
\end{array}\right]\left[\begin{array}{c}
\omega_{1}-\omega_{2}\\
\phi(0)_{1}-\phi(0)_{2}
\end{array}\right]=\left[\begin{array}{c}
0\\
0
\end{array}\right].
\]
From estimate (\ref{estimates for Jacobi matrix}), the matrix $\left[\begin{array}{cc}
\frac{\partial\Delta\tau}{\partial\omega}|_{(\omega_{3},\phi(0)_{1})} & \frac{\partial\Delta\tau}{\partial\phi(0)}|_{(\omega_{2},\phi(0)_{3})}\\
\frac{\partial\zeta(\frac{L_{\Gamma}}{\varepsilon})}{\partial\omega}|_{(\omega_{4},\phi(0)_{1})} & \frac{\partial\zeta(\frac{L_{\Gamma}}{\varepsilon})}{\partial\phi(0)}|_{(\omega_{2},\phi(0)_{4})}
\end{array}\right]$ is invertible as long as $\varepsilon\leq C_{6}.$ So $\omega_{1}=\omega_{2}$
and $\phi(0)_{1}=\phi(0)_{2}.$ We have proved the uniqueness.

For the existence first we perturb $\omega$ within $\frac{\tilde{C}(1+C_{1}+C_{2})}{C_{5}}$
such that $\Delta\tau=0.$ Then $\zeta(\frac{L_{\Gamma}}{\varepsilon})$
will change no more than $K_{2}\frac{\tilde{C}(1+C_{1}+C_{2})}{C_{5}}\varepsilon.$
Then we perturb $\phi(0)$ within 
\[
\tilde{K}_{2}C_{5}^{-1}\tilde{C}(1+C_{1}+C_{2})\exp(C(\tau_{0})(1+C_{1}+C_{2}+C_{3}))\varepsilon^{2}
\]
 such that $\zeta(\frac{L_{\Gamma}}{\varepsilon})=0$. Then $\Delta\tau$
will change no more than $K_{1}\tilde{K}_{2}C_{5}^{-1}\tilde{C}(1+C_{1}+C_{2})\exp(C(\tau_{0})(1+C_{1}+C_{2}+C_{3}))\varepsilon^{3}.$
From 
\[
\varepsilon\leq\frac{C_{5}\exp(-C(\tau_{0})(1+C_{1}+C_{2}+C_{3}))}{2|K_{1}\tilde{K}_{2}|+1}
\]
 we know
\begin{align*}
 & K_{1}\tilde{K}_{2}C_{5}^{-1}\tilde{C}(1+C_{1}+C_{2})\exp(C(\tau_{0})(1+C_{1}+C_{2}+C_{3}))\varepsilon^{3}\\
\leq & \frac{1}{2}\tilde{C}(1+C_{1}+C_{2})\varepsilon^{2}.
\end{align*}
Then by an iteration argument, we have, there exists $|\omega_{\xi,\mu}|\leq C_{3}$
and $|\phi(0)_{\xi,\mu}-\frac{1-\sqrt{1-4\tau_{0}}}{2}|\leq C_{4}\varepsilon^{2}$
such that $\Delta\tau=0$ and $\zeta(\frac{L_{\Gamma}}{\varepsilon})=0.$
From $|\phi_{\xi,\mu,\omega_{_{\xi,\mu}},\tau(0)_{\xi,\mu}}(\frac{L_{\Gamma}}{\varepsilon})-\phi_{\xi,\mu,\omega_{_{\xi,\mu}},\tau(0)_{\xi,\mu}}(0)|\leq C\varepsilon,$
we know $\phi_{\xi,\mu,\omega_{_{\xi,\mu}},\tau(0)_{\xi,\mu}}(\frac{L_{\Gamma}}{\varepsilon})=\phi_{\xi,\mu,\omega_{_{\xi,\mu}},\tau(0)_{\xi,\mu}}(0)$.

\end{proof}

From Corollary \ref{phi(psi)-phi_0(psi)} and Lemma \ref{match the boundary value}
we have 

\begin{cor}\label{phi(xi,mu) estimate} For given $C_{1},C_{2}$
and $\tau_{0}$ there is $C=C(\tau_{0},C_{1},C_{2})$ such that 
\[
\|\phi_{\xi,\mu}(\psi)-\phi_{\tau_{0}}(\psi)\|_{C_{\varepsilon}^{2,\alpha}}\leq C\varepsilon.
\]

\end{cor}

As $|\tau(0)_{\xi,\mu}-\tau_{0}|\leq C\varepsilon^{2}$, we know $(\phi_{\xi,\mu},\zeta_{\xi,\mu})$
is within $C\varepsilon^{2}$ neighborhood of $(\phi_{\tau_{0}},\zeta_{\tau_{0}}).$
So similar to the construction of $\Phi$, we can construct $\tilde{\Phi}$
by projecting $(\phi_{\xi,\mu},\zeta_{\xi,\mu})$ to the orbit $(\phi_{\tau_{0}},\zeta_{\tau_{0}}).$
We have 

\begin{cor}\label{Phi-global}For $\phi_{\xi,\mu,\omega_{\xi,\mu}\phi(0)_{\xi,\mu}}(\psi)$
we have a $C^{1}$ map 
\[
\tilde{\Phi}=\tilde{\Phi}_{\xi,\mu,\tau_{0}}:\Gamma\rightarrow\Gamma
\]
 such that 
\begin{align*}
 & |\phi_{\xi,\mu,\omega_{\xi,\mu}\phi(0)_{\xi,\mu}}(\psi)-\phi_{\tau_{0}}(\tilde{\Phi}(\psi))|+|\zeta_{\xi,\mu,\omega_{\xi,\mu}\phi(0)_{\xi,\mu}}(\psi)-\zeta_{\tau_{0}}(\tilde{\Phi}(\psi))|\\
\leq & C(\tau_{0},C_{1},C_{2})\varepsilon^{2}
\end{align*}
 and
\[
\begin{cases}
|\tilde{\Phi}(\psi)-\psi| & \leq C(\tau_{0},C_{1},C_{2})\varepsilon,\\
|\tilde{\Phi}'(\psi)-1| & \leq C(\tau_{0},C_{1},C_{2})\varepsilon^{2}.
\end{cases}
\]
Moreover suppose $\psi_{i}$ is the $i$th local minimum point of
$\phi_{\xi,\mu,\omega_{\xi,\mu}\phi(0)_{\xi,\mu}}$, then $\tilde{\Phi}(\psi_{i})$
is the $i$th local minimum point of $\phi_{\tau_{0}}$ . In particular
$\tilde{\Phi}(0)=0,\tilde{\Phi}(L_{\Gamma})=L_{\Gamma}.$

\end{cor}

\paragraph{Step 6. Main estimates of 0th mode}

For 
\[
\|\xi_{1}\|_{C_{x_{0}}^{1}},\|\xi_{2}\|_{C_{x_{0}}^{1}}\leq C_{1}\varepsilon^{2},\|\mu_{1}\|_{C_{\varepsilon}^{\alpha}},\|\mu_{2}\|_{C_{\varepsilon}^{\alpha}}\leq C_{2},
\]
 we have
\begin{align*}
 & \|\phi_{\xi_{1},\mu_{1},\omega_{\xi_{1},\mu_{1}},\tau(0)_{\xi_{1},\mu_{1}}}(\psi)-\phi_{\xi_{2},\mu_{2},\omega_{\xi_{2},\mu_{2}},\tau(0)_{\xi_{2},\mu_{2}}}(\psi)\|_{C_{\varepsilon}^{1}}\\
\leq & \|\phi_{\xi_{1},\mu_{1},\omega_{\xi_{1},\mu_{1}},\tau(0)_{\xi_{1},\mu_{1}}}(\psi)-\phi_{\xi_{2},\mu_{2},\omega_{\xi_{1},\mu_{1}},\tau(0)_{\xi_{1},\mu_{1}}}(\psi)\|_{C_{\varepsilon}^{1}}\\
 & +\|\phi_{\xi_{2},\mu_{2},\omega_{\xi_{1},\mu_{1}},\tau(0)_{\xi_{1},\mu_{1}}}(\psi)-\phi_{\xi_{2},\mu_{2},\omega_{\xi_{2},\mu_{2}},\tau(0)_{\xi_{2},\mu_{2}}}(\psi)\|_{C_{\varepsilon}^{1}}.
\end{align*}
Note that from (\ref{beta-mu estimate}) and (\ref{beta xi estimate})
we have 
\begin{align*}
 & \|\phi_{\xi_{1},\mu_{1},\omega_{\xi_{1},\mu_{1}},\tau(0)_{\xi_{1},\mu_{1}}}(\psi)-\phi_{\xi_{2},\mu_{2},\omega_{\xi_{1},\mu_{1}},\tau(0)_{\xi_{1},\mu_{1}}}(\psi)\|_{C_{\varepsilon}^{1}}\\
\leq & \int_{0}^{1}\|\frac{d}{dt}\phi_{t\xi_{1}+(1-t)\xi_{2},\mu_{1},\omega_{\xi_{1},\mu_{1}},\tau(0)_{\xi_{1},\mu_{1}}}(\psi)\|_{C_{\varepsilon}^{1}}d\psi\\
 & +\int_{0}^{1}\|\frac{d}{dt}\phi_{\xi_{2},t\mu_{1}+(1-t)\mu_{2},,\omega_{\xi_{1},\mu_{1}},\tau(0)_{\xi_{1},\mu_{1}}}(\psi)\|_{C_{\varepsilon}^{1}}d\psi\\
\leq & C\varepsilon\|\mu_{2}-\mu_{1}\|_{C_{\varepsilon}^{\alpha}}+C\|\xi_{2}-\xi_{1}\|_{C_{x_{0}}^{1}}.
\end{align*}
By comparing $\phi_{\xi_{2},\mu_{2},\omega_{\xi_{1},\mu_{1}},\tau(0)_{\xi_{1},\mu_{1}}}(\psi)$
with $\phi_{\xi_{1},\mu_{1},\omega_{\xi_{1},\mu_{1}},\tau(0)_{\xi_{1},\mu_{1}}}(\psi)$,
we have for $\phi_{\xi_{2},\mu_{2},\omega_{\xi_{1},\mu_{1}},\tau(0)_{\xi_{1},\mu_{1}}}(\psi)$
\begin{align*}
|\Delta\tau| & \leq C(\tau_{0},C_{1},C_{2},C_{3})(\varepsilon^{2}\|\mu_{2}-\mu_{1}\|_{C_{\varepsilon}^{\alpha}}+\varepsilon\|\xi_{2}-\xi_{1}\|_{C_{x_{0}}^{1}}),\\
|\zeta(\frac{L_{\Gamma}}{\varepsilon})| & \leq C(\tau_{0},C_{1},C_{2},C_{3})(\varepsilon\|\mu_{2}-\mu_{1}\|_{C_{\varepsilon}^{\alpha}}+\|\xi_{2}-\xi_{1}\|_{C_{x_{0}}^{1}}).
\end{align*}
If we denote 
\[
\left(\begin{array}{cc}
B_{11} & B_{12}\\
B_{21} & B_{22}
\end{array}\right)=\frac{\partial(\Delta\tau,\zeta(\frac{L_{\Gamma}}{\varepsilon}))}{\partial(\omega,\phi(0))}^{-1},
\]
we can get 
\begin{align*}
|B_{11}| & \leq2C_{5}^{-1}\varepsilon^{-2}\\
|B_{22}| & \leq2\exp(C(\tau_{0})(1+C_{1}+C_{2}+C_{3}))\varepsilon\\
|B_{12}| & \leq2K_{1}C_{5}^{-1}\exp(C(\tau_{0})(1+C_{1}+C_{2}+C_{3}))\\
|B_{21}| & \leq2K_{2}C_{5}^{-1}\exp(C(\tau_{0})(1+C_{1}+C_{2}+C_{3})).
\end{align*}
So we know for some $C=C(\tau_{0},C_{1},C_{2},C_{3})$ 
\[
\begin{cases}
|\omega_{\xi_{2},\mu_{2}}-\omega_{\xi_{1},\mu_{1}}| & \leq\frac{C}{\varepsilon}(\varepsilon\|\mu_{2}-\mu_{1}\|_{C_{\varepsilon}^{\alpha}}+\|\xi_{2}-\xi_{1}\|_{C_{x_{0}}^{1}})\\
|\phi(0)_{\xi_{2},\mu_{2}}-\phi(0)_{\xi_{1},\mu_{1}}| & \leq C(\varepsilon^{2}\|\mu_{2}-\mu_{1}\|_{C_{\varepsilon}^{\alpha}}+\varepsilon\|\xi_{2}-\xi_{1}\|_{C_{x_{0}}^{1}})
\end{cases}
\]
So from (\ref{beta omega estimate}) and Lemma \ref{estimate for the fundamental solution}
we have 
\begin{align*}
 & |\phi_{\xi_{2},\mu_{2},\omega_{\xi_{2},\mu_{2}},\tau(0)_{\xi_{2},\mu_{2}}}(\psi)-\phi_{\xi_{1},\mu_{1},\omega_{\xi_{1},\mu_{1}},\tau(0)_{\xi_{1},\mu_{1}}}(\psi)|\\
 & +|\zeta_{\xi_{2},\mu_{2},\omega_{\xi_{2},\mu_{2}},\tau(0)_{\xi_{2},\mu_{2}}}(\psi)-\zeta_{\xi_{1},\mu_{1},\omega_{\xi_{1},\mu_{1}},\tau(0)_{\xi_{1},\mu_{1}}}(\psi)|\\
\leq & C(\varepsilon\|\mu_{2}-\mu_{1}\|_{C_{\varepsilon}^{\alpha}}+\|\xi_{2}-\xi_{1}\|_{C_{x_{0}}^{1}}).
\end{align*}
And from (\ref{0th mode ODE})
\begin{align*}
 & \|\phi_{\xi_{2},\mu_{2},\omega_{\xi_{2},\mu_{2}},\tau(0)_{\xi_{2},\mu_{2}}}(\psi)-\phi_{\xi_{1},\mu_{1},\omega_{\xi_{1},\mu_{1}},\tau(0)_{\xi_{1},\mu_{1}}}(\psi)\|_{C_{\varepsilon}^{2,\alpha}}\\
\leq & C(\varepsilon\|\mu_{2}-\mu_{1}\|_{C_{\varepsilon}^{\alpha}}+\|\xi_{2}-\xi_{1}\|_{C_{x_{0}}^{1}}).
\end{align*}
At last note $C_{3}=C_{3}(\tau_{0},C_{1},C_{2})$. So $C(\tau_{0},C_{1},C_{2},C_{3})=C(\tau_{0},C_{1},C_{2}).$ 

\section{\label{sec:The-existence-of}The existence of CMC surfaces }

\subsection{\label{subsec:A-fixed-point}A fixed point argument}

In $0$th mode, if we prescribe $\xi\in C_{x_{0}}^{1}$, $\mu\in C_{\varepsilon}^{\alpha}$,
$\omega_{\xi,\mu}$ and $\phi(0)_{\xi,\mu}$ in $\rho$, we can get
$\phi_{\xi,\mu,\omega_{\xi,\mu},\phi(0)_{\xi,\mu}}$, which we denote
by $\phi_{\xi,\mu}$ for short. We define parameter $s$ using $\dot{\psi}=\frac{d\psi}{ds}=\phi_{\xi,\mu}^{2}+\tau_{\xi,\mu}$
where $\tau_{\xi,\mu}=-\phi_{\xi,\mu}^{2}+\frac{\phi_{\xi,\mu}}{\sqrt{1+\zeta_{\xi,\mu}^{2}}}$.
From (\ref{0th mode ODE}) we know
\[
\dot{\phi}_{\xi,\mu}^{2}+(\phi_{\xi,\mu}^{2}+\tau_{\xi,\mu})^{2}=\phi_{\xi,\mu}^{2}.
\]
And further we have 
\begin{align*}
\ddot{\phi}_{\xi,\mu} & =\phi_{\xi,\mu}-(2+\rho)\phi_{\xi,\mu}(\phi_{\xi,\mu}^{2}+\tau_{\xi,\mu}),\\
\ddot{\psi}_{\xi,\mu} & =\phi_{\xi,\mu}\dot{\phi}_{\xi,\mu}(2+\rho).
\end{align*}
Follow the calculations of mean curvature in Appendix \ref{(APP)The-calculation-of-mean curvature}
and we can get $H(\mathcal{D}_{\phi_{\xi,\mu},p_{0},\varepsilon}(w,\eta))$.
We substitute 
\[
\frac{\dot{\phi}_{\xi,\mu}\ddot{\psi}}{\phi_{\xi,\mu}}-\frac{\ddot{\phi}_{\xi,\mu}\dot{\psi}}{\phi_{\xi,\mu}}=\phi_{\xi,\mu}^{2}-\tau_{\xi,\mu}+\phi_{\xi,\mu}^{2}\rho
\]
for (\ref{one key calculation}).

\begin{Def}\label{new L Q E}

In the following, $L_{\xi,\mu}(\tilde{w},\eta)$ denotes any expression
which is linear differential operator (of order at most $2$), which
satisfies
\begin{align*}
\|L_{\xi,\mu}(\tilde{w},\eta)\|_{C_{\varepsilon}^{\alpha}}\leq & C(\|\tilde{w}\|_{C_{\varepsilon}^{2,\alpha}(SN\Gamma)}+\|\eta\|_{C_{x_{0},\varepsilon}^{2,\alpha}(\Gamma,N\Gamma)}),\\
\|L_{\xi_{1},\mu_{1}}(\tilde{w},\eta)-L_{\xi_{2},\mu_{2}}(\tilde{w},\eta)\|_{C_{\varepsilon}^{\alpha}}\leq & C(\varepsilon\|\mu_{1}-\mu_{2}\|_{C_{\varepsilon}^{\alpha}}+\|\xi_{1}-\xi_{2}\|_{C_{x_{0}}^{1}})\\
 & \times(\|\tilde{w}\|_{C_{\varepsilon}^{2,\alpha}(SN\Gamma)}+\|\eta\|_{C_{x_{0},\varepsilon}^{2,\alpha}(\Gamma,N\Gamma)}).
\end{align*}
$Q_{\xi,\mu}(\tilde{w},\eta)$ denotes any nonlinear differential
operator (of order less than or equal to $2$ ) in $\tilde{w}$ and
$\eta$ which vanishes quadratically in the pair $(\tilde{w},\eta)$
and such that 
\begin{align*}
\|Q_{\xi,\mu}(\tilde{w}_{1},\eta_{1})-Q_{\xi,\mu}(\tilde{w}_{2},\eta_{2})\|_{C_{\varepsilon}^{\alpha}}\leq & C\sup_{i=1,2}(\|\tilde{w}_{i}\|_{C_{\varepsilon}^{2,\alpha}(SN\Gamma)}+\|\eta_{i}\|_{C_{x_{0},\varepsilon}^{2,\alpha}(\Gamma,N\Gamma)})\\
 & \times(\|\tilde{w}_{1}-\tilde{w}_{2}\|_{C_{\varepsilon}^{2,\alpha}(SN\Gamma)}+\|\eta_{1}-\eta_{2}\|_{C_{x_{0},\varepsilon}^{2,\alpha}(\Gamma,N\Gamma)}),\\
\|Q_{\xi_{1},\mu_{1}}(\tilde{w},\eta)-Q_{\xi_{2},\mu_{2}}(\tilde{w},\eta)\|_{C_{\varepsilon}^{\alpha}}\leq & C(\varepsilon\|\mu_{1}-\mu_{2}\|_{C_{\varepsilon}^{\alpha}}+\|\xi_{1}-\xi_{2}\|_{C_{x_{0}}^{1}})\\
 & \times(\|\tilde{w}\|_{C_{\varepsilon}^{2,\alpha}(SN\Gamma)}^{2}+\|\eta\|_{C_{x_{0},\varepsilon}^{2,\alpha}(\Gamma,N\Gamma)}^{2}).
\end{align*}
Finally $E_{\xi,\mu}=E(\phi_{\xi,\mu},\frac{\partial\phi_{\xi,\mu}}{\partial\psi},\psi)$,
where $E$ is defined in Definition \ref{L,Q,E,def}. For $E_{\xi,\mu}$,
we have 
\[
\|E_{\xi_{1},\mu_{1}}-E_{\xi_{2},\mu_{2}}\|\leq C(\varepsilon\|\mu_{1}-\mu_{2}\|_{C_{\varepsilon}^{\alpha}}+\|\xi_{1}-\xi_{2}\|_{C_{x_{0}}^{1}}).
\]

\end{Def}

Note that each time when $\ddot{\phi},\ddot{\psi}$ appear, one term
$\rho$ comes out. Then, letting $\zeta_{\xi,\mu}=\partial_{\psi}\phi_{\xi,\mu}$,
we have, 

\begin{eqnarray}
H(\mathcal{D}_{\phi_{\xi,\mu},p_{0},\varepsilon}(\tilde{w},\eta)) & = & \frac{2}{\varepsilon}+\frac{\rho}{\varepsilon}+\frac{1}{\varepsilon}\tilde{\mathcal{L}}_{\xi,\mu}\tilde{w}+<\mathcal{J}_{\xi,\mu}\eta,\Upsilon>\nonumber \\
 &  & +\varepsilon(F_{1}(\phi_{\xi,\mu},\zeta_{\xi,\mu})\star R_{1}+F_{2}(\phi_{\xi,\mu},\zeta_{\xi,\mu})\star R_{2})\nonumber \\
 &  & +(\varepsilon^{2}+\varepsilon\rho)E_{\xi,\mu}+F_{3}(\phi_{\xi,\mu},\zeta_{\xi,\mu})\star R_{3}(\eta)\label{Perturbed mean curvature}\\
 &  & +(\varepsilon+\varepsilon^{-1}\rho)L_{\xi,\mu}(\tilde{w},\eta)+\varepsilon^{-1}(1+\rho)Q_{\xi,\mu}(\tilde{w},\eta)\nonumber 
\end{eqnarray}
 where 
\begin{eqnarray*}
\tilde{\mathcal{L}}_{\xi,\mu}\tilde{w} & = & -\frac{\dot{\psi}_{\xi,\mu}}{\phi_{\xi,\mu}^{3}}(\frac{\partial^{2}}{\partial s^{2}}+\frac{\partial^{2}}{\partial\theta^{2}})\tilde{w}-(\phi_{\xi,\mu}^{2}-\tau_{\xi,\mu})\frac{\dot{\phi}_{\xi,\mu}}{\phi_{\xi,\mu}^{4}}\frac{\partial w}{\partial s}-\frac{\dot{\psi}_{\xi,\mu}}{\phi_{\xi,\mu}^{3}}\tilde{w},\\
\mathcal{J}_{\xi,\mu}\eta & = & -\frac{\dot{\psi}_{\xi,\mu}^{3}}{\phi_{\xi,\mu}^{3}}\frac{\partial^{2}\eta}{(\partial x_{0})^{2}}-\frac{1}{\varepsilon}(\frac{\dot{\psi}_{\xi,\mu}\ddot{\psi}_{\xi,\mu}}{\phi_{\xi,\mu}^{3}}+2(\phi_{\xi,\mu}^{2}-\tau_{\xi,\mu})\frac{\dot{\phi}_{\xi,\mu}\dot{\psi}_{\xi,\mu}}{\phi_{\xi,\mu}^{4}})\frac{\partial\eta}{\partial x_{0}}\\
 &  & -\phi_{\xi,\mu}^{-2}(2\dot{\phi}_{\xi,\mu}\ddot{\psi}_{\xi,\mu}+2\frac{\dot{\phi}_{\xi,\mu}^{2}\dot{\psi}_{\xi,\mu}}{\phi_{\xi,\mu}}+\frac{\dot{\psi}_{\xi,\mu}^{3}}{\phi_{\xi,\mu}}\\
 &  & -2\frac{\dot{\psi}_{\xi,\mu}^{2}}{\phi_{\xi,\mu}}(\phi_{\xi,\mu}^{2}-\tau_{\xi,\mu})-2\phi_{\xi,\mu}\dot{\phi}_{\xi,\mu}^{2})R(\eta,X_{0})X_{0}.
\end{eqnarray*}

\begin{thm}\label{Perturbed high mode and 1st mode}
\begin{enumerate}
\item The operator 
\[
\tilde{\mathcal{L}}_{\xi,\mu}:\tilde{\Pi}C_{\varepsilon}^{2,\alpha}(SN\Gamma)\rightarrow\tilde{\Pi}C_{\varepsilon}^{0,\alpha}(SN\Gamma)
\]
 is invertible and for some uniform $C$ 
\[
\|\tilde{w}\|_{C_{\varepsilon}^{2,\alpha}}\leq C\|\tilde{\mathcal{L}}_{\xi,\mu}\tilde{w}\|_{C_{\varepsilon}^{0,\alpha}}.
\]
\item The operator
\[
\mathcal{J}_{\xi,\mu}:C_{x_{0},\varepsilon}^{2,\alpha}(\Gamma,N\Gamma)\rightarrow C_{\varepsilon}^{\alpha}(\Gamma,N\Gamma)
\]
 is invertible and for some uniform $C$ 
\begin{equation}
\|\eta\|_{C_{x_{0},\varepsilon}^{2,\alpha}}\leq C\|\mathcal{J}_{\xi,\mu}\eta\|_{C_{\varepsilon}^{\alpha}}.\label{J(xi,eta) estimate}
\end{equation}
\item If 
\[
\mathcal{J}_{\xi,\mu}\hat{\eta}=-\frac{2}{3}\varepsilon\dot{\phi}_{\xi,\mu}P_{1}(R(\Upsilon_{\theta},X_{0},\Upsilon,\Upsilon_{\theta})),
\]
we have 
\[
\|\hat{\eta}\|_{C_{x_{0},\varepsilon}^{2,\alpha}}\leq C\varepsilon^{2}.
\]
\end{enumerate}
\end{thm}

\begin{proof}For the first item, consider the bilinear functional
on the space $\tilde{\Pi}W_{\varepsilon}^{1,2}(SN\Gamma)$ 
\[
B_{\xi,\mu}(w,v)=\int_{SN\Gamma}(v(\phi_{\xi,\mu}\tilde{\mathcal{L}}_{\xi,\mu})w)d\theta d\psi.
\]
By using exactly the same argument as used in Subsection \ref{subsec:High-mode},
we can prove the results.

For the second and third item, consider $\frac{\dot{\psi}_{\xi,\mu}^{3}}{\phi_{\xi,\mu}}\mathcal{J}_{\xi,\mu}\eta$.
Let $\frac{\partial}{\partial y_{0}}=\frac{\dot{\psi}_{\xi,\mu}^{3}}{\phi_{\xi,\mu}^{2}}\frac{\partial}{\partial x_{0}}$.
From Corollary \ref{Phi-global}, we know $y_{0}(L_{\Gamma})=I_{1}'L_{\Gamma},I_{1}'=I_{1}+O(\varepsilon^{2}).$
Now we define $\tilde{y}=(I_{1}')^{-1}y_{0}$ and define $\mathfrak{F}$
and $P_{p}^{\mathfrak{F}(p)}$ as in the paragraph before Lemma \ref{eta estimate}.
\begin{align*}
\frac{\dot{\psi}_{\xi,\mu}^{3}}{\phi_{\xi,\mu}}\mathcal{J}_{\xi,\mu}\eta(q)= & -\frac{\partial^{2}\eta}{\partial y_{0}^{2}}(q)-\Psi_{1}(\phi_{\xi,\mu},\frac{\partial\phi_{\xi,\mu}}{\partial\psi})R|_{q}(\eta(q),X_{0})X_{0}+\frac{1}{\varepsilon}\rho L(\eta)\\
= & \tilde{\mathcal{J}}_{A}'\eta+(I_{2}'-\Psi_{1}|_{q})(P_{\mathfrak{F}^{-1}(q)}^{q}R)(\eta(q),X_{0})X_{0}\\
 & +\Psi_{1}|_{x_{0}}(P_{\mathfrak{F}^{-1}(q)}^{q}R-R|_{q})(\eta(q),X_{0})X_{0}+\frac{1}{\varepsilon}\rho L(\eta)\\
= & \tilde{\mathcal{J}}_{A}'\eta+(I_{2}'-\Psi_{1}|_{q})(P_{\mathfrak{F}^{-1}(q)}^{q}R)(\eta(q),X_{0})X_{0}+\varepsilon L(\eta),
\end{align*}
where $I_{2}'=I_{2}+O(\varepsilon^{2})$ and $I_{1}'^{2}I_{2}'=1$.
$I_{1}'^{2}\tilde{\mathcal{J}}_{A}'$ is conjugate to $\mathcal{J}_{A}$.
The rest of the proof is similar to Theorem \ref{estimate on fundamental solution of 1st mode}.
However, the primitive of $(I_{2}'-\Psi_{1})$ and $\frac{\dot{\psi}_{\xi,\mu}^{3}}{\phi_{\xi,\mu}}\dot{\phi}_{\xi,\mu}$
may not be global smooth function on the geodesic. Nevertheless, from
Corollary \ref{F(phi)-F(phi0)}, we know 
\begin{align*}
\chi(y_{1}) & =\int_{0}^{y_{1}}(I_{2}'-\Psi_{1})dy_{0}=O(\varepsilon),\\
\chi(L_{\Gamma}) & =O(\varepsilon^{2}).
\end{align*}
From the proof of Lemma \ref{better estimate in 1st mode}, by noticing
\[
I_{2}'-\Psi_{1}=(I_{2}'-\Psi_{1}-\frac{\chi(L_{\Gamma})}{L_{\Gamma}})+\frac{\chi(L_{\Gamma})}{L_{\Gamma}},
\]
 we can prove the second item. For $\frac{\dot{\psi}_{\xi,\mu}^{3}}{\phi_{\xi,\mu}}\dot{\phi}_{\xi,\mu}$
we can do the same thing. So all the argument in Theorem \ref{estimate on fundamental solution of 1st mode}
works here and we can prove this theorem.

\end{proof}

Consider
\begin{equation}
\begin{cases}
\frac{1}{\varepsilon}\tilde{\mathcal{L}}_{\xi,\mu}\tilde{w}= & -\Pi_{0}(\varepsilon(F_{1}(\phi_{\xi,\mu},\zeta_{\xi,\mu})\star R_{1}+F_{2}(\phi_{\xi,\mu},\zeta_{\xi,\mu})\star R_{2})+(\varepsilon^{2}+\varepsilon\rho)E_{\xi,\mu}\\
 & +F_{3}(\phi_{\xi,\mu},\zeta_{\xi,\mu})\star R_{3}(\eta)+(\varepsilon+\varepsilon^{-1}\rho)L_{\xi,\mu}(\tilde{w},\eta)+\varepsilon^{-1}(1+\rho)Q_{\xi,\mu}(\tilde{w},\eta)),\\
\mathcal{J}_{\xi,\mu}\eta= & -P_{1}(\varepsilon F_{2}(\phi_{\xi,\mu},\zeta_{\xi,\mu})\star R_{2})+(\varepsilon^{2}+\varepsilon\rho)E_{\xi,\mu}+(\varepsilon+\varepsilon^{-1}\rho)L_{\xi,\mu}(\tilde{w},\eta)\\
 & +\varepsilon^{-1}(1+\rho)Q_{\xi,\mu}(\tilde{w},\eta)).
\end{cases}\label{perturbed high and 1st mode}
\end{equation}
From Theorem \ref{Perturbed high mode and 1st mode} and the properties
of $L_{\xi,\mu},Q_{\xi,\mu},E_{\xi,\mu}$ in Definition \ref{new L Q E},
we can solve this system and get $(\tilde{w}_{\xi,\mu},\eta_{\xi,\mu})$
which satisfies
\[
\|\tilde{w}_{\xi,\mu}\|_{C_{\varepsilon}^{2,\alpha}}+\|\eta_{\xi,\mu}\|_{C_{x_{0},\varepsilon}^{2,\alpha}}\leq C\varepsilon^{2}.
\]

Moreover, from Theorem \ref{0th mode nonlinear ode solution}, Theorem
\ref{Perturbed high mode and 1st mode}, Definition \ref{new L Q E},
we can prove

\begin{lem}\label{Delta mu delta xi influence on w and eta}
\[
\begin{cases}
\|\tilde{w}_{\xi+\Delta\xi,\mu+\Delta\mu}-\tilde{w}_{\xi,\mu}\|_{C_{\varepsilon}^{2,\alpha}} & \leq C\varepsilon^{2}(\varepsilon\|\Delta\mu\|_{C_{\varepsilon}^{\alpha}}+\|\Delta\xi\|_{C_{x_{0}}^{1}}),\\
\|\eta_{\xi+\Delta\xi,\mu+\Delta\mu}-\eta_{\xi,\mu}\|_{C_{x_{0},\varepsilon}^{2,\alpha}} & \leq C\varepsilon(\varepsilon\|\Delta\mu\|_{C_{\varepsilon}^{\alpha}}+\|\Delta\xi\|_{C_{x_{0}}^{1}}).
\end{cases}
\]

\end{lem}

One can calculate
\begin{eqnarray*}
 &  & H(\mathcal{D}_{\phi_{\xi,\mu},p_{0},\varepsilon}(\tilde{w}_{\xi,\mu},\eta_{\xi,\mu}))-\frac{2}{\varepsilon}=\Pi_{0}(H(\mathcal{D}_{\phi_{\xi,\mu},p_{0},\varepsilon}(\tilde{w}_{\xi,\mu},\eta_{\xi,\mu}))-\frac{2}{\varepsilon})\\
 & = & F_{4}(\phi_{\xi,\mu},\zeta_{\xi,\mu})(\Pi_{0}(R(\Upsilon,X_{0},\eta_{\xi,\mu},\Upsilon))+\xi(x_{0}))+\varepsilon^{2}\mu+\Pi_{0}((\varepsilon^{2}+\varepsilon\rho)E_{\xi,\mu}\\
 &  & +(\varepsilon+\varepsilon^{-1}\rho)L_{\xi,\mu}(\tilde{w},\eta)+\varepsilon^{-1}(1+\rho)Q_{\xi,\mu}(\tilde{w},\eta))+\varepsilon^{2}\omega_{\xi,\mu}\phi_{\xi,\mu}^{-1}\frac{d\phi_{\xi,\mu}}{d\psi}.
\end{eqnarray*}

Now we define a map 
\begin{align*}
\Omega:C_{x_{0}}^{1}\times C_{\varepsilon}^{\alpha} & \rightarrow C_{x_{0}}^{1}\times C_{\varepsilon}^{\alpha}\\
(\xi,\mu) & \mapsto(\Omega^{1}(\xi,\mu),\Omega^{2}(\xi,\mu))
\end{align*}
where
\begin{equation}
\begin{cases}
\Omega^{1}(\xi,\mu)= & -\Pi_{0}(R(\Upsilon,X_{0},\eta_{\xi,\mu},\Upsilon))\\
\Omega^{2}(\xi,\mu)= & -\varepsilon^{-2}\Pi_{0}((\varepsilon^{2}+\varepsilon\rho)E_{\xi,\mu}+(\varepsilon+\varepsilon^{-1}\rho)L_{\xi,\mu}(\tilde{w},\eta)\\
 & +\varepsilon^{-1}(1+\rho)Q_{\xi,\mu}(\tilde{w},\eta)).
\end{cases}\label{Omega expression}
\end{equation}

\begin{lem}\label{fixed point argument}For fixed $C_{1},C_{2}$
if $\|\xi\|_{C_{x_{0}}^{1}}\leq C_{1}\varepsilon^{2},\|\mu\|_{C_{\varepsilon}^{\alpha}}\leq C_{2}$,
there is $C>0$ which does not depend on $\varepsilon$ such that
\begin{eqnarray}
 &  & \|(\Omega^{1}(\xi_{1},\mu_{1}),\Omega^{2}(\xi_{1},\mu_{1}))-(\Omega^{1}(\xi_{2},\mu_{2}),\Omega^{2}(\xi_{2},\mu_{2}))\|_{\varepsilon,\alpha}\nonumber \\
 & \leq & C\varepsilon\|(\xi_{1}-\xi_{2},\mu_{1}-\mu_{2})\|_{\varepsilon,\alpha}\label{Omega estimate}
\end{eqnarray}
where the $\|(\cdot,\cdot)\|_{\varepsilon,\alpha}$ norm is defined
in Subsection \ref{2.1.3}.

\end{lem}

\begin{proof}

\begin{align*}
 & \|(\Omega^{1}(\xi_{1},\mu_{1}),\Omega^{2}(\xi_{1},\mu_{1}))-(\Omega^{1}(\xi_{2},\mu_{2}),\Omega^{2}(\xi_{2},\mu_{2}))\|_{\varepsilon,\alpha}\\
\leq & \|\Omega^{1}(\xi_{1},\mu_{1})-\Omega^{1}(\xi_{2},\mu_{2})\|_{C_{x_{0}}^{1}}+\varepsilon\|\Omega^{2}(\xi_{1},\mu_{1})-\Omega^{2}(\xi_{2},\mu_{2})\|_{C_{\varepsilon}^{\alpha}}.
\end{align*}
From Lemma \ref{Delta mu delta xi influence on w and eta}, 
\begin{align*}
\|\Omega^{1}(\xi_{1},\mu_{1})-\Omega^{1}(\xi_{2},\mu_{2})\|_{C_{x_{0}}^{1}} & \leq C\|\eta_{\xi_{1},\mu_{1}}-\eta_{\xi_{2},\mu_{2}}\|_{C_{x_{0}}^{1}}\\
 & \leq C\varepsilon\|(\xi_{1}-\xi_{2},\mu_{1}-\mu_{2})\|_{\varepsilon,\alpha}.
\end{align*}
Notice that $\rho=\rho_{\xi,\mu}=O(\varepsilon^{2})$. From Definition
\ref{new L Q E} and Lemma \ref{Delta mu delta xi influence on w and eta},
we have 
\begin{align*}
\|\Omega^{2}(\xi_{1},\mu_{1})-\Omega^{2}(\xi_{2},\mu_{2})\|_{C_{\varepsilon}^{\alpha}}\leq & C\varepsilon^{-2}(\varepsilon^{2}|E_{\xi_{1},\mu_{1}}-E_{\xi_{2},\mu_{2}}|\\
 & +\varepsilon|L_{\xi_{1},\mu_{1}}(\tilde{w}_{\xi_{1},\mu_{1}},\eta_{\xi_{1},\mu_{1}})-L_{\xi_{2},\mu_{2}}(\tilde{w}_{\xi_{2},\mu_{2}},\eta_{\xi_{2},\mu_{2}})|\\
 & +\varepsilon^{-1}|Q_{\xi_{1},\mu_{1}}(\tilde{w}_{\xi_{1},\mu_{1}},\eta_{\xi_{1},\mu_{1}})-Q_{\xi_{2},\mu_{2}}(\tilde{w}_{\xi_{2},\mu_{2}},\eta_{\xi_{2},\mu_{2}})|\\
 & +O(\varepsilon)|\rho_{\xi_{1},\mu_{1}}-\rho_{\xi_{2},\mu_{2}}|)\\
\leq & C\|(\xi_{1}-\xi_{2},\mu_{1}-\mu_{2})\|_{\varepsilon,\alpha}.
\end{align*}
So we proved this lemma. 

\end{proof}

If $\xi^{0}=0,\mu^{0}=0$, we can get $\tilde{w}_{0,0},\eta_{0,0},$
whose norms only depend on the curvature terms along the geodesic.
So there is $C_{7}$ which only depends on the norms of the curvatures
along the geodesic, such that
\begin{eqnarray*}
\|\Omega^{1}(0,0)\|_{C_{x_{0}}^{1}} & \leq & C_{7}\varepsilon^{2}\\
\varepsilon\|\Omega^{2}(0,0)\|_{C_{\varepsilon}^{\alpha}} & \leq & C_{7}\varepsilon
\end{eqnarray*}
and if we assume 
\begin{eqnarray*}
\xi^{1} & = & \Omega^{1}(0,0)\\
\mu^{1} & = & \Omega^{2}(0,0),
\end{eqnarray*}
from (\ref{Omega estimate}) 
\begin{eqnarray*}
\|\Omega^{1}(\xi^{1},\mu^{1})-\xi^{1}\|_{C_{x_{0}}^{1}} & \leq & C_{7}\varepsilon^{2}\\
\varepsilon\|\Omega^{2}(\xi^{1},\mu^{1})-\mu^{1}\|_{C_{\varepsilon}^{\alpha}} & \leq & C_{7}\varepsilon^{2},
\end{eqnarray*}
where we can use the same constant $C_{7}.$ Then 
\[
\|\Omega(\xi^{1},\mu^{1})-(\xi^{1},\mu^{1})\|_{\varepsilon,\alpha}\leq2C_{7}\varepsilon^{2}.
\]

Let 
\[
\Xi(5C_{7})=\{(\xi,\mu):\|(\xi,\mu)-(\xi^{1},\mu^{1})\|_{\varepsilon,\alpha}\leq5C_{7}\varepsilon^{2}\}
\]
We assume $C_{1}=C_{2}=10C_{7}$. Let 
\[
\Xi(C_{1},C_{2})=\{(\xi,\mu):\|\xi\|_{C_{x_{0}}^{1}}\leq C_{1}\varepsilon^{2},\|\mu\|_{C_{\varepsilon}^{\alpha}}\leq C_{2}\}.
\]
It is obvious that 
\[
\Xi(5C_{7})\subset\Xi(C_{1},C_{2}).
\]
So if $(\xi_{1},\mu_{1}),(\xi_{2},\mu_{2})\in\Xi(5C_{7})$, then
\[
\|\Omega(\xi_{1},\mu_{1})-\Omega(\xi_{2},\mu_{2})\|_{\varepsilon,\alpha}\leq C\varepsilon\|(\xi_{1}-\xi_{2},\mu_{1}-\mu_{2})\|_{\varepsilon,\alpha}.
\]
If we choose $\varepsilon$ such that 
\[
C\varepsilon\leq\frac{1}{100},
\]
then $\Omega_{2}$ maps $\Xi(5C_{7})$ into itself. Note that $\Xi(5C_{7})$
is a complete metric space. From fixed point theorem, there is a unique
\[
(\hat{\xi},\hat{\mu})\in\Xi(5C_{7})
\]
 such that 
\[
\Omega(\hat{\xi},\hat{\mu})=(\hat{\xi},\hat{\mu}).
\]

For this $(\hat{\xi},\hat{\mu}),$ we have 
\[
H(\mathcal{D}_{\phi_{\hat{\xi},\hat{\mu}},p_{0},\varepsilon}(\tilde{w}_{\hat{\xi},\hat{\mu}},\eta_{\hat{\xi},\hat{\mu}}))=\frac{2}{\varepsilon}+\varepsilon^{2}\omega_{\hat{\xi},\hat{\mu}}\phi_{\hat{\xi},\hat{\mu}}^{-1}\frac{\partial\phi_{\hat{\xi},\hat{\mu}}}{\partial\psi}.
\]

\subsection{The energy of the surface\label{subsec:The-energy-of} and the last
step of the proof. }

So far what we've got is a global smooth surface $\mathcal{D}_{\phi_{\hat{\xi},\hat{\mu}},p_{0},\varepsilon}(\tilde{w}_{\hat{\xi},\hat{\mu}},\eta_{\hat{\xi},\hat{\mu}})$
whose mean curvature is 
\[
\frac{2}{\varepsilon}+\omega_{\hat{\xi},\hat{\mu}}\varepsilon^{2}\phi_{\hat{\xi},\hat{\mu}}^{-1}\frac{\partial\phi_{\hat{\xi},\hat{\mu}}}{\partial\psi}.
\]
In (\ref{0th mode ODE}), we have chosen a point $p_{0}$ where $\psi=0$
(and also $x_{0}=0$), i.e. $\phi(\psi)$ attains a local minimum
at $p_{0}.$ We call $p_{0}$ the starting point. From the analysis
above, we find that $\hat{\xi},\hat{\mu}$ and $\omega_{\hat{\xi},\hat{\mu}}$
only depend on $p_{0}$. The last thing we can do is to move the starting
point along the geodesic such that $\omega_{\hat{\xi},\hat{\mu}}=0.$
View $\psi\in(-\varepsilon_{0},\varepsilon_{0})$ as local coordinate
about $p_{0}$ and $\psi(p_{0})=0.$ Now choose $\psi=\delta$ instead
of $\psi=0$ as the starting point. We can do all the analysis above
and get $(\hat{\xi}_{\delta},\hat{\mu}_{\delta},\omega_{\delta}=\omega_{\hat{\xi}_{\delta},\hat{\mu}_{\delta}},\phi_{\delta}(\delta)=\phi_{\hat{\xi}_{\delta},\hat{\mu}_{\delta}}(\delta),\phi_{\delta}=\phi_{\hat{\xi}_{\delta},\hat{\mu}_{\delta}},\tilde{w}_{\delta},\eta_{\delta})$
which satisfy similar estimates as $\delta=0$ case and 
\[
H(\mathcal{D}_{\phi_{\delta},\delta,\varepsilon}(\tilde{w}_{\delta},\eta_{\delta}))=\frac{2}{\varepsilon}+\omega_{\delta}\varepsilon^{2}\phi_{\delta}^{-1}\frac{\partial\phi_{\delta}}{\partial\psi}.
\]
 We make the following notations 
\begin{eqnarray*}
\frac{\partial f_{\delta}(\psi)}{\partial\delta}|_{\delta=0} & = & \lim_{\delta\rightarrow0}\frac{f_{\delta}(\psi)-f_{0}(\psi)}{\delta}\\
\frac{\partial^{\prime}f_{\delta}(\psi)}{\partial\delta}|_{\delta=0} & = & \lim_{\delta\rightarrow0}\frac{f_{\delta}(\psi)-f_{0}(\psi-\delta)}{\delta}.
\end{eqnarray*}

We can take a new angle of view. We identify different starting points
and imagine that the curvature terms are translated along the geodeisc.
We know
\begin{equation}
\frac{\partial'}{\partial\delta}R_{i}=\varepsilon\bar{R}_{i}\label{delta prime deirvative of R}
\end{equation}
and $\bar{R}_{i}$'s derivatives with respect to $x_{0}$ are bounded.
Besides, $\bar{R}_{i}$ belongs to the same subspace (range of $\Pi_{0},\Pi_{1},\tilde{\Pi}$)
as $R_{i}$ does. So if we revise each mode, we know that $\phi_{\delta},\tilde{w}_{\delta},\eta_{\delta}$
vary smoothly in $\delta.$ We want to study $\frac{\partial^{\prime}\phi_{\delta}}{\partial\delta},\frac{\partial^{\prime}\tilde{w}_{\delta}}{\partial\delta},\frac{\partial^{\prime}\eta_{\delta}}{\partial\delta}$. 

Now we start with the initial surface $\mathcal{D}_{\phi_{\hat{\xi},\hat{\mu}},p_{0},\varepsilon}(\tilde{w}_{\hat{\xi},\hat{\mu}},\eta_{\hat{\xi},\hat{\mu}}).$
$\phi_{\hat{\xi},\hat{\mu}}$ satisfies (\ref{0th mode ODE}) with
$\xi=\hat{\xi},\mu=\hat{\mu},\phi(0)=\phi(0)_{\hat{\xi},\hat{\mu}},\omega=\omega_{\hat{\xi},\hat{\mu}}$.
Heuristically when $\delta$ has a variation of size $1$, the $C_{x_{0}}^{1}$
norm of $R_{i}$ will has a variation of $O(\varepsilon)$. At first
we fix $\xi=\hat{\xi}$ and $\mu=\hat{\mu}$. By analyzing both nonlinear
ODE and its linearized equations, we find $\omega$ should be perturbed
as large as $O(\varepsilon)$ and $\phi(0)$ should be perturbed as
large as $O(\varepsilon^{3})$ to match the boundary value. Hence
we know $\|\phi_{\delta}(\psi+\delta)-\phi_{\hat{\xi},\hat{\mu}}(\psi)\|_{C_{\varepsilon}^{2,\alpha}}=O(\varepsilon^{2}).$
From Theorem \ref{Perturbed high mode and 1st mode}, (\ref{perturbed high and 1st mode})and
Definition \ref{new L Q E} we have $\|\tilde{w}_{\delta}(\psi+\delta)-\tilde{w}(\psi)\|_{C_{\varepsilon}^{2,\alpha}}\leq C\varepsilon^{3},\|\eta_{\delta}(\psi+\delta)-\eta(\psi)\|_{C_{x_{0},\varepsilon}^{2,\alpha}}\leq C\varepsilon^{3}.$
Now we replace $(\hat{\xi},\hat{\mu})$ with $(\Omega^{1}(\hat{\xi},\hat{\mu}),\Omega^{2}(\hat{\xi},\hat{\mu})).$
From (\ref{Omega expression}), $\hat{\xi}$ has a $O(\varepsilon^{3})$
variation and $\hat{\mu}$ has a variation of $O(\varepsilon)$ which
implies an $\varepsilon^{2}$ variation again on $\phi_{\delta}$
by (\ref{main estimates for 0th mode-1}) and hence a $O(\varepsilon^{4})$
variation on $\tilde{w}_{\delta}$ and a $O(\varepsilon^{3})$ variation
on $\eta_{\delta}$. However, in the second step of iteration $\hat{\mu}$
only has a variation of $O(\varepsilon^{2})$. So the iteration argument
works. 

\begin{tabular}{|c|c|c|c|c|c|c|}
\hline 
 & 1 & 2 & 3 & $\cdots$ & k & $\cdots$\tabularnewline
\hline 
$\begin{array}{c}
\omega_{\delta}\\
|\cdot|
\end{array}$ & $\varepsilon$ & $\varepsilon$ & $\varepsilon^{2}$ & $\cdots$ & $\varepsilon^{k-1}$ & $\cdots$\tabularnewline
\hline 
$\begin{array}{c}
\phi_{\delta}(\delta)\\
|\cdot|
\end{array}$ & $\varepsilon^{3}$ & $\varepsilon^{3}$ & $\varepsilon^{4}$ & $\cdots$ & $\varepsilon^{k+1}$ & $\cdots$\tabularnewline
\hline 
$\begin{array}{c}
\phi_{\delta}(\cdot+\delta)\\
\|\cdot\|_{C_{\varepsilon}^{2,\alpha}}
\end{array}$ & $\varepsilon^{2}$ & \multicolumn{1}{c|}{$\varepsilon^{2}$} & $\varepsilon^{3}$ & $\cdots$ & $\varepsilon^{k}$ & $\cdots$\tabularnewline
\hline 
$\begin{array}{c}
\tilde{w}_{\delta}(\cdot+\delta)\\
\|\cdot\|_{C_{\varepsilon}^{2,\alpha}}
\end{array}$ & $\varepsilon^{3}$ & $\varepsilon^{4}$ & $\varepsilon^{5}$ & $\cdots$ & $\varepsilon^{k+2}$ & $\cdots$\tabularnewline
\hline 
$\begin{array}{c}
\eta_{\delta}(\cdot+\delta)\\
\|\cdot\|_{C_{x_{0},\varepsilon}^{2,\alpha}}
\end{array}$ & $\varepsilon^{3}$ & $\varepsilon^{3}$ & $\varepsilon^{4}$ & $\cdots$ & $\varepsilon^{k+1}$ & $\cdots$\tabularnewline
\hline 
$\begin{array}{c}
\xi_{\delta}(\cdot+\delta)\\
\|\cdot\|_{C_{x_{0}}^{1}}
\end{array}$ & $\varepsilon^{3}$ & $\varepsilon^{3}$ & $\varepsilon^{4}$ & $\cdots$ & $\varepsilon^{k+1}$ & $\cdots$\tabularnewline
\hline 
$\begin{array}{c}
\mu_{\delta}(\cdot+\delta)\\
\|\cdot\|_{C_{\varepsilon}^{\alpha}}
\end{array}$ & $\varepsilon$ & $\varepsilon^{2}$ & $\varepsilon^{3}$ & $\cdots$ & $\varepsilon^{k}$ & $\cdots$\tabularnewline
\hline 
\end{tabular}

In the above form, for example, the $\varepsilon^{k+1}$ in $(\eta_{\delta}(\cdot+\delta),k)$
position means in the $k$th step of iteration, $\eta_{\delta}(\cdot+\delta)$
has a variation of $\varepsilon^{k+1}$ measured in $\|\cdot\|_{C_{x_{0},\varepsilon}^{2,\alpha}}$
norm.

The above argument can be made precise by taking $\frac{\partial'}{\partial\delta}$
derivative to each mode as well as the expression (\ref{Omega expression}). 

At last we get
\begin{align*}
\|\frac{\partial'}{\partial\delta}\phi_{\delta}|_{\delta=0}\|_{C_{\varepsilon}^{2,\alpha}} & \leq C\varepsilon^{2},\\
\|\frac{\partial'}{\partial\delta}\eta_{\delta}|_{\delta=0}\|_{C_{x_{0},\varepsilon}^{2,\alpha}} & \leq C\varepsilon^{3},\\
\|\frac{\partial'}{\partial\delta}\tilde{w}_{\delta}|_{\delta=0}\|_{C_{\varepsilon}^{2,\alpha}} & \leq C\varepsilon^{3}.
\end{align*}
So we have

\begin{align*}
 & \frac{\partial}{\partial\delta}(\varepsilon(\phi_{\delta}+\tilde{w}_{\delta})+<\eta_{\delta},\Upsilon>)|_{\delta=0}\\
= & \frac{\partial'}{\partial\delta}(\varepsilon(\phi_{\delta}+\tilde{w}_{\delta})+<\eta_{\delta},\Upsilon>)|_{\delta=0}-\frac{\partial}{\partial\psi}(\varepsilon(\phi_{\hat{\xi},\hat{\mu}}+\tilde{w}_{\hat{\xi},\hat{\mu}})+<\eta_{\hat{\xi},\hat{\mu}},\Upsilon>)|_{\psi=0}\\
= & -\varepsilon\frac{\partial\phi_{\hat{\xi},\hat{\mu}}}{\partial\psi}+O(\varepsilon^{3}).
\end{align*}

Consider the energy functional of the surface $\Sigma_{\delta}=\mathcal{D}_{\phi_{\delta},\delta,\varepsilon}(\tilde{w}_{\delta},\eta_{\delta})$
\begin{align*}
EN(\Sigma_{\delta}) & ={\rm Area}(\Sigma_{\delta})-\frac{2}{\varepsilon}{\rm Vol}(\Sigma_{\delta}).
\end{align*}
We have
\begin{eqnarray*}
 &  & \frac{d}{d\delta}EN(\Sigma_{\delta})|_{\delta=0}\\
 & = & \frac{d}{d\delta}(\text{Area(\ensuremath{\Sigma_{\delta}})-\ensuremath{\frac{2}{\varepsilon}}Vol(\ensuremath{\Sigma_{\delta}})})\\
 & = & \int_{\Sigma_{0}}H\frac{\partial}{\partial\delta}(\varepsilon(\phi_{\delta}(\psi)+\tilde{w}_{\delta})+<\eta_{\delta},\Upsilon>)|_{\delta=0}<N,\Upsilon>dS\\
 &  & -\frac{2}{\varepsilon}\int_{\Sigma_{0}}\frac{\partial}{\partial\delta}(\varepsilon(\phi_{\delta}(\psi)+\tilde{w}_{\delta})+<\eta_{\delta},\Upsilon>)|_{\delta=0}<N,\Upsilon>dS\\
 & = & -\varepsilon^{3}\omega_{\hat{\xi},\hat{\mu}}\int_{\Sigma_{0}}\phi_{\hat{\xi},\hat{\mu}}^{-1}[(\frac{\partial\phi_{\hat{\xi},\hat{\mu}}}{\partial\psi})^{2}+O(\varepsilon^{2})]<N,\Upsilon>dS,
\end{eqnarray*}
where $\Sigma_{0}=\mathcal{D}_{\phi_{\hat{\xi},\hat{\mu}},p_{0},\varepsilon}(\tilde{w}_{\hat{\xi},\hat{\mu}},\eta_{\hat{\xi},\hat{\mu}}).$
Note that 
\[
\int_{\Sigma_{0}}\phi_{\hat{\xi},\hat{\mu}}^{-1}(\frac{\partial\phi_{\hat{\xi},\hat{\mu}}}{\partial\psi})^{2}<N,\Upsilon>dS
\]
is always positive.

Similarly for $\delta\in[0,\frac{L_{\Gamma}}{\varepsilon}]$, we can
prove that 
\begin{align*}
 & \frac{d}{d\delta}EN(\Sigma_{\delta})|_{\delta=0}\\
= & -\varepsilon^{3}\omega_{\delta}\int_{\Sigma_{\delta}}\phi_{\delta}^{-1}[(\frac{\partial}{\partial\psi}\phi_{\delta})^{2}+O(\varepsilon^{2})]<N,\Upsilon>dS
\end{align*}
with the integral being always positive. If $E$ is constant when
$\delta\in[0,\frac{L_{\Gamma}}{\varepsilon}]$, we have for every
$\delta$, $\omega_{\delta}=0.$ Then we have infinitely many Delaunay
type constant mean curvature surfaces. If $E$ is not always constant,
we will at least get two zeros of $\frac{d}{d\delta}E(\Sigma_{\delta}),$
where we have $\omega_{\delta}=0.$ Then we get two Delaunay type
constant mean curvature surfaces. The two surfaces are not the same,
because they correspond to the maximal value and minimal value of
$E.$ That the Delaunay type CMC surfaces are embedded is evident.
First the CMC surfaces constructed can be viewed as a smooth map from
$T^{2}$ to $M,$ because the unit normal bundle of $\Gamma$ has
$T^{2}$ topology. And because $T^{2}$ is compact topological space,
we need only to prove that this map is injective. This is easily seen
from the fact that $\Gamma$ is simply closed and embedded and the
estimates of the functions in each mode. So we proved the main theorem.

For Corollary \ref{rotational symmetry}, we see that the non-degeneracy
condition of the Jacobi operator of the geodesic is only used in 1st
mode. However, when the metric around $\Gamma$ has rotational symmetry,
i.e. $\frac{\partial}{\partial\theta}$ is killing vector field, if
we look at the expression of mean curvature (\ref{eq:mean curvature expression}),
we will find that $R_{1},R_{2}$ and $E$ have no 1st mode or high
mode projections. So we may assume $\eta=0$ and $\tilde{w}=0$. The
only thing that we need to do is to solve the 0th mode. So we can
use the procedure in Subsection \ref{subsec:0th-mode} to solve 0th
mode up to the kernel $\frac{\partial\phi}{\partial\psi}$. And then
we use the same argument as in this Subsection to remove the kernel.
So we can get the CMC surfaces of Delaunay type.

\appendix

\section{The calculation of the mean curvature\label{(APP)The-calculation-of-mean curvature}}

We prove (\ref{eq:mean curvature expression}) here. By definition
we have

\begin{eqnarray*}
H(\mathcal{D}_{\phi_{\tau_{0}},p_{0},\varepsilon}(w,\eta)) & = & g^{ss}<N,\nabla_{\partial_{s}}\partial_{s}>+2g^{s\theta}<N,\nabla_{\partial_{\theta}}\partial_{s}>+g^{\theta\theta}<N,\nabla_{\partial_{\theta}}\partial_{\theta}>\\
 & = & \frac{1}{k}(g^{ss}<kN,\nabla_{\partial_{s}}\partial_{s}>+2g^{s\theta}<kN,\nabla_{\partial_{\theta}}\partial_{s}>\\
 &  & +g^{\theta\theta}<kN,\nabla_{\partial_{\theta}}\partial_{\theta}>).
\end{eqnarray*}

First, we would like to point out that in the following, terms like
$\eta,\frac{\partial\eta}{\partial x_{0}}$ and $\varepsilon\frac{\partial^{2}\eta}{\partial x_{0}^{2}}$
are collected in $L(w,\eta).$ It is also similar for $Q(w,\eta).$
This is why in Definition \ref{L,Q,E,def}, we use $C_{x_{0},\varepsilon}^{2,\alpha}$
norm of $\eta.$
\begin{eqnarray*}
g^{ss}<kN,\nabla_{\partial_{s}}\partial_{s}> & = & g^{ss}<N_{0}+a_{1}\partial_{s}+a_{2}\partial_{\theta},\nabla_{\partial_{s}}\partial_{s}>
\end{eqnarray*}
\begin{eqnarray}
<\partial_{s},\nabla_{\partial_{s}}\partial_{s}> & = & \frac{1}{2}\partial_{s}<\partial_{s},\partial_{s}>\nonumber \\
 & = & \frac{\varepsilon^{2}}{2}\partial_{s}(\phi^{2}+\varepsilon^{2}\phi^{2}\dot{\psi}^{2}R(\Upsilon,X_{0},\Upsilon,X_{0})+2\varepsilon\phi\dot{\psi}^{2}R(\Upsilon,X_{0},\eta,X_{0})\nonumber \\
 &  & +\frac{4}{3}\varepsilon\phi\dot{\phi}\dot{\psi}R(\Upsilon,X_{0},\eta,\Upsilon)+2\dot{\phi}\frac{\partial w}{\partial s}+2\dot{\phi}\dot{\psi}<\Upsilon,\frac{\partial\eta}{\partial x_{0}}>_{e}\nonumber \\
 &  & +\varepsilon^{2}L(w,\eta)+Q(w,\eta)+O(\varepsilon^{3})),\nonumber \\
 & = & \frac{\varepsilon^{2}}{2}(2\phi\dot{\phi}+L(w,\eta)+Q(w,\eta)+O(\varepsilon^{2})),\label{eq:partial(s)-for-partial(s)}
\end{eqnarray}
\begin{eqnarray}
<\partial_{\theta},\nabla_{\partial_{s}}\partial_{s}> & = & \partial_{s}<\partial_{\theta},\partial_{s}>-<\nabla_{\partial_{s}}\partial_{\theta},\partial_{s}>\nonumber \\
 & = & \partial_{s}<\partial_{\theta},\partial_{s}>-\frac{1}{2}\partial_{\theta}<\partial_{s},\partial_{s}>\nonumber \\
 & = & \varepsilon^{2}\partial_{s}(\frac{2}{3}\varepsilon^{2}\phi^{3}\dot{\psi}R(\Upsilon,X_{0},\Upsilon,\Upsilon_{\theta})+\frac{2}{3}\varepsilon\phi^{2}\dot{\psi}(R(\Upsilon,X_{0},\eta,\Upsilon_{\theta})\nonumber \\
 &  & +R(\eta,X_{0},\Upsilon,\Upsilon_{\theta}))+\frac{1}{3}\varepsilon\phi^{2}\dot{\phi}R(\eta,\Upsilon,\Upsilon,\Upsilon_{\theta})+\dot{\phi}\frac{\partial w}{\partial\theta}\nonumber \\
 &  & +\phi\dot{\psi}<\frac{\partial\eta}{\partial x_{0}},\Upsilon_{\theta}>_{e}+\varepsilon^{2}L(w,\eta)+Q(w,\eta)+O(\varepsilon^{3}))\nonumber \\
 &  & -\frac{\varepsilon^{2}}{2}\partial_{\theta}(\phi^{2}+\varepsilon^{2}\phi^{2}\dot{\psi}^{2}R(\Upsilon,X_{0},\Upsilon,X_{0})+2\varepsilon\phi\dot{\psi}^{2}R(\Upsilon,X_{0},\eta,X_{0})\nonumber \\
 &  & +\frac{4}{3}\varepsilon\phi\dot{\phi}\dot{\psi}R(\Upsilon,X_{0},\eta,\Upsilon)+2\dot{\phi}\frac{\partial w}{\partial s}+2\dot{\phi}\dot{\psi}<\Upsilon,\frac{\partial\eta}{\partial x^{0}}>_{e}\nonumber \\
 &  & +\varepsilon^{2}L(w,\eta)+Q(w,\eta)+O(\varepsilon^{3})),\nonumber \\
 & = & \varepsilon^{2}(O(\varepsilon^{2})+L(w,\eta)+Q(w,\eta)).\label{eq:partial(theta)-for-partial(s)}
\end{eqnarray}
\begin{eqnarray*}
\nabla_{\partial_{s}}\partial_{s} & = & \varepsilon\nabla_{\partial_{s}}(\dot{\psi}X_{0}+(\dot{\phi}+\frac{\partial w}{\partial s})\Upsilon+\dot{\psi}\frac{\partial\eta^{i}}{\partial x_{0}}X_{i})\\
 & = & \varepsilon(\ddot{\psi}X_{0}+(\ddot{\phi}+\frac{\partial^{2}w}{\partial s^{2}})\Upsilon+(\ddot{\psi}\frac{\partial\eta^{i}}{\partial x^{0}}+\dot{\psi}^{2}\varepsilon\frac{\partial^{2}\eta^{i}}{\partial x_{0}^{2}})X_{i})\\
 &  & +\varepsilon^{2}(\dot{\psi}^{2}\nabla_{X_{0}}X_{0}+(\dot{\phi}+\frac{\partial w}{\partial s})^{2}\nabla_{\Upsilon}\Upsilon+\dot{\psi}^{2}(\frac{\partial\eta^{i}}{\partial x_{0}})(\frac{\partial\eta^{j}}{\partial x_{0}})\nabla_{X_{i}}X_{j}\\
 &  & +2\dot{\psi}(\dot{\phi}+\frac{\partial w}{\partial s})\nabla_{X_{0}}\Upsilon+2\dot{\psi}^{2}\frac{\partial\eta^{i}}{\partial x_{0}}\nabla_{X_{0}}X_{i}+2(\dot{\phi}+\frac{\partial w}{\partial s})\dot{\psi}\frac{\partial\eta^{i}}{\partial x_{0}}\nabla_{\Upsilon}X_{i})
\end{eqnarray*}

We calculate 

\begin{eqnarray*}
<N_{0},\nabla_{\partial_{s}}\partial_{s}> & = & <\frac{1}{\phi}(\dot{\phi}X_{0}-\dot{\psi}\Upsilon),\nabla_{\partial_{s}}\partial_{s}>
\end{eqnarray*}
term by term. There would be $20$ terms totally.

\begin{eqnarray}
<\frac{\dot{\phi}}{\phi}X_{0},\varepsilon\ddot{\psi}X_{0}> & = & \varepsilon\frac{\dot{\phi}\ddot{\psi}}{\phi}<X_{0},X_{0}>\nonumber \\
 & = & \varepsilon\frac{\dot{\phi}\ddot{\psi}}{\phi}(1+\varepsilon^{2}\phi^{2}R(\Upsilon,X_{0},\Upsilon,X_{0})_{p}+2\varepsilon\phi R(\Upsilon,X_{0},\eta,X_{0})_{p}\nonumber \\
 &  & +\varepsilon^{2}L(w,\eta)+Q(w,\eta)+O(\varepsilon^{3}))\nonumber \\
 & = & \varepsilon\frac{\dot{\phi}\ddot{\psi}}{\phi}+\varepsilon^{3}\phi\dot{\phi}\ddot{\psi}R(\Upsilon,X_{0},\Upsilon,X_{0})+2\varepsilon^{2}\dot{\phi}\ddot{\psi}R(\Upsilon,X_{0},\eta,X_{0})\nonumber \\
 &  & +\varepsilon^{3}L(w,\eta)+\varepsilon Q(w,\eta)+O(\varepsilon^{4}),\label{eq:}
\end{eqnarray}

\begin{eqnarray*}
<\frac{\dot{\phi}}{\phi}X_{0},\varepsilon(\ddot{\phi}+\frac{\partial^{2}w}{\partial s^{2}})\Upsilon> & = & \varepsilon\frac{\dot{\phi}}{\phi}(\ddot{\phi}+\frac{\partial^{2}w}{\partial s^{2}})<X_{0},\Upsilon>\\
 & = & \varepsilon\frac{\dot{\phi}}{\phi}(\ddot{\phi}+\frac{\partial^{2}w}{\partial s^{2}})(\frac{2}{3}\varepsilon\phi R(\Upsilon,X_{0},\eta,\Upsilon)\\
 &  & +\varepsilon^{2}L(w,\eta)+Q(w,\eta)+O(\varepsilon^{3}))\\
 & = & \frac{2}{3}\varepsilon^{2}\dot{\phi}\ddot{\phi}R(\Upsilon,X_{0},\eta,\Upsilon)+\varepsilon^{3}L(w,\eta)+\varepsilon Q(w,\eta)+O(\varepsilon^{4}),
\end{eqnarray*}

\begin{eqnarray*}
<\frac{\dot{\phi}}{\phi}X_{0},\varepsilon\ddot{\psi}\frac{\partial\eta^{i}}{\partial x^{0}}X_{i}> & = & \varepsilon\frac{\dot{\phi}\ddot{\psi}}{\phi}\frac{\partial\eta^{i}}{\partial x^{0}}<X_{0},X_{i}>\\
 & = & \varepsilon\frac{\dot{\phi}\ddot{\psi}}{\phi}\frac{\partial\eta^{i}}{\partial x_{0}}(O(\varepsilon^{2})+\varepsilon L(w,\eta)+Q(w,\eta)+O(\varepsilon))\\
 & = & \varepsilon^{3}L(w,\eta)+\varepsilon Q(w,\eta),
\end{eqnarray*}
\begin{eqnarray}
<\frac{\dot{\phi}}{\phi}X_{0},\varepsilon^{2}\dot{\psi}^{2}\frac{\partial^{2}\eta^{i}}{\partial x_{0}^{2}}X_{i}> & = & \varepsilon^{2}\frac{\dot{\phi}\dot{\psi}^{2}}{\phi}\frac{\partial^{2}\eta^{i}}{\partial x_{0}^{2}}(O(\varepsilon^{2})+\varepsilon L(w,\eta)\nonumber \\
 &  & +Q(w,\eta)+O(\varepsilon))\label{2nd-derivative-eta}\\
 & = & \varepsilon^{3}L(w,\eta)+\varepsilon Q(w,\eta),\nonumber 
\end{eqnarray}
\begin{eqnarray*}
<\frac{\dot{\phi}}{\phi}X_{0},\varepsilon^{2}\dot{\psi}^{2}\nabla_{X_{0}}X_{0}> & = & \varepsilon^{2}\frac{\dot{\phi}\dot{\psi}^{2}}{\phi}<X_{0},\nabla_{X_{0}}X_{0}>\\
 & = & \varepsilon^{2}\frac{\dot{\phi}\dot{\psi}^{2}}{\phi}(O(\varepsilon^{2})+\varepsilon L(w,\eta)+Q(w,\eta))\\
 & = & O(\varepsilon^{4})+\varepsilon^{3}L(w,\eta)+\varepsilon^{2}Q(w,\eta),
\end{eqnarray*}
\begin{eqnarray*}
<\frac{\dot{\phi}}{\phi}X_{0},\varepsilon^{2}(\dot{\phi}+\frac{\partial w}{\partial s})^{2}\nabla_{\Upsilon}\Upsilon> & = & \varepsilon^{2}\frac{\dot{\phi}}{\phi}(\dot{\phi}+\frac{\partial w}{\partial s})^{2}<X_{0},\nabla_{\Upsilon}\Upsilon>\\
 & = & \varepsilon^{2}\frac{\dot{\phi}}{\phi}(\dot{\phi}+\frac{\partial w}{\partial s})^{2}(\Upsilon<X_{0},\Upsilon>-<\nabla_{X_{0}}\Upsilon,\Upsilon>).
\end{eqnarray*}
The method to calculate this will be used several times, so we write
it in detail. Although $\Upsilon$ actually depends on $\theta$,
here we are only interested in $\nabla_{\Upsilon}\Upsilon,$ we may
pretend $\Upsilon$ is constant vector in the coordinates $\{x_{0},x_{1},x_{2}\}.$
We may make such assumption where it is convenient. Evidently $\Upsilon(x_{k})=\Upsilon^{k}$.
So 
\begin{eqnarray*}
\Upsilon<X_{0},\Upsilon> & = & \Upsilon(\frac{2}{3}R(X_{k},X_{0},X_{l},\Upsilon)_{p}x_{k}x_{l}+O(r^{3}))\\
 & = & \frac{2}{3}R(X_{k},X_{0},X_{l},\Upsilon)_{p}(\Upsilon^{k}(\varepsilon(\phi+w)\Upsilon^{l}+\eta^{l})\\
 &  & +(\varepsilon(\phi+w)\Upsilon^{k}+\eta^{k})\Upsilon^{l})+O(\varepsilon^{2})+\varepsilon L(w,\eta)+Q(w,\eta)\\
 & = & \frac{2}{3}R(\Upsilon,X_{0},\eta,\Upsilon)+O(\varepsilon^{2})+\varepsilon L(w,\eta)+Q(w,\eta),
\end{eqnarray*}
\begin{eqnarray*}
<\nabla_{X_{0}}\Upsilon,\Upsilon>_{p} & = & 0
\end{eqnarray*}
 for all $p\in\Gamma.$ We consider 
\[
X_{j}<\nabla_{X_{0}}\Upsilon,\Upsilon>_{p}=<\nabla_{X_{j}}\nabla_{X_{0}}\Upsilon,\Upsilon>_{p}+<\nabla_{X_{0}}\Upsilon,\nabla_{X_{j}}\Upsilon>_{p}.
\]
From Lemma \ref{Connection coefficient estimate} we know $(\nabla_{X_{\alpha}}X_{\beta})_{p}=0.$
So 
\begin{eqnarray*}
X_{j}<\nabla_{X_{0}}\Upsilon,\Upsilon>_{p} & = & <\nabla_{X_{j}}\nabla_{X_{0}}\Upsilon,\Upsilon>_{p}\\
 & = & <\nabla_{X_{0}}\nabla_{X_{j}}\Upsilon,\Upsilon>_{p}+R(X_{j},X_{0},\Upsilon,\Upsilon)_{p}.
\end{eqnarray*}
We know $R(X_{j},X_{0},\Upsilon,\Upsilon)_{p}=0.$ And we know $\nabla_{X_{j}}\Upsilon=0$
always holds on the geodesic. So $(\nabla_{X_{0}}\nabla_{X_{j}}\Upsilon)_{p}=0.$
So $X_{j}<\nabla_{X_{0}}\Upsilon,\Upsilon>_{p}=0.$ And we know 
\begin{eqnarray*}
<\nabla_{X_{0}}\Upsilon,\Upsilon>(x_{0},x_{1},x_{2}) & = & O(r^{2})=O(\varepsilon^{2})+\varepsilon L(w,\eta)+Q(w,\eta).
\end{eqnarray*}
Now we get 
\begin{eqnarray*}
<\frac{\dot{\phi}}{\phi}X_{0},\varepsilon^{2}(\dot{\phi}+\frac{\partial w}{\partial s})^{2}\nabla_{\Upsilon}\Upsilon> & = & \frac{2}{3}\varepsilon^{2}\frac{\dot{\phi}^{3}}{\phi}R(\Upsilon,X_{0},\eta,\Upsilon)\\
 &  & +O(\varepsilon^{4})+\varepsilon^{3}L(w,\eta)+\varepsilon^{2}Q(w,\eta),
\end{eqnarray*}
By using the skill above we can calculate all the remaining terms.
We state the result directly.
\begin{eqnarray*}
\varepsilon^{2}<\frac{\dot{\phi}}{\phi}X_{0},\dot{\psi}^{2}(\frac{\partial\eta^{i}}{\partial x_{0}})(\frac{\partial\eta^{j}}{\partial x_{0}})\nabla_{X_{i}}X_{j}> & = & \varepsilon^{3}Q(w,\eta),
\end{eqnarray*}
\begin{eqnarray*}
\varepsilon^{2}<\frac{\dot{\phi}}{\phi}X_{0},2\dot{\psi}(\dot{\phi}+\frac{\partial w}{\partial s})\nabla_{X_{0}}\Upsilon> & = & 2\varepsilon^{3}\dot{\phi}^{2}\dot{\psi}R(\Upsilon,X_{0},\Upsilon,X_{0})\\
 &  & +2\varepsilon^{2}\frac{\dot{\phi}^{2}\dot{\psi}}{\phi}R(\Upsilon,X_{0},\eta,X_{0})\\
 &  & +O(\varepsilon^{4})+\varepsilon^{3}L(w,\eta)+\varepsilon^{2}Q(w,\eta),
\end{eqnarray*}
\begin{eqnarray*}
\varepsilon^{2}<\frac{\dot{\phi}}{\phi}X_{0},2\dot{\psi}^{2}\frac{\partial\eta^{i}}{\partial x_{0}}\nabla_{X_{0}}X_{i}> & = & \varepsilon^{3}L(w,\eta),
\end{eqnarray*}
\begin{eqnarray*}
\varepsilon^{2}<\frac{\dot{\phi}}{\phi}X_{0},2(\dot{\phi}+\frac{\partial w}{\partial s})\dot{\psi}\frac{\partial\eta^{i}}{\partial x_{0}}\nabla_{\Upsilon}X_{i}> & = & \varepsilon^{3}L(w,\eta)+\varepsilon^{2}Q(w,\eta),
\end{eqnarray*}
\begin{eqnarray*}
<-\frac{\dot{\psi}}{\phi}\Upsilon,\varepsilon\ddot{\psi}X_{0}> & = & -\varepsilon\frac{\dot{\psi}\ddot{\psi}}{\phi}<\Upsilon,X_{0}>\\
 & = & -\frac{2}{3}\varepsilon^{2}\dot{\psi}\ddot{\psi}R(\Upsilon,X_{0},\eta,\Upsilon)+\varepsilon^{3}L(w,\eta)+\varepsilon Q(w,\eta)+O(\varepsilon^{4}),
\end{eqnarray*}
\begin{eqnarray*}
<-\frac{\dot{\psi}}{\phi}\Upsilon,\varepsilon(\ddot{\phi}+\frac{\partial^{2}w}{\partial s^{2}})\Upsilon> & = & -\varepsilon\frac{\ddot{\phi}\dot{\psi}}{\phi}-\varepsilon\frac{\dot{\psi}}{\phi}\frac{\partial^{2}w}{\partial s^{2}}\\
 &  & +\varepsilon^{3}L(w,\eta)+\varepsilon Q(w,\eta)+O(\varepsilon^{4}),
\end{eqnarray*}
\begin{eqnarray*}
<-\frac{\dot{\psi}}{\phi}\Upsilon,\varepsilon\ddot{\psi}\frac{\partial\eta^{i}}{\partial x_{0}}X_{i}> & = & -\varepsilon\frac{\dot{\psi}\ddot{\psi}}{\phi}<\frac{\partial\eta}{\partial x_{0}},\Upsilon>,
\end{eqnarray*}
\begin{eqnarray*}
<-\frac{\dot{\psi}}{\phi}\Upsilon,\varepsilon^{2}\dot{\psi}^{2}\frac{\partial^{2}\eta^{i}}{\partial x_{0}^{2}}X_{i}> & = & -\varepsilon^{2}\frac{\dot{\psi}^{3}}{\phi}<\frac{\partial^{2}\eta}{\partial x_{0}^{2}},\Upsilon>,
\end{eqnarray*}
\begin{eqnarray*}
<-\frac{\dot{\psi}}{\phi}\Upsilon,\varepsilon^{2}\dot{\psi}^{2}\nabla_{X_{0}}X_{0}> & = & \varepsilon^{3}\dot{\psi}^{3}R(\Upsilon,X_{0},\Upsilon,X_{0})+\varepsilon^{2}\frac{\dot{\psi}^{3}}{\phi}R(\Upsilon,X_{0},\eta,X_{0})\\
 &  & +\varepsilon^{3}L(w,\eta)+\varepsilon^{2}Q(w,\eta)+O(\varepsilon^{4}),
\end{eqnarray*}
\[
<-\frac{\dot{\psi}}{\phi}\Upsilon,\varepsilon^{2}(\dot{\phi}+\frac{\partial w}{\partial s})^{2}\nabla_{\Upsilon}\Upsilon>=\varepsilon^{3}L(w,\eta)+\varepsilon^{2}Q(w,\eta)+O(\varepsilon^{4}),
\]
\begin{eqnarray*}
<-\frac{\dot{\psi}}{\phi}\Upsilon,\varepsilon^{2}\dot{\psi}^{2}(\frac{\partial\eta^{i}}{\partial x_{0}})(\frac{\partial\eta^{j}}{\partial x_{0}})\nabla_{X_{i}}X_{j}> & = & \varepsilon^{3}Q(w,\eta),
\end{eqnarray*}
\begin{eqnarray*}
<-\frac{\dot{\psi}}{\phi}\Upsilon,\varepsilon^{2}2\dot{\psi}(\dot{\phi}+\frac{\partial w}{\partial s})\nabla_{X_{0}}\Upsilon> & = & \varepsilon^{3}L(w,\eta)+\varepsilon^{2}Q(w,\eta)+O(\varepsilon^{4}),
\end{eqnarray*}
\[
<-\frac{\dot{\psi}}{\phi}\Upsilon,\varepsilon^{2}2\dot{\psi}^{2}\frac{\partial\eta^{i}}{\partial x_{0}}\nabla_{X_{0}}X_{i}>=\varepsilon^{3}L(w,\eta),
\]
\[
<-\frac{\dot{\psi}}{\phi}\Upsilon,\varepsilon^{2}2(\dot{\phi}+\frac{\partial w}{\partial s})\dot{\psi}\frac{\partial\eta^{i}}{\partial x_{0}}\nabla_{\Upsilon}X_{i}>=\varepsilon^{3}L(w,\eta)+\varepsilon^{2}Q(w,\eta).
\]

Collecting all the terms above and notice that
\begin{equation}
\frac{\dot{\phi}\ddot{\psi}}{\phi}-\frac{\ddot{\phi}\dot{\psi}}{\phi}=\phi^{2}-\tau_{0}\label{one key calculation}
\end{equation}
 we get 
\begin{eqnarray}
<N_{0},\nabla_{\partial_{s}}\partial_{s}> & = & \varepsilon(\phi^{2}-\tau_{0})-\varepsilon\frac{\dot{\psi}}{\phi}\frac{\partial^{2}w}{\partial s^{2}}-\varepsilon^{2}\frac{\dot{\psi}^{3}}{\phi}<\frac{\partial^{2}\eta}{(\partial x_{0})^{2}},\Upsilon>-\varepsilon\frac{\dot{\psi}\ddot{\psi}}{\phi}<\frac{\partial\eta}{\partial x_{0}},\Upsilon>\nonumber \\
 &  & +\varepsilon^{3}(\phi\dot{\phi}\ddot{\psi}+2\dot{\phi}^{2}\dot{\psi}+\dot{\psi}^{3})R(\Upsilon,X_{0},\Upsilon,X_{0})\nonumber \\
 &  & +\varepsilon^{2}(2\dot{\phi}\ddot{\psi}+2\frac{\dot{\phi}^{2}\dot{\psi}}{\phi}+\frac{\dot{\psi}^{3}}{\phi})R(\Upsilon,X_{0},\eta,X_{0})\nonumber \\
 &  & +\varepsilon^{2}(\frac{2}{3}\dot{\phi}\ddot{\phi}+\frac{2}{3}\frac{\dot{\phi}^{3}}{\phi}-\frac{2}{3}\dot{\psi}\ddot{\psi})R(\Upsilon,X_{0},\eta,\Upsilon)\nonumber \\
 &  & +\varepsilon^{3}L(w,\eta)+\varepsilon Q(w,\eta)+O(\varepsilon^{4}),\label{eq:N0-partial(s)}
\end{eqnarray}

\begin{eqnarray*}
g^{\theta\theta}<kN,\nabla_{\partial_{\theta}}\partial_{\theta}> & = & g^{\theta\theta}<N_{0}+a_{1}\partial_{s}+a_{2}\partial_{\theta},\nabla_{\partial_{\theta}}\partial_{\theta}>.
\end{eqnarray*}
\begin{eqnarray}
<\partial_{s},\nabla_{\partial_{\theta}}\partial_{\theta}> & = & \partial_{\theta}<\partial_{s},\partial_{\theta}>-\frac{1}{2}\partial_{s}<\partial_{\theta},\partial_{\theta}>\nonumber \\
 & = & \varepsilon^{2}\partial_{\theta}(\frac{2}{3}\varepsilon^{2}\phi^{3}\dot{\psi}R(\Upsilon,X_{0},\Upsilon,\Upsilon_{\theta})+\frac{2}{3}\varepsilon\phi^{2}\dot{\psi}(R(\Upsilon,X_{0},\eta,\Upsilon_{\theta})\nonumber \\
 &  & +R(\eta,X_{0},\Upsilon,\Upsilon_{\theta}))+\frac{1}{3}\varepsilon\phi^{2}\dot{\phi}R(\eta,\Upsilon,\Upsilon,\Upsilon_{\theta})\nonumber \\
 &  & +\dot{\phi}\frac{\partial w}{\partial\theta}+\phi\dot{\psi}<\frac{\partial\eta}{\partial x_{0}},\Upsilon_{\theta}>_{e}+\varepsilon^{2}L(w,\eta)+Q(w,\eta)+O(\varepsilon^{3}))\nonumber \\
 &  & -\frac{\varepsilon^{2}}{2}\partial_{s}(\phi^{2}+2\phi w+\frac{1}{3}\varepsilon^{2}\phi^{4}R(\Upsilon,\Upsilon_{\theta},\Upsilon,\Upsilon_{\theta})\nonumber \\
 &  & +\frac{2}{3}\varepsilon\phi^{3}R(\Upsilon,\Upsilon_{\theta},\eta,\Upsilon_{\theta})+\varepsilon^{2}L(w,\eta)+Q(w,\eta)+O(\varepsilon^{3}))\nonumber \\
 & = & -\varepsilon^{2}\phi\dot{\phi}+O(\varepsilon^{4})+\varepsilon^{2}L(w,\eta)+\varepsilon^{2}Q(w,\eta),\label{eq:partial(s)-for-partial(theta)}
\end{eqnarray}
\begin{eqnarray}
<\partial_{\theta},\nabla_{\partial_{\theta}}\partial_{\theta}> & = & \frac{1}{2}\partial_{\theta}<\partial_{\theta},\partial_{\theta}>\nonumber \\
 & = & \frac{\varepsilon^{2}}{2}\partial_{\theta}(\phi^{2}+2\phi w+\frac{1}{3}\varepsilon^{2}\phi^{4}R(\Upsilon,\Upsilon_{\theta},\Upsilon,\Upsilon_{\theta})\nonumber \\
 &  & +\frac{2}{3}\varepsilon\phi^{3}R(\Upsilon,\Upsilon_{\theta},\eta,\Upsilon_{\theta})+\varepsilon^{2}L(w,\eta)+Q(w,\eta)+O(\varepsilon^{3}))\nonumber \\
 & = & O(\varepsilon^{4})+\varepsilon^{2}L(w,\eta)+\varepsilon^{2}Q(w,\eta),\label{eq:partial(theta)-for-partial(theta)}
\end{eqnarray}
\begin{eqnarray*}
\nabla_{\partial_{\theta}}\partial_{\theta} & = & \varepsilon\frac{\partial w}{\partial\theta}\Upsilon_{\theta}+\varepsilon\frac{\partial^{2}w}{\partial\theta^{2}}\Upsilon+\varepsilon^{2}(\phi+w)^{2}\nabla_{\Upsilon_{\theta}}\Upsilon_{\theta}\\
 &  & +\varepsilon^{2}(\frac{\partial w}{\partial\theta})^{2}\nabla_{\Upsilon}\Upsilon+\varepsilon^{2}(\phi+w)\frac{\partial w}{\partial\theta}(\nabla_{\Upsilon_{\theta}}\Upsilon+\nabla_{\Upsilon}\Upsilon_{\theta})
\end{eqnarray*}
\begin{eqnarray*}
<\frac{\dot{\phi}}{\phi}X_{0},\varepsilon\frac{\partial w}{\partial\theta}\Upsilon_{\theta}> & = & \varepsilon^{3}L(w,\eta)+\varepsilon Q(w,\eta),
\end{eqnarray*}
\begin{eqnarray*}
<\frac{\dot{\phi}}{\phi}X_{0},\varepsilon\frac{\partial^{2}w}{\partial\theta^{2}}\Upsilon> & = & \varepsilon^{4}L(w,\eta)+\varepsilon Q(w,\eta),
\end{eqnarray*}
\begin{eqnarray*}
<\frac{\dot{\phi}}{\phi}X_{0},\varepsilon^{2}(\phi+w)^{2}\nabla_{\Upsilon_{\theta}}\Upsilon_{\theta}> & = & \varepsilon^{2}\frac{\dot{\phi}}{\phi}(\phi+w)^{2}<X_{0},\nabla_{\Upsilon_{\theta}}\Upsilon_{\theta}>.
\end{eqnarray*}
Notice that 
\begin{eqnarray*}
<X_{0},\nabla_{\Upsilon_{\theta}}\Upsilon_{\theta}> & = & \Upsilon_{\theta}<X_{0},\Upsilon_{\theta}>-<\nabla_{\Upsilon_{\theta}}X_{0},\Upsilon_{\theta}>.
\end{eqnarray*}
 We recall $\Upsilon(x_{0},x_{1},x_{2})=\frac{(x_{1}-\eta_{1},x_{2}-\eta_{2})}{\sqrt{(x_{1}-\eta_{1})^{2}+(x_{2}-\eta_{2})^{2}}}=(\cos\theta,\sin\theta)$
and $\Upsilon_{\theta}(x_{0},x_{1},x_{2})=\frac{(-x_{2}+\eta_{2},x_{1}-\eta_{1})}{\sqrt{(x_{1}-\eta_{1})^{2}+(x_{2}-\eta_{2})^{2}}}=(-\sin\theta,\cos\theta).$
We denote $(-x_{2}+\eta_{2},x_{1}-\eta_{1})$ by $\tilde{\partial}_{\theta},$
and $\sqrt{(x_{1}-\eta_{1})^{2}+(x_{2}-\eta_{2})^{2}}$ by $\tilde{r}$
\begin{eqnarray*}
\Upsilon_{\theta}<X_{0},\Upsilon_{\theta}> & = & \frac{1}{\tilde{r}}\tilde{\partial}_{\theta}<X_{0},-\sin\theta X_{1}+\cos\theta X_{2}>\\
 & = & \frac{1}{\tilde{r}}<X_{0},-\Upsilon>+(-\sin\theta)\Upsilon_{\theta}<X_{0},X_{1}>+\cos\theta\Upsilon_{\theta}<X_{0},X_{2}>\\
 & = & \frac{2}{3}\varepsilon\phi R(\Upsilon_{\theta},X_{0},\Upsilon,\Upsilon_{\theta})+\frac{2}{3}R(\Upsilon_{\theta},X_{0},\eta,\Upsilon_{\theta})-\frac{2}{3}R(\Upsilon,X_{0},\eta,\Upsilon)\\
 &  & +\varepsilon L(w,\eta)+\varepsilon^{-1}Q(w,\eta)+O(\varepsilon^{2}),
\end{eqnarray*}
where we calculate $\Upsilon_{\theta}<X_{0},X_{i}>$ directly from
Lemma \ref{gij expansion}. To calculate $<\nabla_{\Upsilon_{\theta}}X_{0},\Upsilon_{\theta}>$
we can regard $\Upsilon_{\theta}$ as constant in $(x_{0},x_{1},x_{2})$
coordinate.
\begin{eqnarray*}
<\nabla_{\Upsilon_{\theta}}X_{0},\Upsilon_{\theta}> & = & \frac{1}{2}X_{0}<\Upsilon_{\theta},\Upsilon_{\theta}>\\
 & = & O(\varepsilon^{2})+\varepsilon L(w,\eta)+Q(w,\eta).
\end{eqnarray*}
We have 
\begin{eqnarray*}
<\frac{\dot{\phi}}{\phi}X_{0},\varepsilon^{2}(\phi+w)^{2}\nabla_{\Upsilon_{\theta}}\Upsilon_{\theta}> & = & \frac{2}{3}\varepsilon^{3}\dot{\phi}\phi^{2}R(\Upsilon_{\theta},X_{0},\Upsilon,\Upsilon_{\theta})\\
 &  & +\frac{2}{3}\varepsilon^{2}\phi\dot{\phi}(R(\Upsilon_{\theta},X_{0},\eta,\Upsilon_{\theta})-R(\Upsilon,X_{0},\eta,\Upsilon))\\
 &  & +\varepsilon^{3}L(w,\eta)+\varepsilon Q(w,\eta)+O(\varepsilon^{4}).
\end{eqnarray*}
\begin{eqnarray*}
<\frac{\dot{\phi}}{\phi}X_{0},\varepsilon^{2}(\frac{\partial w}{\partial\theta})^{2}\nabla_{\Upsilon}\Upsilon> & = & \varepsilon^{3}Q(w,\eta).
\end{eqnarray*}
\begin{eqnarray*}
<\frac{\dot{\phi}}{\phi}X_{0},\varepsilon^{2}(\phi+w)\frac{\partial w}{\partial\theta}\nabla_{\Upsilon_{\theta}}\Upsilon> & = & \varepsilon^{3}L(w,\eta)+\varepsilon^{2}Q(w,\eta).
\end{eqnarray*}
\begin{eqnarray*}
<\frac{\dot{\phi}}{\phi}X_{0},\varepsilon^{2}(\phi+w)\frac{\partial w}{\partial\theta}\nabla_{\Upsilon}\Upsilon_{\theta}> & = & \varepsilon^{3}L(w,\eta)+\varepsilon^{2}Q(w,\eta).
\end{eqnarray*}
 Note that $\nabla_{\Upsilon_{\theta}}\Upsilon$ and $\nabla_{\Upsilon}\Upsilon_{\theta}$
are not the same.
\begin{eqnarray*}
<-\frac{\dot{\psi}}{\phi}\Upsilon,\varepsilon\frac{\partial w}{\partial\theta}\Upsilon_{\theta}> & = & \varepsilon^{3}L(w,\eta)+\varepsilon Q(w,\eta).
\end{eqnarray*}
\begin{eqnarray*}
<-\frac{\dot{\psi}}{\phi}\Upsilon,\varepsilon\frac{\partial^{2}w}{\partial\theta^{2}}\Upsilon> & = & -\varepsilon\frac{\dot{\psi}}{\phi}\frac{\partial^{2}w}{\partial\theta^{2}}+\varepsilon^{3}L(w,\eta)+\varepsilon^{3}Q(w,\eta).
\end{eqnarray*}
\begin{eqnarray*}
<-\frac{\dot{\psi}}{\phi}\Upsilon,\varepsilon^{2}(\phi+w)^{2}\nabla_{\Upsilon_{\theta}}\Upsilon_{\theta}> & = & -\varepsilon^{2}\frac{\dot{\psi}}{\phi}(\phi+w)^{2}<\Upsilon,\nabla_{\Upsilon_{\theta}}\Upsilon_{\theta}>\\
 & = & -\varepsilon^{2}\frac{\dot{\psi}}{\phi}(\phi+w)^{2}(\Upsilon_{\theta}<\Upsilon,\Upsilon_{\theta}>-<\nabla_{\Upsilon_{\theta}}\Upsilon,\Upsilon_{\theta}>),
\end{eqnarray*}
First we have 
\begin{eqnarray*}
\Upsilon_{\theta}<\Upsilon,\Upsilon_{\theta}> & = & \frac{1}{\tilde{r}}(<\Upsilon_{\theta},\Upsilon_{\theta}>-<\Upsilon,\Upsilon>)\\
 &  & +(-\sin\theta\cos\theta\Upsilon_{\theta}<X_{1},X_{1}>+\sin\theta\cos\theta\Upsilon_{\theta}<X_{2},X_{2}>\\
 &  & +(\cos^{2}\theta-\sin^{2}\theta)\Upsilon_{\theta}<X_{1},X_{2}>)\\
 & = & \frac{1}{3}R(\Upsilon,\Upsilon_{\theta},\eta,\Upsilon_{\theta})+\varepsilon L(w,\eta)+\varepsilon^{-1}Q(w,\eta)+O(\varepsilon^{2}),
\end{eqnarray*}
\begin{eqnarray*}
<\nabla_{\Upsilon_{\theta}}\Upsilon,\Upsilon_{\theta}> & = & \frac{1}{\tilde{r}}<\nabla_{\tilde{r}\Upsilon_{\theta}}\Upsilon,\Upsilon_{\theta}>\\
 & = & \frac{1}{\tilde{r}}<\nabla_{\Upsilon}\tilde{r}\Upsilon_{\theta},\Upsilon_{\theta}>\\
 & = & \frac{1}{\tilde{r}}<\Upsilon_{\theta},\Upsilon_{\theta}>+<\nabla_{\Upsilon}\Upsilon_{\theta},\Upsilon_{\theta}>.
\end{eqnarray*}
 From 
\begin{eqnarray*}
\frac{1}{\tilde{r}}<\Upsilon_{\theta},\Upsilon_{\theta}> & = & \frac{1}{\tilde{r}}+\frac{1}{3}\varepsilon\phi R(\Upsilon,\Upsilon_{\theta},\Upsilon,\Upsilon_{\theta})+\frac{2}{3}R(\Upsilon,\Upsilon_{\theta},\eta,\Upsilon_{\theta})\\
 &  & +\varepsilon L(w,\eta)+\varepsilon^{-1}Q(w,\eta)+O(\varepsilon^{2}),
\end{eqnarray*}
\begin{eqnarray*}
<\nabla_{\Upsilon}\Upsilon_{\theta},\Upsilon_{\theta}> & = & \frac{1}{2}\Upsilon<\Upsilon_{\theta},\Upsilon_{\theta}>\\
 & = & \frac{1}{3}\varepsilon\phi R(\Upsilon,\Upsilon_{\theta},\Upsilon,\Upsilon_{\theta})+\frac{1}{3}R(\Upsilon,\Upsilon_{\theta},\eta,\Upsilon_{\theta})\\
 &  & +O(\varepsilon^{2})+\varepsilon L(w,\eta)+Q(w,\eta).
\end{eqnarray*}
We get 
\begin{eqnarray*}
<\nabla_{\Upsilon_{\theta}}\Upsilon,\Upsilon_{\theta}> & = & \frac{1}{\tilde{r}}+\frac{2}{3}\varepsilon\phi R(\Upsilon,\Upsilon_{\theta},\Upsilon,\Upsilon_{\theta})+R(\Upsilon,\Upsilon_{\theta},\eta,\Upsilon_{\theta})\\
 &  & +\varepsilon L(w,\eta)+\varepsilon^{-1}Q(w,\eta)+O(\varepsilon^{2}),
\end{eqnarray*}
\begin{eqnarray*}
<-\frac{\dot{\psi}}{\phi}\Upsilon,\varepsilon^{2}(\phi+w)^{2}\nabla_{\Upsilon_{\theta}}\Upsilon_{\theta}> & = & \varepsilon(\phi^{2}+\tau_{0})+\varepsilon\frac{\dot{\psi}}{\phi}w+\frac{2}{3}\varepsilon^{3}\phi^{2}\dot{\psi}R(\Upsilon,\Upsilon_{\theta},\Upsilon,\Upsilon_{\theta})\\
 &  & +\frac{2}{3}\phi\dot{\psi}R(\Upsilon,\Upsilon_{\theta},\eta,\Upsilon_{\theta})\\
 &  & +\varepsilon^{3}L(w,\eta)+\varepsilon Q(w,\eta)+O(\varepsilon^{4}).
\end{eqnarray*}
\[
<-\frac{\dot{\psi}}{\phi}\Upsilon,\varepsilon^{2}(\frac{\partial w}{\partial\theta})^{2}\nabla_{\Upsilon}\Upsilon>=\varepsilon^{3}Q(w,\eta).
\]
\[
<-\frac{\dot{\psi}}{\phi}\Upsilon,\varepsilon^{2}(\phi+w)\frac{\partial w}{\partial\theta}\nabla_{\Upsilon_{\theta}}\Upsilon>=\varepsilon^{3}L(w,\eta)+\varepsilon^{2}Q(w,\eta).
\]
\[
<-\frac{\dot{\psi}}{\phi}\Upsilon,\varepsilon^{2}(\phi+w)\frac{\partial w}{\partial\theta}\nabla_{\Upsilon}\Upsilon_{\theta}>=\varepsilon^{3}L(w,\eta)+\varepsilon^{2}Q(w,\eta).
\]
We collect all the terms and get 
\begin{eqnarray}
 &  & <N_{0},\nabla_{\partial_{\theta}}\partial_{\theta}>\nonumber \\
 & = & \varepsilon(\phi^{2}+\tau_{0})-\varepsilon\frac{\dot{\psi}}{\phi}\frac{\partial^{2}w}{\partial\theta^{2}}+\varepsilon\frac{\dot{\psi}}{\phi}w\nonumber \\
 &  & +\frac{2}{3}\varepsilon^{3}\phi^{2}\dot{\psi}R(\Upsilon,\Upsilon_{\theta},\Upsilon,\Upsilon_{\theta})+\frac{2}{3}\varepsilon^{3}\dot{\phi}\phi^{2}R(\Upsilon_{\theta},X_{0},\Upsilon,\Upsilon_{\theta})\nonumber \\
 &  & +\frac{2}{3}\varepsilon^{2}\phi\dot{\phi}R(\Upsilon_{\theta},X_{0},\eta,\Upsilon_{\theta})-\frac{2}{3}\varepsilon^{2}\phi\dot{\phi}R(\Upsilon,X_{0},\eta,\Upsilon)\nonumber \\
 &  & +\frac{2}{3}\phi\dot{\psi}R(\Upsilon,\Upsilon_{\theta},\eta,\Upsilon_{\theta})+\varepsilon^{3}L(w,\eta)+\varepsilon Q(w,\eta)+O(\varepsilon^{4}).\label{eq:N0-partial(theta)}
\end{eqnarray}
For 
\[
<N,\nabla_{\partial_{\theta}}\partial_{s}>=<N_{0}+a_{1}\partial_{s}+a_{2}\partial_{\theta},\nabla_{\partial_{\theta}}\partial_{s}>
\]
we don't need very precise expansion because $g^{s\theta}$ is small
relatively.
\begin{eqnarray}
<\partial_{s},\nabla_{\partial_{\theta}}\partial_{s}> & = & O(\varepsilon^{4})+\varepsilon^{2}L(w,\eta)+\varepsilon^{2}Q(w,\eta),\label{eq:partial(s)-for-theta-s}
\end{eqnarray}
\begin{eqnarray}
<\partial_{\theta},\nabla_{\partial_{\theta}}\partial_{s}> & = & \varepsilon^{2}\phi\dot{\phi}+O(\varepsilon^{4})+\varepsilon^{2}L(w,\eta)+\varepsilon^{2}Q(w,\eta).\label{eq:theta-for-theta-s}
\end{eqnarray}
\begin{eqnarray}
<N_{0},\nabla_{\partial_{\theta}}\partial_{s}> & = & O(\varepsilon^{3})+\varepsilon^{2}L(w,\eta)+\varepsilon Q(w,\eta).\label{eq:N0--theta-s}
\end{eqnarray}
Now we can calculate the mean curvature. From (\ref{eq:sigma(1,2,3)}),
(\ref{eq:inverse matrix of metric}), (\ref{eq:a1-a2-expressions}),(\ref{eq:partial(s)-for-partial(s)}),
(\ref{eq:partial(theta)-for-partial(s)}) and (\ref{eq:N0-partial(s)})
we know 
\begin{eqnarray*}
g^{ss}<kN,\nabla_{\partial_{s}}\partial_{s}> & = & \varepsilon^{-2}\phi^{-2}(1-(\varepsilon^{2}\dot{\psi}^{2}R(\Upsilon,X_{0},\Upsilon,X_{0})+2\varepsilon\phi^{-1}\dot{\psi}^{2}R(\Upsilon,X_{0},\eta,X_{0})\\
 &  & +\frac{4}{3}\varepsilon\phi^{-1}\dot{\phi}\dot{\psi}R(\Upsilon,X_{0},\eta,\Upsilon)+2\phi^{-2}\dot{\phi}\frac{\partial w}{\partial s}+2\phi^{-2}\dot{\phi}\dot{\psi}<\Upsilon,\frac{\partial\eta}{\partial x_{0}}>_{e}\\
 &  & +\varepsilon^{2}L(w,\eta)+Q(w,\eta)+O(\varepsilon^{3})))\\
 &  & (a_{1}\frac{\varepsilon^{2}}{2}(2\phi\dot{\phi}+L(w,\eta)+Q(w,\eta)+O(\varepsilon^{2}))\\
 &  & +a_{2}\varepsilon^{2}(O(\varepsilon^{2})+L(w,\eta)+Q(w,\eta))\\
 &  & +\varepsilon(\phi^{2}-\tau_{0})-\varepsilon\frac{\dot{\psi}}{\phi}\frac{\partial^{2}w}{\partial s^{2}}-\varepsilon^{2}\frac{\dot{\psi}^{3}}{\phi}<\frac{\partial^{2}\eta}{\partial x_{0}^{2}},\Upsilon>-\varepsilon\frac{\dot{\psi}\ddot{\psi}}{\phi}<\frac{\partial\eta}{\partial x_{0}},\Upsilon>\\
 &  & +\varepsilon^{3}(\phi\dot{\phi}\ddot{\psi}+2\dot{\phi}^{2}\dot{\psi}+\dot{\psi}^{3})R(\Upsilon,X_{0},\Upsilon,X_{0})\\
 &  & +\varepsilon^{2}(2\dot{\phi}\ddot{\psi}+2\frac{\dot{\phi}^{2}\dot{\psi}}{\phi}+\frac{\dot{\psi}^{3}}{\phi})R(\Upsilon,X_{0},\eta,X_{0})\\
 &  & +\varepsilon^{2}(\frac{2}{3}\dot{\phi}\ddot{\phi}+\frac{2}{3}\frac{\dot{\phi}^{3}}{\phi}-\frac{2}{3}\dot{\psi}\ddot{\psi})R(\Upsilon,X_{0},\eta,\Upsilon)\\
 &  & +\varepsilon^{3}L(w,\eta)+\varepsilon Q(w,\eta)+O(\varepsilon^{4}))
\end{eqnarray*}

\begin{eqnarray*}
 & = & \varepsilon^{-2}\phi^{-2}(a_{1}\varepsilon^{2}\phi\dot{\phi}+\varepsilon(\phi^{2}-\tau_{0})-\varepsilon\frac{\dot{\psi}}{\phi}\frac{\partial^{2}w}{\partial s^{2}}-2\varepsilon(\phi^{2}-\tau_{0})\frac{\dot{\phi}}{\phi^{2}}\frac{\partial w}{\partial s}\\
 &  & -\varepsilon^{2}\frac{\dot{\psi}^{3}}{\phi}<\frac{\partial^{2}\eta}{\partial x_{0}^{2}},\Upsilon>-\varepsilon(\frac{\dot{\psi}\ddot{\psi}}{\phi}+2(\phi^{2}-\tau_{0})\frac{\dot{\phi}\dot{\psi}}{\phi^{2}})<\frac{\partial\eta}{\partial x_{0}},\Upsilon>\\
 &  & +\varepsilon^{3}(\phi\dot{\phi}\ddot{\psi}+2\dot{\phi}^{2}\dot{\psi}+\dot{\psi}^{3}-(\phi^{2}-\tau_{0})\dot{\psi}^{2})R(\Upsilon,X_{0},\Upsilon,X_{0})\\
 &  & +\varepsilon^{2}(2\dot{\phi}\ddot{\psi}+2\frac{\dot{\phi}^{2}\dot{\psi}}{\phi}+\frac{\dot{\psi}^{3}}{\phi}-2\frac{\dot{\psi}^{2}}{\phi}(\phi^{2}-\tau_{0}))R(\Upsilon,X_{0},\eta,X_{0})\\
 &  & +\varepsilon^{2}(\frac{2}{3}\dot{\phi}\ddot{\phi}+\frac{2}{3}\frac{\dot{\phi}^{3}}{\phi}-\frac{2}{3}\dot{\psi}\ddot{\psi}-\frac{4}{3}(\phi^{2}-\tau_{0})\frac{\dot{\phi}\dot{\psi}}{\phi})R(\Upsilon,X_{0},\eta,\Upsilon)\\
 &  & +\varepsilon^{3}L(w,\eta)+\varepsilon Q(w,\eta)+O(\varepsilon^{4})).
\end{eqnarray*}
From (\ref{eq:sigma(1,2,3)}), (\ref{eq:inverse matrix of metric}),
(\ref{eq:a1-a2-expressions}), (\ref{eq:partial(s)-for-partial(theta)}),
(\ref{eq:partial(theta)-for-partial(theta)}) and(\ref{eq:N0-partial(theta)})
we know 
\begin{eqnarray*}
g^{\theta\theta}<kN,\nabla_{\partial_{\theta}}\partial_{\theta}> & = & \varepsilon^{-2}\phi^{-2}(1-(2\phi^{-1}w+\frac{1}{3}\varepsilon^{2}\phi^{2}R(\Upsilon,\Upsilon_{\theta},\Upsilon,\Upsilon_{\theta})\\
 &  & +\frac{2}{3}\varepsilon\phi R(\Upsilon,\Upsilon_{\theta},\eta,\Upsilon_{\theta})+\varepsilon^{2}L(w,\eta)+Q(w,\eta)+O(\varepsilon^{3})))\\
 &  & (a_{1}(-\varepsilon^{2}\phi\dot{\phi}+O(\varepsilon^{4})+\varepsilon^{2}L(w,\eta)+\varepsilon^{2}Q(w,\eta))\\
 &  & +a_{2}(O(\varepsilon^{4})+\varepsilon^{2}L(w,\eta)+\varepsilon^{2}Q(w,\eta))\\
 &  & +\varepsilon(\phi^{2}+\tau_{0})-\varepsilon\frac{\dot{\psi}}{\phi}\frac{\partial^{2}w}{\partial\theta^{2}}+\varepsilon\frac{\dot{\psi}}{\phi}w+\frac{2}{3}\varepsilon^{3}\phi^{2}\dot{\psi}R(\Upsilon,\Upsilon_{\theta},\Upsilon,\Upsilon_{\theta})\\
 &  & +\frac{2}{3}\varepsilon^{3}\dot{\phi}\phi^{2}R(\Upsilon_{\theta},X_{0},\Upsilon,\Upsilon_{\theta})+\frac{2}{3}\varepsilon^{2}\phi\dot{\phi}R(\Upsilon_{\theta},X_{0},\eta,\Upsilon_{\theta})\\
 &  & -\frac{2}{3}\varepsilon^{2}\phi\dot{\phi}R(\Upsilon,X_{0},\eta,\Upsilon)+\frac{2}{3}\varepsilon^{2}\phi\dot{\psi}R(\Upsilon,\Upsilon_{\theta},\eta,\Upsilon_{\theta})\\
 &  & +\varepsilon^{3}L(w,\eta)+\varepsilon Q(w,\eta)+O(\varepsilon^{4}))\\
 & = & \varepsilon^{-2}\phi^{-2}(-\varepsilon^{2}a_{1}\phi\dot{\phi}+\varepsilon(\phi^{2}+\tau_{0})-\varepsilon\frac{\dot{\psi}}{\phi}\frac{\partial^{2}w}{\partial\theta^{2}}-\varepsilon\frac{\dot{\psi}}{\phi}w\\
 &  & +\frac{1}{3}\varepsilon^{3}\phi^{2}\dot{\psi}R(\Upsilon,\Upsilon_{\theta},\Upsilon,\Upsilon_{\theta})+\frac{2}{3}\varepsilon^{3}\dot{\phi}\phi^{2}R(\Upsilon_{\theta},X_{0},\Upsilon,\Upsilon_{\theta})\\
 &  & +\frac{2}{3}\varepsilon^{2}\phi\dot{\phi}R(\Upsilon_{\theta},X_{0},\eta,\Upsilon_{\theta})-\frac{2}{3}\varepsilon^{2}\phi\dot{\phi}R(\Upsilon,X_{0},\eta,\Upsilon)\\
 &  & +\varepsilon^{3}L(w,\eta)+\varepsilon Q(w,\eta)+O(\varepsilon^{4}))
\end{eqnarray*}
 From (\ref{eq:sigma(1,2,3)}), (\ref{eq:inverse matrix of metric}),
(\ref{eq:a1-a2-expressions}), (\ref{eq:partial(s)-for-theta-s}),
(\ref{eq:theta-for-theta-s}) and (\ref{eq:N0--theta-s}) we know
\begin{eqnarray*}
g^{s\theta}<kN,\nabla_{\partial_{\theta}}\partial_{s}> & = & \varepsilon^{-2}\phi^{-2}(-(\frac{2}{3}\varepsilon^{2}\phi\dot{\psi}R(\Upsilon,X_{0},\Upsilon,\Upsilon_{\theta})+\frac{2}{3}\varepsilon\dot{\psi}(R(\Upsilon,X_{0},\eta,\Upsilon_{\theta})\\
 &  & +R(\eta,X_{0},\Upsilon,\Upsilon_{\theta}))+\frac{1}{3}\varepsilon\dot{\phi}R(\eta,\Upsilon,\Upsilon,\Upsilon_{\theta})\\
 &  & +\phi^{-2}\dot{\phi}\frac{\partial w}{\partial\theta}+\phi^{-1}\dot{\psi}<\frac{\partial\eta}{\partial x_{0}},\Upsilon_{\theta}>_{e}+\varepsilon^{2}L(w,\eta)+Q(w,\eta)+O(\varepsilon^{3})))\\
 &  & (a_{1}(O(\varepsilon^{4})+\varepsilon^{2}L(w,\eta)+\varepsilon^{2}Q(w,\eta))\\
 &  & +a_{2}(\varepsilon^{2}\phi\dot{\phi}+O(\varepsilon^{4})+\varepsilon^{2}L(w,\eta)+\varepsilon^{2}Q(w,\eta))\\
 &  & +O(\varepsilon^{3})+\varepsilon^{2}L(w,\eta)+\varepsilon Q(w,\eta))\\
 & = & \varepsilon^{-2}\phi^{-2}(\varepsilon^{3}L(w,\eta)+\varepsilon Q(w,\eta)+O(\varepsilon^{4})).
\end{eqnarray*}
So 
\begin{eqnarray*}
H & = & \frac{1}{k}\varepsilon^{-2}\phi^{-2}(2\varepsilon\phi^{2}-\varepsilon\frac{\dot{\psi}}{\phi}(\frac{\partial^{2}w}{\partial s^{2}}+\frac{\partial^{2}w}{\partial\theta^{2}})-2\varepsilon(\phi^{2}-\tau_{0})\frac{\dot{\phi}}{\phi^{2}}\frac{\partial w}{\partial s}-\varepsilon\frac{\dot{\psi}}{\phi}w\\
 &  & -\varepsilon^{2}\frac{\dot{\psi}^{3}}{\phi}<\frac{\partial^{2}\eta}{(\partial x_{0})^{2}},\Upsilon>-\varepsilon(\frac{\dot{\psi}\ddot{\psi}}{\phi}+2(\phi^{2}-\tau_{0})\frac{\dot{\phi}\dot{\psi}}{\phi^{2}})<\frac{\partial\eta}{\partial x_{0}},\Upsilon>\\
 &  & +\varepsilon^{3}(\phi\dot{\phi}\ddot{\psi}+2\dot{\phi}^{2}\dot{\psi}+\dot{\psi}^{3}-(\phi^{2}-\tau_{0})\dot{\psi}^{2})R(\Upsilon,X_{0},\Upsilon,X_{0})\\
 &  & +\frac{1}{3}\varepsilon^{3}\phi^{2}\dot{\psi}R(\Upsilon,\Upsilon_{\theta},\Upsilon,\Upsilon_{\theta})+\frac{2}{3}\varepsilon^{3}\dot{\phi}\phi^{2}R(\Upsilon_{\theta},X_{0},\Upsilon,\Upsilon_{\theta})\\
 &  & +\varepsilon^{2}(2\dot{\phi}\ddot{\psi}+2\frac{\dot{\phi}^{2}\dot{\psi}}{\phi}+\frac{\dot{\psi}^{3}}{\phi}-2\frac{\dot{\psi}^{2}}{\phi}(\phi^{2}-\tau_{0}))R(\Upsilon,X_{0},\eta,X_{0})\\
 &  & +\varepsilon^{2}(\frac{2}{3}\dot{\phi}\ddot{\phi}+\frac{2}{3}\frac{\dot{\phi}^{3}}{\phi}-\frac{2}{3}\dot{\psi}\ddot{\psi}-\frac{4}{3}(\phi^{2}-\tau_{0})\frac{\dot{\phi}\dot{\psi}}{\phi}-\frac{2}{3}\phi\dot{\phi})R(\Upsilon,X_{0},\eta,\Upsilon)\\
 &  & +\frac{2}{3}\varepsilon^{2}\phi\dot{\phi}R(\Upsilon_{\theta},X_{0},\eta,\Upsilon_{\theta})\\
 &  & +\varepsilon^{3}L(w,\eta)+\varepsilon Q(w,\eta)+O(\varepsilon^{4}))\\
\end{eqnarray*}
From (\ref{eq:k-expression}) we know 
\begin{eqnarray*}
\frac{1}{k} & = & 1-\frac{\varepsilon^{2}}{2}\dot{\phi}^{2}R(\Upsilon,X_{0},\Upsilon,X_{0})-\varepsilon\frac{\dot{\phi}^{2}}{\phi}R(\Upsilon,X_{0},\eta,X_{0})\\
 &  & +\frac{2}{3}\varepsilon\frac{\dot{\phi}\dot{\psi}}{\phi}R(\Upsilon,X_{0},\eta,\Upsilon)+\varepsilon^{2}L(w,\eta)+Q(w,\eta)+O(\varepsilon^{3}).
\end{eqnarray*}
So at last we get (\ref{eq:mean curvature expression}).

\section{The proof of Lemma \ref{average 1 lemma} and Lemma \ref{average 0 lemma}
\label{(APP)-Average1 and 0}}

\paragraph{The proof of Lemma \ref{average 1 lemma}. }

The goal is to prove that
\begin{eqnarray*}
 &  & \int_{a_{1}}^{b_{1}}\frac{1}{\phi}(2\dot{\phi}\ddot{\psi}+2\frac{\dot{\phi}^{2}\dot{\psi}}{\phi}+\frac{\dot{\psi}^{3}}{\phi}-2\frac{\dot{\psi}^{2}}{\phi}(\phi^{2}-\tau_{0})-2\phi\dot{\phi}^{2})d\psi\cdot\int_{a_{1}}^{b_{1}}\frac{\phi^{2}}{\dot{\psi}^{3}}d\psi\\
 & = & (b_{1}-a_{1})^{2},
\end{eqnarray*}
where $[a_{1},b_{1}]$ is one period for $\phi(\psi).$ 

\begin{proof}
\[
\int_{a_{1}}^{b_{1}}\frac{\phi^{2}}{\dot{\psi}^{3}}d\psi=\int_{a_{1}}^{b_{1}}\frac{1}{\phi}(1+\phi_{\psi}^{2})^{\frac{3}{2}}d\psi.
\]
From
\[
\phi_{\psi\psi}-\phi^{-1}(1+\phi_{\psi}^{2})+2(1+\phi_{\psi}^{2})^{\frac{3}{2}}=0
\]
we have
\[
\int_{a_{1}}^{b_{1}}\frac{\phi^{2}}{\dot{\psi}^{3}}d\psi=\int_{a_{1}}^{b_{1}}(\frac{1+\phi_{\psi}^{2}}{2\phi^{2}}-\frac{\phi_{\psi\psi}}{2\phi})d\psi.
\]
Note that
\begin{align*}
\int_{a_{1}}^{b_{1}}-\frac{\phi_{\psi\psi}}{2\phi}d\psi & =-\frac{1}{2}(\frac{\phi_{\psi}}{2\phi}|_{a_{1}}^{b_{1}}-\int_{a_{1}}^{b_{1}}\phi_{\psi}\frac{-\phi_{\psi}}{\phi^{2}}d\psi)\\
 & =-\frac{1}{2}\int_{a_{1}}^{b_{1}}\phi_{\psi}\frac{-\phi_{\psi}}{\phi^{2}}d\psi.
\end{align*}
So we have
\[
\int_{a_{1}}^{b_{1}}\frac{\phi^{2}}{\dot{\psi}^{3}}d\psi=\int_{a_{1}}^{b_{1}}\frac{1}{2\phi^{2}}d\psi.
\]
 Also from direct computation one can get
\begin{align*}
 & \int_{a_{1}}^{b_{1}}\frac{1}{\phi}(2\dot{\phi}\ddot{\psi}+2\frac{\dot{\phi}^{2}\dot{\psi}}{\phi}+\frac{\dot{\psi}^{3}}{\phi}-2\frac{\dot{\psi}^{2}}{\phi}(\phi^{2}-\tau_{0})-2\phi\dot{\phi}^{2})d\psi\\
= & \int_{a_{1}}^{b_{1}}(\phi^{2}(2\phi_{\psi\psi}(1+\phi_{\psi}^{2})^{-\frac{5}{2}}-(1+\phi_{\psi}^{2})^{-\frac{5}{2}}\phi_{\psi\psi}\phi_{\psi}^{2})\\
 & +\phi(3(1+\phi_{\psi}^{2})^{-\frac{3}{2}}\phi_{\psi}^{2}+(1+\phi_{\psi}^{2})^{-\frac{3}{2}}))d\psi\\
= & \int_{a_{1}}^{b_{1}}(2\phi^{2}+\frac{2\phi^{2}\phi_{\psi\psi}}{(1+\phi_{\psi}^{2})^{\frac{3}{2}}})d\psi.
\end{align*}
We assume that 
\[
\phi(s)=\sqrt{\tau_{0}}\exp(\sigma(s)).
\]
Then we have 
\begin{align*}
\dot{\psi} & =\sqrt{\tau_{0}}\exp(\sigma(s))\sqrt{1-\sigma_{s}^{2}},\\
\phi_{\psi} & =\frac{\sigma_{s}}{\sqrt{1-\sigma_{s}^{2}}},\\
\phi_{\psi\psi} & =\frac{\sigma_{ss}}{\sqrt{\tau_{0}}\exp(\sigma)(1-\sigma_{s}^{2})^{2}}.
\end{align*}
Also we have 
\begin{align*}
1-\sigma_{s}^{2} & =4\tau_{0}\cosh^{2}\sigma,\\
\sigma_{ss} & =-2\tau_{0}\sinh2\sigma.
\end{align*}
For $\phi$ there are two particular points in one period such that
$\sigma=0.$ Suppose these two points are 
\[
s=s_{1},s=s_{2}
\]
and suppose $[s_{1},s_{3}]$ is one period for $\phi(s)$. Note that
$s_{2}\in(s_{1},s_{3}).$

Here we use one particular property of Delaunay surface 
\[
\sigma(s_{2}-t)=-\sigma(s_{2}+t).
\]
Then from direct calculation we know
\begin{align*}
\int_{a_{1}}^{b_{1}}\frac{1}{2\phi^{2}}d\phi & =\int_{s_{1}}^{s_{3}}\cosh^{2}\sigma ds\\
\int_{a_{1}}^{b_{1}}(2\phi^{2}+\frac{2\phi^{2}\phi_{\psi\psi}}{(1+\phi_{\psi}^{2})^{\frac{3}{2}}})d\psi & =4\tau_{0}^{2}\int_{s_{1}}^{s_{3}}\cosh^{2}\sigma ds
\end{align*}
and 
\[
b_{1}-a_{1}=2\tau_{0}\int_{s_{1}}^{s_{3}}\cosh^{2}\sigma ds.
\]
So we know
\begin{eqnarray*}
 &  & \int_{a_{1}}^{b_{1}}\frac{1}{\phi}(2\dot{\phi}\ddot{\psi}+2\frac{\dot{\phi}^{2}\dot{\psi}}{\phi}+\frac{\dot{\psi}^{3}}{\phi}-2\frac{\dot{\psi}^{2}}{\phi}(\phi^{2}-\tau_{0})-2\phi\dot{\phi}^{2})d\psi\cdot\int_{a_{1}}^{b_{1}}\frac{\phi^{2}}{\dot{\psi}^{3}}d\psi\\
 & = & (b_{1}-a_{1})^{2}.
\end{eqnarray*}

\end{proof}

\paragraph{Proof of Lemma \ref{average 0 lemma}. }

Without loss of generality, we may assume $a_{\tau_{0}}=a_{1},b_{\tau_{0}}=b_{1}$.
The goal is to prove
\[
\int_{a_{1}}^{b_{1}}\phi(\frac{\partial\phi}{\partial\psi})\phi^{-2}(\frac{2}{3}\dot{\phi}\ddot{\phi}+\frac{2}{3}\frac{\dot{\phi}^{3}}{\phi}-\frac{2}{3}\dot{\psi}\ddot{\psi}-\frac{4}{3}(\phi^{2}-\tau_{0})\frac{\dot{\phi}\dot{\psi}}{\phi}+\frac{4}{3}\phi\dot{\phi}\dot{\psi})d\psi=0
\]
which is equivalent to 
\[
\int_{a_{1}}^{b_{1}}\phi_{\psi}^{2}\phi^{-2}\dot{\psi}(\phi\ddot{\phi}+\dot{\phi}^{2}-2(\phi^{2}+\tau_{0})(\phi^{2}-\tau_{0}))d\psi=0.
\]
\begin{proof}Note that
\begin{eqnarray*}
 &  & \int_{a_{1}}^{b_{1}}\phi_{\psi}^{2}\phi^{-2}\dot{\psi}(\phi\ddot{\phi}+\dot{\phi}^{2}-2(\phi^{2}+\tau_{0})(\phi^{2}-\tau_{0}))d\psi=0\\
 & = & \int_{a_{1}}^{b_{1}}\frac{\sigma_{s}^{2}}{1-\sigma_{s}^{2}}\frac{1}{\tau\exp(2\sigma)}\tau_{0}\exp(2\sigma)(1-\sigma_{s}^{2})(\phi\ddot{\phi}+\dot{\phi}^{2}-2(\phi^{2}+\tau_{0})(\phi^{2}-\tau_{0}))d\psi\\
 & = & \int_{a_{1}}^{b_{1}}\sigma_{s}^{2}(\phi\ddot{\phi}+\dot{\phi}^{2}-2(\phi^{2}+\tau_{0})(\phi^{2}-\tau_{0}))d\psi.
\end{eqnarray*}
\begin{eqnarray*}
 &  & \int_{a_{1}}^{b_{1}}\sigma_{s}^{2}(\phi\ddot{\phi}+\dot{\phi}^{2})d\psi\\
 & = & \int_{a_{1}}^{b_{1}}\sigma_{s}^{2}(\phi\dot{\phi})_{s}ds\\
 & = & -\int_{a_{1}}^{b_{1}}\phi\dot{\phi}2\sigma_{s}\sigma_{ss}ds\\
 & = & -2\int_{a_{1}}^{b_{1}}\exp(\sigma)\exp(\sigma)\sigma_{s}^{2}\sigma_{ss}ds\\
 & = & 4\tau_{0}^{2}\int_{a_{1}}^{b_{1}}\exp(2\sigma)(1-4\tau_{0}\cosh^{2}(\sigma))\sinh(2\sigma)ds\\
 & = & 2\tau_{0}^{2}\int_{a_{1}}^{b_{1}}(1-4\tau_{0}\cosh^{2}(\sigma))\sinh^{2}(2\sigma)ds.
\end{eqnarray*}
\begin{eqnarray*}
 &  & -2\int_{a_{1}}^{b_{1}}\sigma_{s}^{2}(\phi^{2}+\tau_{0})(\phi^{2}-\tau_{0})ds\\
 & = & -2\int_{a_{1}}^{b_{1}}\sigma_{s}^{2}(\tau_{0}\exp(2\sigma)+\tau_{0})(\tau_{0}\exp(2\sigma)-\tau_{0})ds\\
 & = & -2\tau_{0}^{2}\int_{a_{1}}^{b_{1}}\sigma_{s}^{2}4\sinh(\sigma)\cosh(\sigma)\exp(2\sigma)ds\\
 & = & -2\tau_{0}^{2}\int_{a_{1}}^{b_{1}}\sigma_{s}^{2}2\sinh(2\sigma)\exp(2\sigma)ds\\
 & = & -2\tau_{0}^{2}\int_{a_{1}}^{b_{1}}\sigma_{s}^{2}\sinh^{2}(2\sigma)ds\\
 & = & -2\tau_{0}^{2}\int_{a_{1}}^{b_{1}}(1-4\tau_{0}\cosh^{2}(\sigma))\sinh^{2}(2\sigma)ds.
\end{eqnarray*}
So we proved that
\[
\int_{a_{1}}^{b_{1}}\phi_{\psi}^{2}\phi^{-2}\dot{\psi}(\phi\ddot{\phi}+\dot{\phi}^{2}-2(\phi^{2}+\tau_{0})(\phi^{2}-\tau_{0}))d\psi=0.
\]

\end{proof}

\section{The proof of Lemma \ref{Phi-estimates}\label{(APP):Proof-of-Lemma Phi-estimate}}

First we prove 
\[
|\Phi'(\psi)-1|\leq C(\tau_{0})(\varepsilon+\|\xi\|_{C^{0}}+\varepsilon^{2}\|\mu\|_{C^{0}}+\varepsilon^{2}|\omega|).
\]
Note that 
\[
\Phi'(\psi)=\frac{d\Phi(\psi)}{ds_{0}}\frac{ds_{0}}{dl}\frac{dl}{d\psi},
\]
where 
\[
dl=\sqrt{<\partial_{s_{0}},\partial_{s_{0}}>ds_{0}^{2}+<\partial_{\tau},\partial_{\tau}>d\tau^{2}}
\]
is the arc length of $(\phi(\psi),\zeta(\psi))$. On the curve $(\phi(\psi),\zeta(\psi))$
we have 
\begin{equation}
1=\sqrt{<\partial_{s_{0}},\partial_{s_{0}}>(\frac{ds_{0}}{dl})^{2}+<\partial_{\tau},\partial_{\tau}>(\frac{d\tau}{d\psi})^{2}(\frac{d\psi}{dl})^{2}}.\label{Phi(1)}
\end{equation}
In $C(\tau_{0})(\varepsilon+\|\xi\|_{C^{0}}+\varepsilon^{2}\|\mu\|_{C^{0}}+\varepsilon^{2}|\omega|)$
neighborhood of $(\phi_{\tau(0)}(\psi),\zeta_{\tau(0)}(\psi))$, 
\begin{equation}
|<\partial_{s_{0}},\partial_{s_{0}}>-1|\leq C(\tau_{0})(\varepsilon+\|\xi\|_{C^{0}}+\varepsilon^{2}\|\mu\|_{C^{0}}+\varepsilon^{2}|\omega|).\label{Phi(2)}
\end{equation}
 And 
\begin{equation}
|\frac{d\tau}{d\psi}|\leq C\varepsilon(\varepsilon+\|\xi\|_{C^{0}}+\varepsilon^{2}\|\mu\|_{C^{0}}+\varepsilon^{2}|\omega|).\label{Phi(3)}
\end{equation}
Note that 
\begin{align*}
\frac{dl}{d\psi} & =\sqrt{(\frac{d\phi}{d\psi})^{2}+(\frac{d\zeta}{d\psi})^{2}}\\
 & =\sqrt{\zeta^{2}+(\phi^{-1}(1+\phi^{2})-(2+\rho)(1+\zeta^{2})^{\frac{3}{2}})^{2}}
\end{align*}
and
\begin{align*}
\frac{ds_{0}}{d\Phi(\psi)} & =\sqrt{(\frac{d\phi_{\tau(0)}}{d\Phi(\psi)})^{2}+(\frac{d\zeta_{\tau(0)}}{d\Phi(\psi)})^{2}}\\
 & =\sqrt{\zeta_{\tau(0)}^{2}+(\phi_{\tau(0)}^{-1}(1+\phi_{\tau(0)}^{2})-2(1+\zeta_{\tau(0)}^{2})^{\frac{3}{2}})^{2}}|_{\Phi(\psi)}.
\end{align*}
So we have
\begin{equation}
|\frac{dl}{d\psi}/\frac{ds_{0}}{d\Phi(\psi)}-1|\leq C(\tau_{0})(\varepsilon+\|\xi\|_{C^{0}}+\varepsilon^{2}\|\mu\|_{C^{0}}+\varepsilon^{2}|\omega|)\label{compare of the two}
\end{equation}
Note that $\frac{ds_{0}}{d\Phi(\psi)}$ has both positive upper bound
and positive lower bound which only depend on $\tau_{0}$ and so does
$\frac{dl}{d\psi}$ when $\varepsilon$ is sufficiently small.

So we know $(\frac{d\psi}{dl})^{2}$ is bounded and together with
(\ref{Phi(1)})(\ref{Phi(2)})(\ref{Phi(3)}) we get 
\begin{equation}
|\frac{ds_{0}}{dl}-1|\leq C(\tau_{0})(\varepsilon+\|\xi\|_{C^{0}}+\varepsilon^{2}\|\mu\|_{C^{0}}+\varepsilon^{2}|\omega|).\label{dl-estimate}
\end{equation}
From (\ref{compare of the two})(\ref{dl-estimate}), we get 
\begin{align*}
|\Phi'(\psi) & -1|=|\frac{\sqrt{\zeta^{2}+(\phi^{-1}(1+\zeta^{2})-(2+\rho)(1+\zeta^{2})^{\frac{3}{2}})^{2}}|_{\psi}}{\sqrt{\zeta_{\tau(0)}^{2}+(\phi_{\tau(0)}^{-1}(1+\phi_{\tau(0)}^{2})-2(1+\zeta_{\tau(0)}^{2})^{\frac{3}{2}})^{2}}|_{\Phi(\psi)}}\frac{ds_{0}}{dl}-1|\\
 & \leq C(\tau_{0})(\varepsilon+\|\xi\|_{C^{0}}+\varepsilon^{2}\|\mu\|_{C^{0}}+\varepsilon^{2}|\omega|).
\end{align*}
By integration we know that 
\[
|\Phi(\psi)-\psi|\leq\frac{C(\tau_{0})}{\varepsilon}(\varepsilon+\|\xi\|_{C^{0}}+\varepsilon^{2}\|\mu\|_{C^{0}}+\varepsilon^{2}|\omega|).
\]

\section{The proof of Lemma \ref{Aij estimate}\label{(APP)Aij estimate}}

\begin{align*}
 & \left(\begin{array}{cc}
\beta_{1}(\psi_{i}) & \beta_{2}(\psi_{i})\\
\frac{\partial\beta_{1}}{\partial\psi}(\psi_{i}) & \frac{\partial\beta_{2}}{\partial\psi}(\psi_{i})
\end{array}\right)\\
= & \left(\begin{array}{cc}
1+e_{11}^{i} & e_{12}^{i}\\
\kappa+e_{21}^{i} & 1+e_{22}^{i}
\end{array}\right)\cdots\left(\begin{array}{cc}
1+e_{11}^{1} & e_{12}^{1}\\
\kappa+e_{21}^{1} & 1+e_{22}^{1}
\end{array}\right)\\
= & \left(\begin{array}{cc}
A_{11}^{i} & A_{12}^{i}\\
A_{21}^{i} & A_{22}^{i}
\end{array}\right)
\end{align*}
where $|e_{kl}^{j}|\leq C(\tau_{0})(\varepsilon^{2}+\|\xi\|_{C^{0}}+\varepsilon^{2}\|\mu\|_{C^{0}}+\varepsilon^{2}|\omega|)$. 

\begin{lem}\label{shifting matrix lemma} For $\varepsilon$ sufficiently
small, if $|a_{ij}|\leq C(\tau_{0})(\varepsilon^{2}+\|\xi\|_{C^{0}}+\varepsilon^{2}\|\mu\|_{C^{0}}+\varepsilon^{2}|\omega|)$
and for $k\in\mathbb{Z},k\leq\frac{C(\tau_{0})}{\varepsilon}$, $|f|\leq kC(\tau_{0})(\varepsilon^{2}+\|\xi\|_{C^{0}}+\varepsilon^{2}\|\mu\|_{C^{0}}+\varepsilon^{2}|\omega|)$,
there are $\tilde{a}_{11},\tilde{a}_{21},\tilde{a}_{22},\tilde{f}$
such that the following holds
\begin{equation}
\left(\begin{array}{cc}
1+a_{11} & a_{12}\\
\kappa+a_{21} & 1+a_{22}
\end{array}\right)\left(\begin{array}{cc}
1 & f\\
0 & 1
\end{array}\right)=\left(\begin{array}{cc}
1 & \tilde{f}\\
0 & 1
\end{array}\right)\left(\begin{array}{cc}
1+\tilde{a}_{11} & 0\\
\kappa+\tilde{a}_{21} & 1+\tilde{a}_{22}
\end{array}\right),\label{shifting matrix}
\end{equation}
with 
\begin{align*}
|\tilde{f}| & \leq(k+2)C(\tau_{0})(\varepsilon^{2}+\|\xi\|_{C^{0}}+\varepsilon^{2}\|\mu\|_{C^{0}}+\varepsilon^{2}|\omega|),\\
|\tilde{a}_{11}|,|\tilde{a}_{22}| & \leq(\kappa k+1)C(\tau_{0})(\varepsilon^{2}+\|\xi\|_{C^{0}}+\varepsilon^{2}\|\mu\|_{C^{0}}+\varepsilon^{2}|\omega|).
\end{align*}

\end{lem}

\begin{proof} (of Lemma \ref{shifting matrix lemma}) Note that 
\begin{align*}
|(\kappa+a_{21})f| & \leq\frac{C(\tau_{0})}{\varepsilon}(\varepsilon^{2}+\|\xi\|_{C^{0}}+\varepsilon^{2}\|\mu\|_{C^{0}}+\varepsilon^{2}|\omega|)\\
 & \leq C(\tau_{0})(1+C_{1}+C_{2}+C_{3})\varepsilon.
\end{align*}
So we can choose $\varepsilon$ sufficiently small such that
\[
1+a_{22}+f(\kappa+a_{21})\geq1-C(\tau_{0},C_{1},C_{2},C_{3})\varepsilon
\]
and it is invertible. Then one can check directly that
\[
\begin{cases}
\tilde{a}_{11} & =a_{11}-(\kappa+a_{21})\frac{(1+a_{11})f+a_{12}}{1+a_{22}+f(\psi_{1}+a_{21})},\\
\tilde{a}_{21} & =a_{21,}\\
\tilde{a}_{22} & =(\kappa+a_{21})f+a_{22},\\
\tilde{f} & =\frac{(1+a_{11})f+a_{12}}{1+a_{22}+f(\kappa+a_{21})}
\end{cases}
\]
satisfies (\ref{shifting matrix}). So one can easily get 
\begin{align*}
|\tilde{f}| & \leq(k+2)C(\tau_{0})(\varepsilon^{2}+\|\xi\|_{C^{0}}+\varepsilon^{2}\|\mu\|_{C^{0}}+\varepsilon^{2}|\omega|),\\
|\tilde{a}_{11}|,|\tilde{a}_{22}| & \leq(\kappa k+1)C(\tau_{0})(\varepsilon^{2}+\|\xi\|_{C^{0}}+\varepsilon^{2}\|\mu\|_{C^{0}}+\varepsilon^{2}|\omega|).
\end{align*}
 \end{proof}

Note that
\begin{align*}
 & \left(\begin{array}{cc}
1+e_{11}^{i} & e_{12}^{i}\\
\kappa+e_{21}^{i} & 1+e_{22}^{i}
\end{array}\right)\cdots\left(\begin{array}{cc}
1+e_{11}^{1} & e_{12}^{1}\\
\kappa+e_{21}^{1} & 1+e_{22}^{1}
\end{array}\right)\\
= & \left(\begin{array}{cc}
1+e_{11}^{i} & e_{12}^{i}\\
\kappa+e_{21}^{i} & 1+e_{22}^{i}
\end{array}\right)\cdots\left(\begin{array}{cc}
1+e_{11}^{2} & e_{12}^{2}\\
\kappa+e_{21}^{2} & 1+e_{22}^{2}
\end{array}\right)\\
 & \left(\begin{array}{cc}
1 & -f_{12}^{1}\\
0 & 1
\end{array}\right)\left(\begin{array}{cc}
1+\tilde{e}_{11}^{1} & 0\\
\kappa+\tilde{e}_{21}^{1} & 1+\tilde{e}_{22}^{1}
\end{array}\right).
\end{align*}
Because $i$ is at most as big as $\frac{L_{\Gamma}}{\varepsilon\psi_{1}}$
and $|f_{12}^{1}|,|e_{ij}^{k}|\leq C(\tau_{0})(\varepsilon^{2}+\|\xi\|_{C^{0}}+\varepsilon^{2}\|\mu\|_{C^{0}}+\varepsilon^{2}|\omega|),$
we can use Lemma \ref{shifting matrix lemma} for $i$ times. Note
that by induction 
\begin{align*}
|f_{12}^{j}| & \leq2jC(\tau_{0})(\varepsilon^{2}+\|\xi\|_{C^{0}}+\varepsilon^{2}\|\mu\|_{C^{0}}+\varepsilon^{2}|\omega|)
\end{align*}
and $2j\leq\frac{2L_{\Gamma}}{\varepsilon\psi_{1}}\leq\frac{C(\tau_{0})}{\varepsilon}.$
At last we get
\begin{align*}
 & \left(\begin{array}{cc}
A_{11}^{i} & A_{12}^{i}\\
A_{21}^{i} & A_{22}^{i}
\end{array}\right)\\
= & \left(\begin{array}{cc}
1 & -f_{12}^{i}\\
0 & 1
\end{array}\right)\left(\begin{array}{cc}
1+\tilde{e}_{11}^{i} & 0\\
\kappa+\tilde{e}_{21}^{i} & 1+\tilde{e}_{22}^{i}
\end{array}\right)\cdots\left(\begin{array}{cc}
1+\tilde{e}_{11}^{1} & 0\\
\kappa+\tilde{e}_{21}^{1} & 1+\tilde{e}_{22}^{1}
\end{array}\right)
\end{align*}

where
\[
|f_{12}^{i}|,|\tilde{e}_{kl}^{j}|\leq\frac{C(\tau_{0})}{\varepsilon}(\varepsilon^{2}+\|\xi\|_{C^{0}}+\varepsilon^{2}\|\mu\|_{C^{0}}+\varepsilon^{2}|\omega|).
\]
If we replace $-f_{12}^{i}$ and $\tilde{e}_{kl}^{j}$ by $Er=\varepsilon C(\tau_{0})(1+C_{1}+C_{2}+C_{3})$
then all $A_{kl}^{i}$ will become bigger. Note that 
\begin{align*}
 & \left(\begin{array}{cc}
1+Er & 0\\
\kappa+Er & 1+Er
\end{array}\right)\cdots\left(\begin{array}{cc}
1+Er & 0\\
\kappa+Er & 1+Er
\end{array}\right)\\
= & \left(\begin{array}{cc}
(1+Er)^{i} & 0\\
i(\kappa+Er)(1+Er)^{i-1} & (1+Er)^{i}
\end{array}\right)
\end{align*}
Note that $i\leq\frac{C(\tau_{0})}{\varepsilon}$. So
\begin{align*}
(1+Er)^{i} & \leq\exp C(\tau_{0})(1+C_{1}+C_{2}+C_{3})\\
i(\kappa+Er)(1+Er)^{i-1} & \leq\frac{1}{\varepsilon}\exp C(\tau_{0})(1+C_{1}+C_{2}+C_{3})
\end{align*}
In the same way if we replace $-f_{12}^{i}$ all $\tilde{e}_{kl}^{j}$
by $-Er$ then all $A_{kl}^{i}$ will become smaller. Note that
\begin{align*}
 & \left(\begin{array}{cc}
1-Er & 0\\
\kappa-Er & 1-Er
\end{array}\right)\cdots\left(\begin{array}{cc}
1-Er & 0\\
\kappa-Er & 1-Er
\end{array}\right)\\
= & \left(\begin{array}{cc}
(1-Er)^{i} & 0\\
i(\kappa-Er)(1-Er)^{i-1} & (1-Er)^{i}
\end{array}\right)
\end{align*}
and
\begin{align*}
(1-Er)^{i} & \geq\exp(-C(\tau_{0})(1+C_{1}+C_{2}+C_{3}))\\
i(\kappa-Er)(1-Er)^{i-1} & \geq\frac{1}{\varepsilon}\exp(-C(\tau_{0})(1+C_{1}+C_{2}+C_{3}))
\end{align*}

So we know
\begin{eqnarray*}
\exp(-C(\tau_{0})(1+C_{1}+C_{2}+C_{3})) & \leq A_{22}^{i}\leq\exp C(\tau_{0})(1+C_{1}+C_{2}+C_{3}),\\
\exp(-C(\tau_{0})(1+C_{1}+C_{2}+C_{3}))\varepsilon^{-1} & \leq A_{21}^{i}\leq\exp C(\tau_{0})(1+C_{1}+C_{2}+C_{3})\varepsilon^{-1},\\
|A_{12}^{i}| & \leq C(\tau_{0},C_{1},C_{2},C_{3})\varepsilon\exp C(\tau_{0})(1+C_{1}+C_{2}+C_{3}).
\end{eqnarray*}
There is another way to calculate
\begin{align*}
 & \left(\begin{array}{cc}
A_{11}^{i} & A_{12}^{i}\\
A_{21}^{i} & A_{22}^{i}
\end{array}\right)\\
= & \left(\begin{array}{cc}
1+e_{11}^{i} & e_{12}^{i}\\
\kappa+e_{21}^{i} & 1+e_{22}^{i}
\end{array}\right)\cdots\left(\begin{array}{cc}
1+e_{11}^{1} & e_{12}^{1}\\
\kappa+e_{21}^{1} & 1+e_{22}^{1}
\end{array}\right)\\
= & \left(\begin{array}{cc}
1+\bar{e}_{11}^{i} & \bar{e}_{12}^{i}\\
\kappa+\bar{e}_{21}^{i} & 1+\bar{e}_{22}^{i}
\end{array}\right)\left(\begin{array}{cc}
1 & h_{12}^{1}\\
0 & 1
\end{array}\right)\cdots\left(\begin{array}{cc}
1+e_{11}^{1} & e_{12}^{1}\\
\kappa+e_{21}^{1} & 1+e_{22}^{1}
\end{array}\right)\\
= & \cdots\\
= & \left(\begin{array}{cc}
1+\bar{e}_{11}^{i} & \bar{e}_{12}^{i}\\
\kappa+\bar{e}_{21}^{i} & 1+\bar{e}_{22}^{i}
\end{array}\right)\cdots\left(\begin{array}{cc}
1+\bar{e}_{11}^{1} & \bar{e}_{12}^{1}\\
\kappa+\bar{e}_{21}^{1} & 1+\bar{e}_{22}^{1}
\end{array}\right)\left(\begin{array}{cc}
1 & h_{12}^{i}\\
0 & 1
\end{array}\right).
\end{align*}
And we can prove that
\[
\exp(-(C+C_{1}+C_{2}+C_{3}))\leq A_{11}^{i}\leq\exp(C+C_{1}+C_{2}+C_{3}).
\]

\section{The proof of Lemma \ref{Linearization estimates}\label{(APP)Proof-of-Lemma Linearization}}

\paragraph{1. $\beta_{\mu}$ estimate}

From
\[
\frac{d}{dt}\phi_{\xi,\mu+t\Delta\mu,\omega,\tau(0)}(\psi)|_{t=0}=\beta_{\mu}(\psi)
\]
we have 
\begin{align}
\begin{cases}
\mathcal{L}_{\xi,\mu,\omega,\tau(0)}\beta_{\mu}(\psi)=-\varepsilon^{3}(1+\phi_{\psi}^{2})^{\frac{3}{2}}\Delta\mu,\\
\beta_{\mu}(0)=0,\\
\beta_{\mu}^{\prime}(0)=0.
\end{cases}\label{beta-mu equation}
\end{align}
We have

\begin{align*}
\beta_{\mu}(\psi) & =-\varepsilon^{3}\int_{0}^{\psi}R(t)^{-1}(\beta_{2}(\psi)\beta_{1}(t)-\beta_{1}(\psi)\beta_{2}(t))(1+\phi_{\psi}^{2}(t))^{\frac{3}{2}}\Delta\mu(t)dt,\\
\beta'_{\mu}(\psi) & =-\varepsilon^{3}\int_{0}^{\psi}R(t)^{-1}(\beta_{2}'(\psi)\beta_{1}(t)-\beta_{1}'(\psi)\beta_{2}(t))(1+\phi_{\psi}^{2}(t))^{\frac{3}{2}}\Delta\mu(t)dt.
\end{align*}
So from (\ref{R(t) estimate}) 
\[
\|\beta_{\mu}(\psi)\|_{C_{\varepsilon}^{1}}\leq C(\tau_{0},C_{1},C_{2},C_{3})\varepsilon\|\Delta\mu\|_{C^{0}}.
\]

\paragraph{2. $\beta_{\xi}$ estimate}

From 
\[
\frac{d}{dt}\phi_{\xi+t\Delta\xi,\mu,\omega,\tau(0)}(\psi)|_{t=0}=\beta_{\xi}(\psi)
\]

\begin{equation}
\begin{cases}
\mathcal{L}_{\xi,\mu,\omega,\tau(0)}\beta_{\xi}(\psi)= & -\varepsilon F_{4}(\phi,\phi_{\psi})(1+\phi_{\psi}^{2})^{\frac{3}{2}}\Delta\xi\\
\beta_{\xi}(0)=0,\\
\beta_{\xi}^{\prime}(0)=0.
\end{cases}\label{eq:equation for beta-xi}
\end{equation}
It is easy to prove that 
\[
\|\beta_{\xi}(\psi)\|_{C_{\varepsilon}^{1}}\leq\frac{C(\tau_{0},C_{1},C_{2},C_{3})}{\varepsilon}\|\Delta\xi\|_{C^{0}}.
\]
 We are going to prove (\ref{beta xi estimate}).

Consider
\begin{eqnarray*}
 &  & |(\beta_{\xi},\beta_{\xi}^{\prime})\cdot(\tau_{\phi},\tau_{\zeta})|_{\bar{\psi}}|\\
 & = & |\frac{d}{dt}\tau|_{\bar{\psi}}|\\
 & = & |\frac{d}{dt}\int_{0}^{\bar{\psi}}\phi\phi_{\psi}\rho d\psi|\\
 & = & |\int_{0}^{\bar{\psi}}\varepsilon^{2}(\frac{\partial\hat{F}_{1}}{\partial\phi}\beta_{\xi}+\frac{\partial\hat{F}_{1}}{\partial\zeta}\beta_{\xi}^{\prime})+\varepsilon^{3}(\phi_{\psi}\beta_{\xi}+\phi\beta_{\xi}^{\prime})\mu\\
 &  & +\varepsilon(\frac{\partial\hat{F}_{2}}{\partial\phi}\beta_{\xi}+\frac{\partial\hat{F}_{2}}{\partial\zeta}\beta_{\xi}^{\prime})\xi+\varepsilon\hat{F}_{2}\Delta\xi+\varepsilon^{3}\omega(\frac{\partial\hat{F}_{3}}{\partial\phi}\beta_{\xi}+\frac{\partial\hat{F}_{3}}{\partial\zeta}\beta_{\xi}^{\prime})d\psi|\\
 & \leq & C(\tau_{0},C_{1},C_{2},C_{3})\|\Delta\xi\|_{C^{0}},
\end{eqnarray*}
where 
\begin{align*}
\hat{F}_{1}= & \phi\frac{\partial\phi}{\partial\psi}F_{1}(\phi,\phi_{\psi})\star R_{1},\\
\hat{F}_{2}= & \phi\frac{\partial\phi}{\partial\psi}F_{4}(\phi,\phi_{\psi}),\\
\hat{F}_{3}= & \phi_{\psi}^{2}.
\end{align*}

On the points $\psi_{i}$ where $\phi$ attains its local minimum,
we have $\tau_{\zeta}=0$ and $|\tau_{\phi}|$ has uniform positive
lower bound which only depends on $\tau_{0}$. So 
\begin{align*}
|\beta_{\xi}(\psi_{i})| & \leq C(\tau_{0},C_{1},C_{2},C_{3})\|\Delta\xi\|_{C^{0}}\\
|\beta_{\xi}^{\prime}(\psi_{i})| & \leq\frac{C(\tau_{0},C_{1},C_{2},C_{3})}{\varepsilon}\|\Delta\xi\|_{C^{0}}.
\end{align*}
So from (\ref{Beta piecewise estimate}), when $\psi\in[\psi_{i},\psi_{i+1}]$,
\begin{eqnarray}
 &  & \|\beta_{\xi}(\psi)-[\beta_{\xi}(\psi_{i})h(\tau(0))((\psi-\psi_{i})\frac{\partial\phi}{\partial\psi}+v_{i}(\psi))\nonumber \\
 &  & +\beta_{\xi}^{\prime}(\psi_{i})h(\tau(0))\frac{\partial\phi}{\partial\psi}]\|_{C_{\varepsilon}^{1}}\nonumber \\
 & \leq & C(\tau_{0},C_{1},C_{2},C_{3})\varepsilon\|\Delta\xi\|_{C^{0}}.\label{beta(xi) estimate by delta xi}
\end{eqnarray}
This tells us that the dominant part of $\beta_{\xi}(\psi)$ is $\beta_{\xi}^{\prime}(\psi_{i})h(\tau(0))\frac{\partial\phi}{\partial\psi}.$

Then we have, for $\bar{\psi}\in[\psi_{k},\psi_{k+1})$ 
\begin{align*}
 & |\int_{0}^{\bar{\psi}}\varepsilon^{2}(\frac{\partial\hat{F}_{1}}{\partial\phi}\beta_{\xi}+\frac{\partial\hat{F}_{1}}{\partial\zeta}\beta_{\xi}^{\prime})d\psi|\\
\leq & |\sum_{i=0}^{k}\int_{\psi_{i-1}}^{\psi_{i}}\varepsilon^{2}(\frac{\partial\hat{F}_{1}}{\partial\phi}\beta_{\xi}+\frac{\partial\hat{F}_{1}}{\partial\zeta}\beta_{\xi}^{\prime})d\psi\\
 & +\int_{\psi_{k}}^{\bar{\psi}}\varepsilon^{2}(\frac{\partial\hat{F}_{1}}{\partial\phi}\beta_{\xi}+\frac{\partial\hat{F}_{1}}{\partial\zeta}\beta_{\xi}^{\prime})d\psi|\\
\leq & \frac{C(\tau_{0},C_{1},C_{2},C_{3})}{\varepsilon}\|\Delta\xi\|_{C^{0}}|\sum_{i=0}^{k}\int_{\psi_{i-1}}^{\psi_{i}}\varepsilon^{2}(\frac{\partial\hat{F}_{1}}{\partial\phi}\frac{\partial\phi}{\partial\psi}+\frac{\partial\hat{F}_{1}}{\partial\zeta}\frac{\partial^{2}\phi}{\partial\psi^{2}})d\psi|\\
 & +C(\tau_{0},C_{1},C_{2},C_{3})\varepsilon\|\Delta\xi\|_{C^{0}}\\
\leq & C(\tau_{0},C_{1},C_{2},C_{3})\varepsilon\|\Delta\xi\|_{C^{0}}|\sum_{i=0}^{k}(\hat{F}_{1}(\psi_{i+1})-\hat{F}_{1}(\psi_{i}))|\\
 & +C(\tau_{0},C_{1},C_{2},C_{3})\varepsilon\|\Delta\xi\|_{C^{0}}\\
\leq & C(\tau_{0},C_{1},C_{2},C_{3})\varepsilon\|\Delta\xi\|_{C^{0}}.
\end{align*}
It is easy to see 
\begin{align*}
 & |\int_{0}^{\bar{\psi}}(\varepsilon^{3}(\beta_{\xi}\phi_{\psi}+\phi\beta_{\xi}^{\prime})\mu+\varepsilon(\frac{\partial\hat{F}_{2}}{\partial\phi}\beta_{\xi}+\frac{\partial\hat{F}_{2}}{\partial\zeta}\beta_{\xi}^{\prime})\xi\\
 & +\varepsilon^{3}\omega(\frac{\partial\hat{F}_{3}}{\partial\phi}\beta_{\xi}+\frac{\partial\hat{F}_{3}}{\partial\zeta}\beta_{\xi}^{\prime}))d\psi|\\
\leq & C(\tau_{0},C_{1},C_{2},C_{3})\varepsilon\|\Delta\xi\|_{C^{0}}.
\end{align*}

The last term is 
\begin{align*}
\int_{0}^{\bar{\psi}}\varepsilon\hat{F}_{2}\Delta\xi d\psi & =\varepsilon\int_{0}^{\bar{\psi}}\phi(\frac{\partial\phi}{\partial\psi})F_{4}(\phi,\phi_{\psi})\Delta\xi d\psi.
\end{align*}
From Lemma \ref{average 0 lemma}, using the argument of the proof
of Lemma \ref{better estimate in 1st mode}, also using Corollary
\ref{F(phi)-F(phi0)}, we can prove
\[
|\int_{0}^{\bar{\psi}}\varepsilon\hat{F}_{2}\Delta\xi d\psi|\leq C(\tau_{0},C_{1},C_{2},C_{3})\varepsilon\|\Delta\xi\|_{C_{x_{0}}^{1}}.
\]

So we have 
\begin{equation}
|\beta_{\xi}(\psi_{i})|\leq C(\tau_{0},C_{1},C_{2},C_{3})\varepsilon\|\Delta\xi\|_{C_{x_{0}}^{1}}.\label{beta-xi-psi}
\end{equation}
Once we prove that 
\[
|\beta_{\xi}'(\psi_{i})|\leq C(\tau_{0},C_{1},C_{2},C_{3})\|\Delta\xi\|_{C_{x_{0}}^{1}},
\]
we can deduce (\ref{beta xi estimate}) from (\ref{beta(xi) estimate by delta xi}).

Note that
\begin{eqnarray}
 &  & |\beta_{\xi}^{\prime}(\psi_{i+1})-(\beta_{\xi}(\psi_{i})\beta_{1,i}^{\prime}(\psi_{i+1})+\beta_{\xi}^{\prime}(\psi_{i})\beta_{2,i}^{\prime}(\psi_{i+1}))|\nonumber \\
 & = & |\int_{\psi_{i}}^{\psi_{i+1}}R_{i}^{-1}(\beta_{2,i}^{\prime}(\psi)\beta_{1,i}(t)-\beta_{1,i}^{\prime}(\psi)\beta_{2,i}(t))\mathcal{L}_{\xi,\mu,\omega,\tau(0)}\beta_{\xi}(t)dt|\nonumber \\
 & \leq & C(\tau_{0},C_{1},C_{2},C_{3})\varepsilon\|\Delta\xi\|_{C^{0}}\label{beta'xi}
\end{eqnarray}
 where
\[
R_{i}(\psi)=\left|\begin{array}{cc}
\beta_{1,i}(\psi) & \beta_{2,i}(\psi)\\
\beta_{1,i}'(\psi) & \beta_{2,i}'(\psi)
\end{array}\right|
\]
is the corresponding Wronskian. From Lemma \ref{estimates of beta(j,i-1)}
and (\ref{beta(i-1)(psi i)}) we can deduce

\begin{align*}
|\beta_{\xi}^{\prime}(\psi_{i+1})-\beta_{\xi}^{\prime}(\psi_{i})| & \leq C(\tau_{0},C_{1},C_{2},C_{3})\varepsilon(\|\Delta\xi\|_{C_{x_{0}}^{1}}+\varepsilon\beta'_{\xi}(\psi_{i})).
\end{align*}
From $\beta_{\xi}^{\prime}(0)=0$, by an induction argument, we get
\[
|\beta_{\xi}^{\prime}(\psi_{i})|\leq C(\tau_{0},C_{1},C_{2},C_{3})\|\Delta\xi\|_{C_{x_{0}}^{1}}.
\]
So we have 
\[
\|\beta_{\xi}(\psi)\|_{C_{\varepsilon}^{1}}\leq C(\tau_{0},C_{1},C_{2},C_{3})\|\Delta\xi\|_{C_{x_{0}}^{1}}.
\]

\paragraph{3. $\beta_{\omega}$ estimate}

From 
\[
\frac{d}{dt}\phi_{\xi,\mu,\omega+t,\tau(0)}(\psi)|_{t=0}=\beta_{\omega}(\psi),
\]
in the same way as we do for $\beta_{\mu}$ we can get (\ref{beta omega estimate}).

\paragraph{4. Proof of (\ref{estimates for Jacobi matrix})}

\begin{eqnarray*}
 &  & \frac{d}{dt}\tau(\bar{\psi})\\
 & = & \frac{d}{dt}\int_{0}^{\bar{\psi}}\phi\phi_{\psi}\rho d\psi\\
 & = & \int_{0}^{\bar{\psi}}\varepsilon^{2}(\frac{\partial\hat{F}_{1}}{\partial\phi}\beta_{\omega}+\frac{\partial\hat{F}_{1}}{\partial\phi_{\psi}}\beta_{\omega}^{\prime})+\varepsilon^{3}(\beta_{\omega}\phi_{\psi}+\phi\beta_{\omega}^{\prime})\mu\\
 &  & +\varepsilon(\frac{\partial\hat{F}_{2}}{\partial\phi}\beta_{\omega}+\frac{\partial\hat{F}_{2}}{\partial\phi_{\psi}}\beta_{\omega}^{\prime})R(\xi)+\varepsilon^{3}\omega(\frac{\partial\hat{F}_{3}}{\partial\phi}\beta_{\omega}+\frac{\partial\hat{F}_{3}}{\partial\phi_{\psi}}\beta_{\omega}^{\prime})+\varepsilon^{3}\phi_{\psi}^{2}d\psi
\end{eqnarray*}
In the similar way as we did for $\xi$, we can get 
\begin{eqnarray*}
 &  & |\int_{0}^{\bar{\psi}}\varepsilon^{2}(\frac{\partial\hat{F}_{1}}{\partial\phi}\beta_{\omega}+\frac{\partial\hat{F}_{1}}{\partial\phi_{\psi}}\beta_{\omega}^{\prime})+\varepsilon^{3}(\beta_{\omega}\phi_{\psi}+\phi\beta_{\omega}^{\prime})\mu\\
 &  & +\varepsilon(\frac{\partial\hat{F}_{2}}{\partial\phi}\beta_{\omega}+\frac{\partial\hat{F}_{2}}{\partial\phi_{\psi}}\beta_{\omega}^{\prime})R(\xi)+\varepsilon^{3}\omega(\frac{\partial\hat{F}_{3}}{\partial\phi}\beta_{\omega}+\frac{\partial\hat{F}_{3}}{\partial\phi_{\psi}}\beta_{\omega}^{\prime})|\\
 & \leq & C(\tau_{0},C_{1},C_{2},C_{3})\varepsilon^{3}
\end{eqnarray*}
The dominant term turns out to be 
\[
\int_{0}^{\bar{\psi}}\varepsilon^{3}\phi_{\psi}^{2}d\psi.
\]
 So for $(\xi,\mu,\omega,\phi(0))$ satisfying (\ref{C1C2C3C(tau0)}),
we can choose $\varepsilon$ sufficiently small, such that there is
a uniform constant $C_{5}=C_{5}(\tau_{0})>0$, which does not depend
on $C_{1},C_{2},C_{3},\varepsilon$ such that 
\begin{equation}
\frac{\partial}{\partial\omega}\tau(\frac{L_{\Gamma}}{\varepsilon})\geq C_{5}(\tau_{0})\varepsilon^{2}.\label{eq:C5}
\end{equation}

If we perturb $\phi(0)$ from $\frac{1-\sqrt{1-4\tau(0)}}{2}$ to
$\frac{1-\sqrt{1-4\tau(0)}}{2}+t$, the linearized function is just
$\beta_{1}(\psi).$ 
\begin{align*}
 & |\frac{\partial}{\partial\phi(0)}(\tau(\frac{L_{\Gamma}}{\varepsilon})-\tau(0))|\\
= & |\int_{0}^{\frac{L_{\Gamma}}{\varepsilon}}\varepsilon^{2}(\frac{\partial\hat{F}_{1}}{\partial\phi}\beta_{1}+\frac{\partial\hat{F}_{1}}{\partial\zeta}\beta_{1}^{\prime})+\varepsilon^{3}(\phi_{\psi}\beta_{1}+\phi\beta_{1}^{\prime})\mu\\
 & +\varepsilon(\frac{\partial\hat{F}_{2}}{\partial\phi}\beta_{1}+\frac{\partial\hat{F}_{2}}{\partial\zeta}\beta_{1}^{\prime})R(\xi)+\varepsilon^{3}\omega(\frac{\partial\hat{F}_{3}}{\partial\phi}\beta_{1}+\frac{\partial\hat{F}_{3}}{\partial\zeta}\beta_{1}^{\prime})d\psi|\\
 & \leq K_{1}(\tau_{0},C_{1},C_{2},C_{3})\varepsilon,
\end{align*}
where we deal with $\varepsilon^{2}(\frac{\partial\hat{F}_{1}}{\partial\phi}\beta_{1}+\frac{\partial\hat{F}_{1}}{\partial\phi_{\psi}}\beta_{1}^{\prime})$
in the same way as we did for $\beta_{\xi}$ (note that we have Lemma
\ref{Aij estimate}, \ref{estimate for the fundamental solution}). 

The estimates for $|\frac{\partial\zeta(\frac{L_{\Gamma}}{\varepsilon})}{\partial\omega}|,\frac{\partial\zeta(\frac{L_{\Gamma}}{\varepsilon})}{\partial\phi(0)}$
can be proved from Lemma \ref{Aij estimate}, \ref{estimate for the fundamental solution}
and (\ref{beta omega estimate}). At last we proved (\ref{estimates for Jacobi matrix}).

\subsection*{Acknowledgement }

The author shows his great respect and thanks to Professor Frank Pacard who gave him this problem as well as many deep insights, including how to deal with the high mode, 1st mode and how to do fixed point argument. The author owes him a lot.
 The author thanks Professor Gang Tian for long time help,
Professor Jie Qing for discussions on the ODE and Jinxing Xu for discussions
on Appendix \ref{(APP)Aij estimate}. The author was supported
by FMJH and partially supported by NSFC grant NO.11301284, NO.11571185, NO.11871283.

\bibliographystyle{plain}
\bibliography{Final-Version.bbl}

\end{document}